\documentclass{article}
\usepackage{authblk}
\usepackage[a4paper,
 total={140mm,237mm},
 left=35mm,
 top=30mm,
 ]{geometry}
\usepackage[utf8]{inputenc}
\usepackage{mathpazo}
\usepackage{hyperref}
\usepackage[anythingbreaks]{breakurl}
\usepackage{doi}
\usepackage{relsize}
\usepackage{mathtools}

\usepackage[T1]{fontenc}
\usepackage{amssymb, amsmath}
\usepackage[dvipsnames]{xcolor}
\usepackage[english]{babel}
\usepackage{amsthm}
\usepackage{enumitem}
\usepackage{tikz}
\usetikzlibrary{arrows}
\usetikzlibrary{calc}
\usepackage{bm}
\usepackage{csquotes}
\usepackage[shortcuts]{extdash}
\usepackage{xcolor}
\usepackage{authblk}
\usepackage{tabularx}
\usepackage[shortcuts]{extdash}

\usepackage[bb=dsserif]{mathalpha}

\usepackage{graphicx} %

\usepackage{url}
\usepackage{breakurl}
\usepackage{xfrac}
 
\newcommand{\introthmname}{}
\newtheorem{introthminn}{\introthmname}

\usepackage{thmtools}
\usepackage{xfrac}
\usepackage{thm-restate}
\usepackage[capitalise]{cleveref}
\usepackage{verbatim}
\usepackage{mathabx}
\usepackage{graphicx}
\usepackage{amsfonts}
\usepackage{pb-diagram}
\usepackage[normalem]{ulem}

\usepackage{adjustbox}
\usepackage{mathtools}
\usepackage{todonotes}

\usepackage{mathrsfs}
\usepackage{bm}
\usepackage{enumitem}

\usepackage[format=plain,
            labelfont=it,
            textfont=it]{caption}
\usepgflibrary{shapes.geometric}
\usepackage{float}
\allowdisplaybreaks

 \hypersetup{breaklinks=true, %
           colorlinks=false,   %
           hidelinks
        }

\usepackage[maxnames=20, url=false,  doi=false]{biblatex} %

\AtEveryBibitem{\clearlist{language}}
\renewbibmacro{in:}{}

\bibliography{global.bib}

\usepackage{subfiles}

\newtheorem{theorem}{Theorem}[section]

\newtheorem{corollary}[theorem]{Corollary}

\newcommand{\ignore}[1]{}
\newtheorem{lemma}[theorem]{Lemma}
\newtheorem{prop}[theorem]{Proposition}
\newtheorem{fact}[theorem]{Fact}
\theoremstyle{remark}
\newtheorem{remark}[theorem]{Remark}
\newtheorem{notation}[theorem]{Notation}

\theoremstyle{definition}
\newtheorem{example}[theorem]{Example}

\newtheorem{examples}[theorem]{Examples}
\newtheorem{definition}[theorem]{Definition}
\newtheorem{question}[theorem]{Question}
\newtheorem{observation}[theorem]{Observation}

  \newtheorem{thmx}{Theorem}

\title{Minimal and intrinsic topologies\\ on 
monoids of elementary embeddings}

\author[1]{J. de la Nuez Gonz\'{a}lez}
\author[2]{Zaniar Ghadernezhad}
\author[3]{Paolo Marimon}
\author[4]{Michael Pinsker}
\affil[1]{School of Mathematics, Korea Institute for Advanced Study (KIAS). \newline Seoul, Republic of Korea.}
\affil[2]{School of Physical and Chemical Sciences, Queen Mary University of London. London, United Kingdom.}
\affil[3]{Department of Computer Science, University of Oxford. \newline Oxford, United Kingdom.}
\affil[4]{Institut f\"{u}r Diskrete Mathematik \& Geometrie, TU Wien. Vienna, Austria.}

\def\acts{\curvearrowright}

\DeclareMathOperator{\acl}{acl}
\DeclareMathOperator{\dcl}{dcl}

\DeclareMathOperator{\Aut}{Aut}
\DeclareMathOperator{\End}{End}
\DeclareMathOperator{\EEmb}{EEmb}

\newcommand{\sA}{\mathbb{A}}

\newcommand{\actson}{\curvearrowright}

  \def\G{\overline{G}}
\def\a{\overline{a}}
\def\b{\overline{b}}
\def\f{\overline{f}}

\newcommand{\M}[0]{\mathcal{M}}

\newcommand{\U}[0]{\mathbb{U}}

\newcommand{\nn}[0]{S}%
\newcommand{\N}[0]{\mathbb{N}}

\renewcommand{\U}[0]{\mathcal{U}}
\newcommand{\V}[0]{\mathcal{V}}
\newcommand{\W}[0]{\mathcal{W}}

\DeclareMathOperator{\im}{Im}

\newcommand{\pw}[0]{\tau_{\mathrm{pw}}}
\newcommand{\zar}[0]{\tau_{\mathrm{Z}}}
\newcommand{\mpw}[0]{\tau_{\mathrm{mp}}}

 \newcommand{\C}[0]{\mathcal{C}}
  \newcommand{\D}[0]{\mathcal{D}}
 \newcommand{\E}[0]{\mathcal{E}}

\newcommand{\X}[0]{\mathfrak{X}}
\newcommand{\id}[0]{\ind}

\def\Ind#1#2{#1\setbox0=\hbox{$#1x$}\kern\wd0\hbox to 0pt{\hss$#1\mid$\hss}
	\lower.9\ht0\hbox to 0pt{\hss$#1\smile$\hss}\kern\wd0}
\def\ind{\mathop{\mathpalette\Ind{}}}
\def\Notind#1#2{#1\setbox0=\hbox{$#1x$}\kern\wd0\hbox to 0pt{\mathchardef
		\nn="3236\hss$#1\nn$\kern1.4\wd0\hss}\hbox to
	0pt{\hss$#1\mid$\hss}\lower.9\ht0
	\hbox to 0pt{\hss$#1\smile$\hss}\kern\wd0}

\makeatletter
\newcommand{\tpitchfork}{%
  \vbox{
    \baselineskip\z@skip
    \lineskip-.52ex
    \lineskiplimit\maxdimen
    \m@th
    \ialign{##\crcr\hidewidth\smash{$-$}\hidewidth\crcr$\pitchfork$\crcr}
  }%
}

\newenvironment{subproof}[1][\proofname]{%
\begin{proof}[#1]%
	}{%
\end{proof}%
}

\newcommand{\fps}[1]{[#1]^{<\omega}} 

\newcommand{\orq}[0]{\equiv}
\global\long\def\cl{\mathrm{acl}}

\renewcommand\leq{\leqslant}
\renewcommand\geq{\geqslant}

\newcommand{\rr}[0]{\mathcal{R}}
\newcommand{\ie}[0]{\mathrm{IE}}

\DeclareMathOperator{\Sym}{Sym}

\newcommand{\gc}[0]{G^{\sim}}

\begin{document}

\maketitle
\begin{abstract}
To every $\omega$-categorical structure $\M$ one can associate two spaces of symmetries which determine the structure up to first-order bi-inter\-pretability: the topological group $\Aut(M)$ of its automorphisms and the topological monoid $\EEmb(M)$ of its elementary embeddings, both equipped with the topology of pointwise convergence $\pw$. We investigate the relation of $\pw$ to other topologies on these spaces: in particular, when $\pw$ is minimal, i.e.~does not admit any strictly coarser Hausdorff semigroup topology.

A common method to prove minimality of $\pw$ on $\EEmb(M)$ is to show that it coincides with the algebraically defined semigroup Zariski topology $\zar$. We show that $\pw$ differs from $\zar$ on $\EEmb(M)$ whenever $\Aut(M)$ has a non-trivial centre. In spite of this, we then prove that whenever algebraic closure on $M$ is modular, %
$\pw$ is minimal on $\EEmb(M)$. This covers,
for example, countable vector spaces and projective spaces over finite fields. {Turning to $\Aut(M)$, we describe the semigroup topologies coarser than $\pw$ on the automorphism groups of structures for which algebraic %
{independence} satisfies independent 3-amalgamation.}
 We conclude by proving that for  the real and the rational Urysohn space and sphere, the metric pointwise topology $\mpw$ %
is minimal on $\EEmb(M)$,
equals $\zar$, and is strictly coarser than $\pw$.

\end{abstract}

\textbf{MSC classes:} 20M20, 54H15, 22A15, 03C45, 03C15.

\tableofcontents

\section{Introduction}
\subsection{Spaces of symmetries and their topologies}

Given a countable structure $\M$, there are two natural spaces of symmetries preserving its structure as given by first-order logic: its automorphism group $\Aut(\M)$ 
{(consisting of the symmetries that are bijections), }%
and its monoid of elementary embeddings $\EEmb(\M)$.
From these objects, considered respectively as a permutation group and a transformation monoid, one can draw conclusions about the original structure $\M$. For example, we can tell whether or not $\M$ is \textbf{$\omega$-categorical}, i.e.~up to isomorphism the only countable model of its first-order theory; by the theorem of Ryll-Nardzewski, Engeler, and Svenonius (see the textbook~\cite{tent2012course}), this is the case if and only if $\Aut(\M)$ is \textbf{oligomorphic}, that is, it acts on all finite powers of its domain %
with finitely many orbits. 
 In the $\omega$-categorical case, we can even measure exactly ``how well'' $\Aut(\M)$ and $\EEmb(\M)$ still know $\M$: they determine $\M$ up to first-order interdefinability, i.e. for any other structure, its automorphism group (respectively monoid of elementary embeddings) acts the same way if and only if that structure has a first-order definition without parameters in $\M$, and vice-versa.\\

One can then abstract further: $\Aut(\M)$ and $\EEmb(\M)$ are, respectively, a topological group and a topological monoid when equipped with the \textbf{topology $\pw$ of pointwise convergence}, for which a subbasis of open sets is given by sets of all maps which send a particular point $x\in \Omega$ to another specific point $y\in \Omega$ (here and throughout, $\Omega$ denotes the domain of the structure $\M$).  Moreover, $\Aut(\M)$ is $\pw$-closed in the group of all permutations of $\Omega$, and $\EEmb(\M)$ is $\pw$-closed in the full monoid $\Omega^\Omega$. This implies that $\pw$ is \textbf{Polish}  on these spaces, i.e., separable and induced by a complete metric.
From these topological-algebraic structures, we can still extract considerable information about $\M$: in the $\omega$-categorical case both the topological group $\Aut(\M)$ and the topological monoid $\EEmb(\M)$ determine $\M$ up to first-order %
 {bi-interpretability,} a more generous equivalence than the interdefinability mentioned above
(Coquand; see~\cite{AhlbrandtZiegler}, and~\cite{Lascar}). %
 To provide another example, %
 from $\Aut(\M)$ we can tell whether or not $\M$ has a certain combinatorial property called the Ramsey property: this is the case if and only if the topological group $\Aut(\M)$ is extremely amenable, i.e.~every continuous action of it on a compact Hausdorff space has a fixed point~\cite{KPT}.\\

This invites a further step of abstraction where we view these spaces of symmetries only as algebraic structures by forgetting the topology, giving rise to two natural questions:
\begin{question}\label{q1} How much information about $\M$ is still encoded in the pure group $\Aut(\M)$ or the pure monoid $\EEmb(\M)$? 
\end{question}
\begin{question}\label{q2} What is the relation of $\pw$ to the other (Hausdorff/Polish) topologies compatible with the algebraic structure of $\Aut(\M)$ or $\EEmb(\M)$?
\end{question}

\subsection{The case of groups}

Over the past four decades, a rich literature has emerged on these questions in the context of automorphism groups.
 Regarding \cref{q1}, in the $\omega$-categorical case, {we often have that the pointwise convergence topology is the same in all possible actions of $\Aut(\M)$ as a closed oligomorphic permutation group; this then implies that the algebraic structure of $\Aut(\M)$ encodes $\M$ up to bi-interpretability.}
As exemplary references, we point to~\cite{DixonNeumannThomas, Truss} for the notion of the small index property; to~\cite{Rubin} for %
{$\forall\exists$}-interpretations; and to~\cite{HodgesHodkinsonLascarShelah, KechrisRosendal} for ample generics. {In fact, when $\Aut(\M)$ has ample generics, then $\pw$ is even the unique  Polish group topology on it (which is a stronger statement than the pointwise topology being the same in all actions); this then provides a complete answer to~\cref{q2} in the Polish case.} {The surveys~\cite{MacphersonSurvey, rosendal2025aspectsautomaticcontinuity} contain many further references to notable works.}  Counterexamples {to uniqueness of $\pw$ amongst Polish group topologies, or even to $\M$ being determined up to bi-interpretability by the abstract group $\Aut(\M)$}, such as that given in~\cite{EvansHewitt}, are hard to construct; indeed, it is consistent with $\mathrm{ZF}$ that $\pw$ is always the unique Polish group topology on $\Aut(\M)$~\cite{Lascar} (cf.~\cite{solovay1970model, shelah1984can}). {Not only the topology $\pw$, but} even the action $\Aut(\M)\actson \Omega$ can be reconstructed from the pure group structure in many cases under the additional assumption that $\M$ has \textbf{no algebraicity}, i.e., for each finite $A\subseteq \Omega$ the pointwise stabiliser of $A$ acts only with infinite orbits on $\Omega\setminus A$~\cite{Rubin}.\\

Looking at \cref{q2}, in many interesting special circumstances $\pw$ is such that there are no strictly coarser Hausdorff group topologies on $\Aut(\M)$. Topological groups with this property are known as \textbf{minimal}, and such groups have been  intensively studied~\cite{dikranjan2013minimality}. 
Examples of minimal topological groups arising from automorphism groups of countable structures are  the full symmetric group on a countable set 
 by~\cite{gaughan1967topological}, and more generally the automorphism groups of $\omega$-categorical stable structures~\cite{ben2016weakly}. The latter is a rich class of model-theoretically ``tame’’ structures, including, for example, infinite-dimensional vector spaces and projective spaces over finite fields. Indeed, the automorphism groups of stable $\omega$-categorical structures are not only minimal but \textbf{totally minimal}:  every continuous surjective homomorphism  to a Hausdorff topological group is open.
In their survey of minimal topological groups, Dikranjan and Megrelishvili explicitly ask which oligomorphic permutation groups are minimal
~\cite[Question 2.3]{dikranjan2013minimality}.\\

On the positive side, from~\cite{ghaddln24} we have some additional results outside the stable context under the model-theoretic weakening of stability known as simplicity: the automorphism groups of simple transitive $\omega$-categorical one-based structures with weak elimination of imaginaries are minimal~\cite{ghaddln24}. This includes, for example, the automorphism group of any  vector space of infinite dimension over a finite field %
with a {symplectic or unitary 
form}~\cite{cherlin2003finite}. On the negative side, we can identify three 
recurring 
sources of non-minimality: non-transitivity, the presence of intervals of some form, and the presence of a metric structure. The first source can lead to failures of minimality by ``localising'' the topology of pointwise convergence to an orbit of $\Aut(\M)$ (\cref{ex:bipartite}). 
From the second source, an interval topology $\tau_{\mathrm{int}}$ can be defined, for example on $\Aut(\mathbb{Q}; <)$, with the basic open sets consisting of maps sending a particular element into a specific interval (rather than to another specific  element); in this example, $\tau_{\mathrm{int}}$  is minimal and strictly coarser than $\pw$~\cite{chang2017minimum, ghaddln24}. The third source occurs outside of the $\omega$-categorical setting, and allows one to define a ``metric pointwise topology’’ $\mpw$ coarser than $\pw$; an example is the real Urysohn sphere~\cite{uspenskij1990group}.

\subsection{The case of monoids}

With the notable exception of some early works by Lascar~\cite{Lascar, Lascar82}, the study of the behaviour of $\pw$ on the monoid of elementary embeddings $\EEmb(\M)$ of a structure $\M$  (or other more general spaces of symmetries) is much younger: it was only initiated about 15 years ago~\cite{Reconstruction}. Since then, a considerable number of results have emerged on the topic, as recently surveyed in~\cite{marimon2025guide}. The motivation for this line of research is threefold. Firstly,  studying monoids of elementary embeddings sheds light also on automorphism groups: this perspective can already be found in Lascar's work~\cite{Lascar, Lascar82}. Secondly, such study can help us develop tools and examples for the general investigation of topological semigroups, as demonstrated in the seminal paper~\cite{elliott2023automatic}; compared to groups, substantially different methods to prove, say, minimality are required, and results differ: for example, every monoid of elementary embeddings of an $\omega$-categorical structure has a Polish topology strictly finer than $\pw$~\cite[Proposition~5.1]{elliott2023polish}. Finally, the study of $\EEmb(\M)$ often serves as an initial step to understanding even richer spaces of symmetries~\cite{bodirsky2015topological} such as endomorphism monoids (maps preserving existential positive  formulas) and polymorphism clones (higher-arity maps preserving primitive positive formulas), with the latter having important applications 
to constraint satisfaction problems in
theoretical computer science~\cite{BPP-projective-homomorphisms, RomanMichael}.\\

Moving from $\Aut(\M)$ to $\EEmb(\M)$, there are surprising differences in topological behaviour that we explore in this paper. At a first glance, $\EEmb(\M)$ may look extremely close to $\Aut(\M)$ in the contexts we study. In fact, for any \textbf{saturated} structure $\M$ (which includes $\omega$-categorical ones), it coincides with the $\pw$-closure of $\Aut(\M)$ in $\Omega^\Omega$. Yet, there is a fascinating tension between the weaker algebraic structure of a monoid (which inhibits the use of many  techniques coming from groups) and the fact that we have a  richer space of maps (making it potentially easier to see the first-order behaviour of $\M$ from the pure  algebraic structure of the monoid).\\

Arguably, as mentioned in~\cite{marimon2025guide}, $\EEmb(\M)$ often behaves better than $\Aut(\M)$: already Lascar gives an extremely satisfying answer to \cref{q1} for $\EEmb(\M)$ by proving that any isomorphism between the monoids of  elementary embeddings of two countable saturated $G$-finite structures is continuous when restricted to their automorphism groups. In fact, it is even a topological isomorphism between the monoids as shown recently  in~\cite{marimon2025guide}. No result at this level of generality is known for isomorphisms between  automorphism groups. %
As another example, by~\cite[Theorem~3.2]{pinskerschindler},  if two $\omega$-categorical structures with no algebraicity have isomorphic monoids
of elementary self-embeddings, then this isomorphism is automatically a topological one; in fact, it follows from results in~\cite{Rubin} (see~\cite[Theorem~58]{marimon2025guide}) that any such 
isomorphism is even induced by %
a bijection between the domains of the structures yielding an isomorphism of the actions of the associated transformation monoids. %
No such result is known for groups, with various special cases proved in work following~\cite{Rubin}.\\

Also with regard to minimality,  $\EEmb(\M)$ seems to behave better than $\Aut(\M)$.
 From~\cite{pinskerschindler}, we know that $\pw$ is minimal on the monoid of elementary embeddings of any $\omega$-categorical structure with no algebraicity. This  already shows that standard sources of failure of minimality on $\Aut(\M)$ do not seem to work for $\EEmb(\M)$, since both $(\mathbb{Q}; <)$ (where one can define an interval topology on $\Aut(\M)$)  and the random bipartite graph (where one can define a localised topology) have no algebraicity.

\subsection{The Zariski topology}\label{subsect:Zariski}

Minimality of $\pw$ on  $\EEmb(\M)$ can often be shown by proving that it  has a purely algebraic description in terms of the \textbf{semigroup Zariski topology} $\zar$ from~\cite{elliott2023automatic}; the latter is known to be contained in every other Hausdorff semigroup topology {(that is, every Hausdorff topology making multiplication continuous)}. In $\zar$, the closed sets are generated by  solution sets of equations over the semigroup in one free variable with parameters from the semigroup. Since it is defined in purely algebraic terms, any algebraic isomorphism between monoids is a homeomorphism of their Zariski topologies. 
A priori,  $\zar$ is not particularly well-behaved or expressive: it need not be Hausdorff or a semigroup topology. However, it so happens that it often coincides with $\pw$ on   $\EEmb(\M)$, as first demonstrated in~\cite{pinskerschindler} for all $\omega$-categorical structures with no algebraicity. This was already observed in~\cite{elliott2023polish}, where it was then asked whether there exists an $\omega$-categorical structure  at all where $\pw\neq\zar$ on its endomorphism monoid (monoids of the form $\EEmb(\M)$ are a special case of these; they are precisely those endomorphism monoids with a dense set of invertibles). The question can be answered affirmatively even for $\EEmb(\M)$ %
 by the results of~\cite{BodirskyEvansKompatscherPinsker} (cf.~\cite{marimon2025guide}), and in~\cite{pinskerschindler} another more natural and tangible example was provided; yet these two examples remain the only known ones to this day.
 In fact, showing that $\pw=\zar$ is the standard path to prove minimality of $\pw$ on endomorphism monoids. This approach was taken, in particular, in~\cite{elliott2023automatic,elliott2023polish,pinskerschindlerQ}, where %
 other methods were then used to show that $\pw$ is also maximal, and hence unique, among Polish semigroup topologies on various natural endomorphism monoids.
\\ 

To compare this situation  with the case of groups,  $\zar$ is hyperconnected on $\Aut(\M)$ (and hence not Hausdorff) whenever $\M$ has no algebraicity~\cite{bardyla2025note}. {Moreover, }
the more intensively studied \textbf{group Zariski  topology}~\cite{bryant1977verbal}, where inversion may be used in the definition of closed sets, is often not a group topology on $\Aut(\M)$~\cite{ghaddln24}; with notable exceptions being the automorphism groups of $(\mathbb{N}; =)$ and $(\mathbb{Q}; <)$.

\section{Results}

We now present our main results; all relevant definitions can be found in Section~\ref{sect:prelims}. We formulate our results in terms of {a permutation group}
$G$ acting on a countable set $\Omega$. To connect this with the previous discussion, we equip $\Omega^\Omega$ with the topology of pointwise convergence $\pw$. It is easy to see that $G\actson \Omega$ is closed in the set of all permutations with respect to $\pw$ if and only if it equals $\Aut(\M)$ for some first-order structure $\M$ on $\Omega$. As mentioned above, the closure $\G$ of $G$ in $\Omega^\Omega$ is precisely the monoid $\EEmb(\M)$ of elementary self-embeddings of $\M$ whenever $\M$ is saturated; in particular, this is the case whenever $\M$ is $\omega$-categorical. All theorems will feature assumptions on the operator $\acl$ of \textbf{algebraic closure} (with respect to $G\acts\Omega$): for any finite $A\subseteq \Omega$,  %
we let $\acl(A)$  consist of all elements with a finite orbit under the action of the pointwise stabiliser  $G_A$. We will always assume, in particular,   
that  algebraic closure is \textbf{locally finite}, i.e., $\acl(A)$ is finite for all finite $A\subseteq \Omega$.  This is satisfied by the automorphism groups of 
$\omega$-categorical structures. An extreme form is the aforementioned notion of no algebraicity which means $\acl(A)=A$ for all finite $A\subseteq \Omega$.\\

Our first theorem provides, for the first time, general conditions implying $\pw\neq\zar$ on monoids of elementary embeddings; hitherto  only two examples were known,  cf.~\cref{subsect:Zariski}.

\begin{thmx}[{\cref{thm:zar}}]\label{mainthm: a} Let $G\acts \Omega$ be a closed permutation group with locally finite algebraic closure, and suppose that $G$ has a non-trivial centre. %
Then the semigroup Zariski topology $\zar$ is neither Hausdorff on $G$, nor on $\G$. In particular, it is properly contained in $\pw$. 
\end{thmx}

Natural examples of  structures whose  automorphism groups satisfy the above assumptions are countable vector spaces over finite fields of size $>2$, %
or infinitely many disjoint copies of an edge $K_2$ (which is essentially the source of the original counterexample of~\cite{pinskerschindler}).\\

\cref{mainthm: a} tells us that the standard method of showing $\pw=\zar$ in order to prove minimality of $\pw$ on $\EEmb(\M)$  cannot work for $\omega$-categorical structures with a non-trivial 
central automorphism. 
Such an automorphism can be viewed as a strong
form of algebraicity; recall from Section~\ref{subsect:Zariski} that no algebraicity implies $\pw=\zar$. We thus develop  alternative techniques to establish minimality as long as the lattice of algebraically closed sets is modular:

\begin{thmx}[\cref{t:minimality endomorphism monoids}]\label{mainthm: b}Let $G\acts\Omega$ be the automorphism group of a countable saturated structure such that {algebraic closure is  locally finite  and modular.}
Then, $\pw$ is minimal amongst Hausdorff semigroup topologies on $\G$.  
\end{thmx}
{\cref{mainthm: b} subsumes the results of~\cite{pinskerschindler} which were obtained under the stronger assumptions of $\omega$-categoricity and  no algebraicity. In particular, it yields minimality of $\pw$ for the {monoid of elementary embeddings}
of any infinite-dimensional vector space over a finite field, {even though, over fields of size $>2$, we know that $\pw\neq\zar$ from  Theorem~\ref{mainthm: a}.}
{More generally,} the conclusion of \cref{mainthm: b} holds for all (model-theoretically) simple one-based {$\omega$-categorical} structures %
with  %
weak elimination of imaginaries.
Simple one-based {$\omega$-categorical} structures 
include
all smoothly approximable structures~\cite{cherlin2003finite}, which contain all $\omega$-categorical $\omega$-stable structures~\cite{cherlin1985}, and all simple finitely homogeneous structures~\cite{baldwin2024simplehomogeneousstructuresindiscernible}. We discuss further applications %
also beyond this context in Section~\ref{sec:discussion}.} \\ %

{Moving %
to the context of groups, we improve the main result of~\cite{ghaddln24} (Theorem B therein) by describing {all} %
semigroup topologies on automorphism groups %
which are coarser than $\pw$ under more general conditions:} 

\begin{thmx}[\cref{thm:onebased}] \label{mainthm: c} 
Let $G\acts\Omega$ be the automorphism group of a countable saturated structure such that algebraic closure is  locally finite
and { algebraic independence} satisfies independent $3$-amalgamation.
Then, every semigroup topology on $G$ coarser than $\pw$ is a generalised pointwise topology. If $G$ is transitive, then $\pw$ is minimal amongst {$T_0$} semigroup topologies on $G$; if $G$ is moreover primitive, then {the only semigroup topology on $G$ strictly coarser than $\pw$ is the indiscrete topology.}
\end{thmx}
We describe generalised pointwise topologies in \cref{d:generalized pointwise stabilizer}. {\cref{mainthm: c} improves on~\cite[Theorem B]{ghaddln24} in several ways: first, we classify semigroup topologies on $G$ coarser than $\pw$ (rather than group topologies); second, in the non-transitive case, we remove the assumption that points are algebraically closed; third, in the 
transitive case we show minimality of $\pw$ amongst the {$T_0$} semigroup topologies (rather than Hausdorff group topologies); {fourth, in the primitive case, {without imposing any separation assumptions, we show there are no non-indiscrete semigroup topologies strictly coarser than $\pw$.}\\

In the non-transitive context, minimality of $\pw$ can fail by localising the topology to an orbit (an example where this happens is  the random bipartite graph). On top of our theorem applying to the aforementioned examples of simple one-based $\omega$-categorical structures with weak elimination of imaginaries, Theorem~\ref{thm:onebased} works under more general assumptions which include, for example, various simple (not one-based) $\omega$-categorical Hrushovski constructions~\cite{Wagner:ST,ghaddln24}. Since the original argument in~\cite{ghaddln24} heavily relied on taking conjugates of pointwise stabilisers, our proof follows a substantially different route, with the initial portion of the argument for \cref{mainthm: c} running in parallel to that for \cref{mainthm: b}.\\

We conclude by studying topologies on the semigroup of elementary embeddings of Ury\-sohn spaces. Recall that in the group case, the presence of a metric is a source of failure of  minimality for $\pw$. 
By adapting techniques from~\cite{pinskerschindler} to the metric setting, we show that under general conditions, the Zariski topology is a Hausdorff semigroup topology and hence minimal amongst  all such topologies on the monoid of elementary embeddings; that it agrees with the metric pointwise topology $\mpw$; and that it is strictly coarser than $\pw$. In particular, we obtain the following (for the mentioned spaces the isometric embeddings are precisely the elementary embeddings):

\begin{thmx}[{\cref{thm:urysohnmain}}] \label{mainthm: d}Let $\U$ be any of the following Urysohn spaces:
\begin{itemize}
    \item the rational Urysohn space;
    \item the rational Urysohn sphere;
    \item the real Urysohn space;
    \item the real Urysohn sphere;
    \item the rational ultrametric Urysohn space.
\end{itemize}
Then, the semigroup Zariski topology $\zar$ on the space $\mathrm{IE}(\U)$ of isometric embeddings of $\U$ is strictly coarser than $\pw$ and  
minimal amongst Hausdorff semigroup topologies on $\mathrm{IE}(\U)$. Moreover, $\zar$ %
{equals} the metric pointwise topology $\mpw$.
\end{thmx}

\section{Preliminaries}\label{sect:prelims}

{Throughout this paper, $\Omega$ denotes an infinite set. In Sections~\ref{sec:chains} to~\ref{sec:discussion}, we will further assume that $\Omega$ is countable unless we specify otherwise.}
Groups and monoids in this paper will 
usually act {faithfully} on $\Omega$. We write $\fps{\Omega}$ for the collection of finite subsets of $\Omega$.

\subsection{Group actions}

Given a group action  $G\acts\Omega$ and $B\in[\Omega]^{<\omega}$, we write $G_B$ for the \textbf{pointwise stabiliser} of $B$ and $G_{\{B\}}$ for its \textbf{setwise stabiliser}. The \textbf{definable closure of} $B$, denoted by $\dcl(B)$,  is the set of all elements in $\Omega$ which are fixed by the action of 
$G_B$.  The \textbf{algebraic closure of} $B$, denoted by $\acl(B)$, is the set of elements of $\Omega$ with finite orbit under the action of $G_B$.  %
We say that $G\acts\Omega$ has \textbf{locally finite} algebraic closure if $\acl(B)$ is finite for all  $B\in[\Omega]^{<\omega}$.
We denote by $\X$ the set of all finite  \textbf{algebraically closed} subsets of $\Omega$, i.e.~those $B\in[\Omega]^{<\omega}$ with $\acl(B)=B$. {The action $G\acts\Omega$ has \textbf{no algebraicity} if all finite sets are algebraically closed.}\\

For $\Delta\subseteq\Omega$ infinite, we say that $\Delta$ is algebraically closed if for all finite $B\subseteq\Delta$, $\acl(B)\subseteq\Delta$.
Note that this definition of algebraic closure on infinite sets does not agree in general with group-theoretic algebraic closure (cf.~\cite[Chapter 4.1]{hodges1993model}).\\ 

    For a tuple $\overline{a}:=(a_1, \dots, a_n)$ of elements of $\Omega$, we write $\acl(\overline{a}):=\acl(\{a_1, \dots a_n\})$ for its algebraic closure. %
    Similarly, for tuples of elements $\overline{b}_1, \dots, \overline{b}_n$ from $\Omega$, we write $\acl(\overline{b}_1 \dots \overline{b}_n)$ or $\acl(\overline{b}_1, \dots, \overline{b}_n)$ for the algebraic closure of the set of all elements appearing in them.\\

    {Let $G\acts\Omega$ be a group action with locally finite algebraic closure. We say that algebraic closure is \textbf{modular} 
    if for all $A,B,C\in\X$ with $A\subseteq B$,}  
  \[\acl(A\cup(C\cap B))=\acl(A\cup C)\cap B\;.\]  

{Clearly, algebraic closure is modular for any group action with no algebraicity. Classic examples of structures whose automorphism groups have modular algebraic closure are %
{infinite-dimensional vector spaces over finite fields.}
In Section~\ref{sec:discussion}, we  discuss further  examples.\\

{If $n\geq 1$ and  $\a\in \Omega^n$, then we write $G\cdot \a$ for the orbit of $\a$ under the diagonal action $G\acts \Omega^n$.} A group action $G\acts\Omega$ is called \textbf{oligomorphic} if $G\acts\Omega^n$ has finitely many orbits for all $n\geq 1$.  It is well-known and easy to prove that oligomorphic actions have locally finite algebraic closure.\\

Given two finite tuples $\overline{c},\overline{c}'$ of elements of $\Omega$ of the same length, we write $\overline{c}\orq\overline{c}'$ to denote the fact that $\overline{c}$ and $\overline{c}'$ belong to the same orbit under the diagonal action of $G$. For $C, C'\in\fps{\Omega}$, we write $C\orq C'$ to indicate that there are enumerations $\overline{c}$ and $\overline{c}'$ of $C$ and $C'$ respectively such that $\overline{c}\orq\overline{c}'$. For $D\in\fps{\Omega}$, we write $\overline{c}\orq_{D}\overline{c}'$ as an abbreviation for  $\overline{c}\overline{d}\orq \overline{c}'\overline{d}$, where $\overline{d}$ is some enumeration of $D$. Similarly, for $D, C,C'\in\fps{\Omega}$, we write $C\orq_{D}C'$ if there {are} %
enumerations $\overline{c}$ and $\overline{c}'$ of $C$ and $C'$ respectively such that $\overline{c}\orq_{D}\overline{c}'$.%

\begin{observation}\label{closure transitivity and local finiteness}
	{Let $G\acts \Omega$, and consider the binary relation $R$ on $\Omega$ defined by  $R(a,b)$ if and only if $a\in\acl(b)$. Then $R$ is an equivalence relation when restricted to any  $G$-orbit $\Delta$ with the property that $G\acts\Delta$ has locally finite algebraic closure.}%
\end{observation}
\begin{proof}
	 By definition of $\acl$, the relation $R$ is reflexive. Since {for all $a,b\in\Omega$}, $a\in\acl(b)$ implies $\acl(a)\subseteq\acl(b)$, it is also transitive. For symmetry, let $a,b\in \Delta$, and assume  $a\in\acl(b)$. Then  $\acl(a)\subseteq\acl(b)$. Since the size of the algebraic closure is the same for all elements of $\Delta$, finiteness now implies  $\acl(a)=\acl(b)$, and hence  $b\in\acl(a)$.
\end{proof}

 {Sometimes, it will be convenient to treat tuples as sets. In such cases, for tuples $\a$ and $\b$, we write %
      $\overline{a}\cap\overline{b}$ for the set of elements of $\Omega$ that belong to both $\overline{a}$ and $\overline{b}$. We define $\overline{a}\cup\overline{b}$ and $\overline{a}\subseteq\overline{b}$ similarly.

 \subsection{Semigroup topologies}

 {A \textbf{semigroup topology} on a semigroup $S$ is a topology for which multiplication $S\times S\to S$ is (jointly) continuous; we do not include any separation axiom in this definition, and state explicitly when a topology is assumed to be $T_0$, $T_1$ or Hausdorff. A group endowed with a semigroup topology is also known as a \textbf{paratopological group}.}

 {Given a topological space $(X, \tau)$, for $a\in X$, we shall write $\mathcal{N}_\tau(a)$ for the set of $\tau$-open neighbourhoods of $a$.}

\subsubsection{The topology of pointwise convergence}\label{subsubsect:pw}

Let $\Omega$ be a set and consider the \textbf{full monoid} $\Omega^\Omega$, where multiplication is given by composition of functions.
Given $C\in\fps{\Omega}$ and a submonoid $S\subseteq\Omega^{\Omega}$ we write $\nn_{C}$ for the collection of elements of {$S$}
fixing $C$ pointwise. For a {finite}  partial function $\phi\colon A\to B$, where  {$A,B\in[\Omega]^{<\omega}$}, we write $\nn_{\phi}$ for the set of elements 
$s\in S$ with $s_{\restriction A}=\phi$. {For $a,b\in\Omega$, we write $\langle a, b\rangle$ for the map sending $a$ to $b$.} The monoid $\Omega^\Omega$ is naturally endowed with the \textbf{topology of pointwise convergence} $\pw$, for which a basis is provided by the sets $S_\phi$. {(Similarly, the sets $S_{\langle a, b\rangle}$ for $a,b\in\Omega$ are a subbasis for this topology.)} \\

When $\Omega$ is countable, $\pw$ makes   $\Omega^\Omega$ a Polish topological semigroup, i.e.~it is separable and completely metrizable, and multiplication is continuous. Similarly, any closed submonoid of $\Omega^\Omega$ endowed with the subspace topology is still a Polish topological semigroup. %
Given a group $G$ of permutations of $\Omega$, we denote by $\overline{G}$ the closure of $G$ in $\Omega^\Omega$ with respect to $\pw$. It is easy to see this is also a monoid.

\subsubsection{The semigroup Zariski topology}
Let $S\subseteq \Omega^\Omega$ be a monoid. The \textbf{Zariski  topology on $S$} is the topology generated by solution sets of inequalities of the form
\begin{equation}\label{eq:inequality}
    \lambda_{n} s \lambda_{n-1}s\cdots \lambda_1 s \lambda_0\neq \eta_{m} s \eta_{m-1} s \cdots \eta_1 s \eta_0\;,
\end{equation}
 where $\lambda_0, \dots, \lambda_n,\eta_0,\dots, \eta_m\in S$. It need not be Hausdorff or a semigroup topology in general. However, for $S$ the endomorphism monoid of an $\omega$-categorical structure, it often coincides with the pointwise convergence topology on $S$, with the {only counterexamples provided so far in~\cite{pinskerschindler} and~\cite{marimon2025guide}.}
 In general, it is contained in any $T_1$ semigroup topology on $S$~\cite[Proposition 2.1]{elliott2023automatic}. We say that a topology on a semigroup $S$ is \textbf{semitopological} (or shift-continuous) if for each $s\in S$, the shift-maps $\lambda_s, \rho_s: S\to S$ given by $u\mapsto us$ and $u\mapsto su$ respectively are continuous. The Zariski topology is semitopological on any semigroup~\cite[Proposition 2.2]{elliott2023automatic}.

 \subsubsection{Localised and generalised pointwise topologies}\label{subsub:localised}

{Let $S\subseteq \Omega^\Omega$ be a monoid, and let  $\Delta\subseteq \Omega$ be $S$-invariant. Then, by the \textbf{localised pointwise convergence topology} at $\Delta$, denoted by  $\pw^\Delta$, we mean the topology on $S$ whose subbasis of open sets is given by the 
sets ${\nn}_\phi$ as in Section~\ref{subsubsect:pw} where $\phi\colon A\to B$ for  $A,B\subseteq \Delta$ finite.}\\
  
{Whilst the topologies $\pw^\Delta$ can yield Hausdorff group topologies coarser than $\pw$ on automorphism groups of $\omega$-categorical structures (\cref{ex:bipartite}), we will see in \cref{prop:localized are not T1} that, {unless they equal $\pw$,} they 
are %
{not}
Hausdorff on their monoids of elementary embeddings.}\\

Localised pointwise topologies may be further generalised to what we shall call generalised pointwise topologies. 
 These topologies are needed for the precise statement of \cref{thm:onebased} (\cref{mainthm: c} in the introduction); they will not be needed outside~\cref{section groups}.\\

Let $\Omega$ be a countable set and let $G\acts\Omega$ be a group acting on $\Omega$ with locally finite algebraic closure.
  Given $a\in\Omega$, we write $[a]:=\cl(a)\cap G\cdot a$. {It follows from~\cref{closure transitivity and local finiteness} that the sets of the form $[a]$ form a partition of $\Omega$ into finite sets, and hence are the equivalence classes of an equivalence relation $\sim$ with finite classes.} 
	    {Let $\Delta\subseteq \Omega$ be $G$-invariant.}
    {We write $G(\Delta)$ for the group of all permutations of $\Delta$ which on each finite subset of $\Delta$ agree with some element of $G$.}  %
     {By $\gc(\Delta)$ we denote the set of elements of $G(\Delta)$ that %
     {preserve each $\sim$-class setwise.} %
     It is easy to see that $\gc(\Delta)$ is a profinite normal subgroup of $G(\Delta)$ (\cref{ob:gcdelta}).}

  \begin{definition} \label{d:generalized pointwise stabilizer}
    {Let $\Delta\subseteq \Omega$ be $G$-invariant, and let $K$ be a topologically closed subgroup of  $\gc(\Delta)$ which is normal in $G(\Delta)$. %
    The  \textbf{generalised pointwise} topology  $\pw^{\Delta,K}$ on {$\overline{G}$}
    is given by the collection of basic open sets of the form} 
  	\begin{equation}\label{basic neighourhoods}
     \overline{G}_{\phi}^{K}:=\{s\in %
     \overline{G}
     \;|\;\exists h\in K\,\,s_{\restriction A\cap\Delta}=(\phi h)_{\restriction A\cap\Delta}\}\;, 
  	\end{equation}
  	{where $\phi$ is the restriction of an element of $G$ to a set in  $\X_\Delta:=\{A\cap\Delta\;|\; A\in\X\}$.
    We will also write $\overline{G}_{A}^{K}$ for $\overline{G}_{\mathrm{Id}_{A}}^{K}$.} {When $\Delta=\emptyset$, we take $G(\Delta)=\gc(\Delta)=\{1\}$, and so we view the indiscrete topology $\tau_\mathrm{ind}:=\{\emptyset, \overline{G}\}$ as the generalised pointwise topology associated with this degenerate case.}
  \end{definition}
 {The group $G$ inherits the  generalised topology from $\overline{G}$ by restriction.} {Note that by setting $K=\{1\}$,  localised topologies are just a special case of generalised pointwise topologies.}%

\subsection{Model-theoretic background} 

Essentially all groups and semigroups we will study arise as the automorphism groups and monoids of elementary embeddings of countable first-order structures. We expect some familiarity with basic concepts in model theory such as that of a first-order structure, types, and elementary embeddings. We refer the reader to~\cite[Chapters 1-4]{tent2012course}.

\subsubsection{Structures and their symmetries}  We denote structures by $\M$ or $\M'$ and their domains by $\Omega$ or $\Omega'$. We denote first-order theories by $T$ and first-order languages by $\mathcal{L}$ (possibly with indices).  With the exception of Section~\ref{sec:uri}, we always work in a countable language. Given $\mathcal{L}$-structures $\M$ and $\M'$, an \textbf{elementary embedding} $f:\M\to\M'$ is an injective map preserving the truth of all first-order formulas. If $\M$ is a substructure of $\M'$ with the inclusion map being elementary, we write $\M\preceq \M'$. We write $\Aut(\M)$ for the automorphism group of $\M$ and $\EEmb(\M)$ for its monoid of elementary (self-)embeddings.

\subsubsection{Saturation, $\omega$-categoricity, homogeneity}  We say that a countable structure $\M$ is $\bm{\omega}$\textbf{-categorical} if its automorphism group $\Aut(\M)$ is oligomorphic. %
Given a structure $\M$ and an infinite cardinal $\kappa\leq|\Omega|$, we say that $\M$ is \textbf{$\kappa$-saturated} if for all subsets $B\subseteq \M$ with $|B|<\kappa$ and all $n\in\mathbb{N}$, $\M$ realises all $n$-types over $B$. We say that $\M$ is \textbf{saturated} if it is $\kappa$-saturated for $\kappa=|\Omega|$. It is easy to see that countable saturated structures are such that $\overline{\Aut(\M)}=\EEmb(\M)$. Moreover, two saturated models of the same complete theory and cardinality are always isomorphic~\cite[Lemma 5.2.8]{tent2012course}, and every $\omega$-categorical structure is saturated. It is also well-known that a complete theory (possibly in an uncountable language) has a countable saturated model if and only if it is \textbf{small}, that is, it has only countably many $n$-types for each $n\in\mathbb{N}$~\cite[Lemma 4.3.9]{tent2012course}. Moreover, whenever $G\acts \Omega$ is the automorphism group of a countable saturated structure $\M$, {there is a countable saturated structure $\M'$ in a countable language with the same automorphism group.} {We will also often implicitly make use of the fact that if $\M''$ is a countable model of a small theory, then $\M''$ has a countable saturated elementary extension $\M\succeq\M''$.}

We shall often work under the following assumption, which we refer to as \ref{sat}:

  	\begin{enumerate}[label=($\clubsuit$),ref={$(\clubsuit)$}]
 
  \item \label{sat} $G$ is the automorphism group of  a countable saturated  structure $\M$ with  domain $\Omega$ and has locally finite algebraic closure (hence, {$\overline{G}=\EEmb(\M)$}).
  	\end{enumerate}

A  structure $\M$ is \textbf{homogeneous} if any isomorphism between finitely generated substructures of $\M$ extends to an automorphism of $\M$. The \textbf{age} of a structure $\M$ is the class of finitely generated substructures of $\M$, up to isomorphism. Fra\"{i}ss\'{e}'s Theorem (see \cite[Theorem 4.4.4]{tent2012course}) states that an isomorphism-closed  class $\mathcal{K}$ of finitely generated structures in a countable language $\mathcal{L}$ is the age of a countable homogeneous structure $\mathcal{M}$ if and only if it is closed under substructures and has the joint embedding and amalgamation properties; $\mathcal{M}$ is then unique up to isomorphism, and called the \textbf{Fra\"{i}ss\'{e} limit of $\mathcal{K}$.} 
 Well-known examples of  homogeneous structures include the random graph (the Fra\"{i}ss\'{e} limit of the class of finite graphs),  $(\mathbb{Q};<)$, any infinite-dimensional vector space over a finite field (in the language with scalar functions for the field elements), and the rational Urysohn space, whose age {is the class of} finite metric spaces with rational distances. Homogeneous structures in a finite relational language (even with finitely many constants) are $\omega$-categorical, and called \textbf{finitely homogeneous structures}.

\subsubsection{Algebraic independence}

 {Relations of independence and related techniques from model theory will play a central role in Sections~\ref{section groups} and~\ref{sec:min monoids}. In this paper, we are mainly interested in algebraic independence, the natural notion of independence arising from algebraic closure (\cref{def:alg ind}),  which generalises linear independence in the context of vector spaces. Following modern model-theoretic approaches initiated by~\cite{adler2005explanation, adler2009geometric}, in Section~\ref{sec: calcindependence}, we introduce a more abstract calculus of %
 independence relations %
  in order to make our arguments more transparent and more easily generalisable.}  

 \begin{definition} \label{def:alg ind}
 {Let $G\acts\Omega$.} %
     We define \textbf{algebraic independence} $\ind^{a}$ as follows: for all %
    $A,B,C\in[\Omega]^{<\omega}$,
    \[A\ind^{a}_C B \text{ if and only if } \acl(A\cup C)\cap\acl(B \cup C)=\acl(C)\;.\]
        {In the above setting, we say that $A$  is \textbf{ algebraically independent from} $B$ \textbf{over} $C$.}
    {If $C=\emptyset$, we omit it, writing only %
    ${A\ind^a B}$.  
    We also apply this notation for tuples instead of sets.}
    \end{definition}

  {We shall frequently make use of Neumann's Lemma~\cite[Fact 3.1.5]{d2023axiomatic} (cf.~\cite{neumann1976structure}):
   \begin{lemma}[Neumann's Lemma] {Let $G\acts\Omega$ with locally finite algebraic closure. Then, for any $A, B, C\in\fps{\Omega}$, there is $A'\equiv_C A$ such that $A'\ind^{a}_C B$.}
   \end{lemma}

    Below, we define independent $3$-amalgamation, which is needed for the statement of \cref{mainthm: c}, and which will thus be important in Section~\ref{section groups}. This notion arises from the study of simple theories in model theory~\cite{kim1997simple, kim2013simplicity}. Indeed, countable saturated models of simple one-based theories with weak elimination of imaginaries serve as our main example of structures for which algebraic independence satisfies this property. We further discuss this point in 
    \cref{sec:discussion}. 

    \begin{definition} Let $G\acts\Omega$ have locally finite algebraic closure. We say that algebraic %
    {independence} 
    satisfies \textbf{independent $\bm{3}$-amalgamation} if whenever $A,B_{1},B_{2},C_{1},C_{2}\in[\Omega]^{<\omega}$ are algebraically closed and such that $B_{1}\ind^{a}_{A}B_{2}$, $A\subseteq B_{i}$ and $C_{i}\ind^{a}_{A}B_{i}$ for $i=1,2$ and $C_{1}\orq_{A}C_{2}$,  there exists $D\in[\Omega]^{<\omega}$ such that
    $D\orq_{B_{i}} C_{i}$ for $i=1,2$ and $D\ind^{a}_{A}(B_{1}\cup B_{2})$.
    \end{definition}

\section{Failures of Hausdorffness in topologies coarser than \texorpdfstring{$\pw$}{the topology of pointwise convergence}}
In this section, we investigate certain (not necessarily strictly) coarser topologies than $\pw$ on monoids of elementary embeddings of countable structures. In particular, in Subsection~\ref{sub:localised}, we prove that localised topologies which are not equal to $\pw$ fail to be $T_1$; and in Subsection~\ref{sub:zar}, we prove that the semigroup Zariski topology is not Hausdorff as long as the automorphism groups of structures have locally finite algebraic closure and non-trivial centre.

\subsection{{Proper} localised topologies {are not \texorpdfstring{$T_1$}{T1} on 
\texorpdfstring{$\EEmb(\M)$}{monoids of elementary embeddings}} }\label{sub:localised}

 Localised pointwise topologies are a natural source of Hausdorff group topologies coarser than $\pw$ in automorphism groups. In this subsection we give a quick proof that on monoids of elementary embeddings of countable saturated structures they are never $T_1$. The proof uses some basic saturation arguments from model theory. We hope it may help the reader appreciate what is to come in terms of the topological behaviour of $\overline{G}$ being very different from that of $G$.

 \begin{example}\label{ex:bipartite}  Consider the random bipartite graph $\mathcal{B}$ with orbits  $X$ and $Y$. This is just the Fra\"{i}ss\'{e} limit of the class of bipartite graphs where each side of the partition is named by a unary predicate.  Let $G:=\Aut(\mathcal{B})$ and consider the localised pointwise convergence topology $\pw^X$. It is easy to observe that the identity is closed in $\pw^X$ (on either $G$ or $\overline{G}$)~\cite[Example 3D]{ghaddln24}. For $G$ this is equivalent to $\pw^X$ being Hausdorff, and indeed $\pw^X$ is a Hausdorff group topology on $G$ strictly coarser than $\pw$. However, as we will see in \cref{prop:localized are not T1}, $\pw^X$ on $\overline{G}$ is not even $T_1$.
  \end{example}

\begin{lemma}\label{lem:saturatedcase} Let $G\acts\Omega$ be the automorphism group of a countable saturated structure. Let $X,Y\subseteq\Omega$ be disjoint. Suppose that for all $B\in[X]^{<\omega}$ %
{we have that} $\dcl(B)\cap Y=\emptyset$. Then, for each $a\in Y$, there are $\alpha, \beta\in\overline{G}$ such that $\alpha\upharpoonright_X=\beta\upharpoonright_X$ and $\alpha(a)\neq\beta(a)$.

\end{lemma}
\begin{proof} We consider the countable saturated structure $\M$ with automorphism group $G\acts\Omega$. Let $\overline{b}:=(b_0, b_1, b_2, \dots)$ be a tuple of length $\omega$ enumerating $\M$ with $a=b_0$. Consider the infinitary type $p(\overline{x}):=\mathrm{tp}(\overline{b})$. For the tuples of variables of length $\omega$, $\overline{x}$ and $\overline{y}$, consider the partial type $\Gamma(\overline{x}, \overline{y})$ given by
\[p(\overline{x})\cup p(\overline{y}) \cup\{x_0\neq y_0\}\cup\{x_i=y_i\vert\  b_i\in X\}.\]
Note that $\Gamma(\overline{x}, \overline{y})$ is finitely consistent. To see this, for some finite indexing set $J\subseteq\omega$, consider the restriction of $\Gamma$ to variables $x_i$ and $y_i$ from $J$, which we call $\Gamma_J(\overline{x}_J, \overline{y}_J)$. Since $\dcl((b_i\vert i\in J)\cap X)\cap Y=\emptyset$, there is $g\in G$ such that $g\upharpoonright_{(b_i\vert i\in J)\cap X}=\mathrm{Id}_{(b_i\vert i\in J)\cap X}$ and $g(a)\neq a$, meaning that $\Gamma_J(\overline{x}_J, \overline{y}_J)$ is satisfied by the tuple $((b_i\vert i\in J),(g(b_i)\vert i\in J))$. Since $\Gamma(\overline{x}, \overline{y})$ is finitely satisfiable, by compactness and the L\"{o}wenheim-Skolem Theorem it is also satisfiable in some model $\M'$, which can be chosen to be countable and saturated {by possibly taking a suitable elementary extension}. In particular,  $\M'$ is isomorphic to $\M$ and so, we know that $\Gamma(\overline{x}, \overline{y})$ is satisfied in $\M$ by a tuple $(\overline{c}, \overline{d})$. Consider the maps $\alpha$ and $\beta$ given by $\alpha: b_i\mapsto c_i$ and $\beta:b_i\mapsto d_i$ for $i<\omega$. Since $\alpha$ and $\beta$ preserve the types of arbitrary finite tuples, they are elementary embeddings of $\M$ by construction. Moreover, by construction, $\alpha(a)=c_0\neq d_0=\beta(a)$ and for each $b_i\in X$, $\alpha(b_i)=\beta(b_i)$, meaning that $\alpha\upharpoonright_X=\beta\upharpoonright_X$, as desired.
\end{proof}

\begin{prop}\label{prop:localized are not T1}
    Let $G\acts\Omega$ be the automorphism group of a countable saturated structure. Let $X\subseteq\Omega$ be a $G$-invariant set. Then, $\pw^X$ is either equal to $\pw$ or not $T_{1}$ on $\overline{G}$. 
\end{prop}
\begin{proof} Firstly, note that if $Z\subseteq\Omega\setminus X$ is a $G$-orbit and for some $B\in[X]^{<\omega}$ we have  $\dcl(B)\cap Z\neq \emptyset$, then $\pw^{X}=\pw^{X\cup Z}$. {To see this, note first that trivially $\pw^{X}\subseteq \pw^{X\cup Z}$. For the other inclusion,   consider an arbitrary non-empty subbasic open   set %
$\G_{\langle a, b\rangle}$
of $\pw^{X\cup Z}$, %
for $a,b\in X\cup Z$,  and let $\gamma\in \G_{\langle a, b\rangle}$. Since both $X$ and $Z$ are $G$-invariant, either $a,b\in X$, or $a,b\in Z$. In the first case, we have $\G_{\langle a, b\rangle}\in \pw^X$, so consider  the second case. Then there is some $B\in\fps{X}$ such that $a\in\dcl(B)$. In particular,  $b\in\dcl(\gamma(B))$, and so $\gamma\in\G_{\gamma_{\upharpoonright B}}\subseteq \G_{\langle a, b\rangle}$. Thus, $\G_{\langle a, b\rangle}\in \pw^X$,  whence  $\pw^X=\pw^{X\cup Z}$.} Therefore, we may assume without loss of generality that $X\subsetneq \Omega$ and for all $G$-orbits $Z\subseteq\Omega\setminus X$ and all $B\in[X]^{<\omega}$, $\dcl(B)\cap Z=\emptyset$. But then, we are in the setting of Lemma~\ref{lem:saturatedcase} and we can build $\alpha, \beta\in\overline{G}$ agreeing on $X$ but disagreeing on some point $a\in\Omega\setminus X$. In particular, $\beta$ belongs to any open neighbourhood of $\alpha$ in $\pw^X$, meaning that $\pw^X$ is not $T_{1}$.
\end{proof}

\subsection{On the centre and the Zariski topology}\label{sub:zar}

In this subsection, we prove that the semigroup Zariski topology $\tau_Z$ is not Hausdorff on $G$ or $\overline{G}$ whenever $G\acts\Omega$ has locally finite algebraic closure and non-trivial centre. These assumptions are easily satisfied by various natural infinite permutation groups such as the automorphism group $\mathrm{GL}(\aleph_0, q)$ of a countably infinite vector space over a finite field of size $q>2$, or the automorphism group of infinitely many disjoint copies of an edge $K_2$. Indeed, the latter case was at the core of the original example of~\cite{pinskerschindler} of the endomorphism monoid of a finitely homogeneous structure  for which $\zar$ is not Hausdorff.\\ %

Let $S$ be a semigroup. A \textbf{semigroup mixed equality} is an expression of the form
\begin{equation}\label{eq:cequality}
   \lambda_{n}  x \lambda_{n-1}  x \cdots   \lambda_1 x \lambda_0= \eta_{m}\  x \eta_{m-1} \ x \cdots  \eta_1 
   x  \eta_0\;,
\end{equation}
where $\lambda_0, \dots, \lambda_n,\eta_0,\dots, \eta_m\in S$ for $n\geq m$, and $x$ is a variable. By a \textbf{semigroup mixed disequality}, we mean an expression of the form of (\ref{eq:cequality}) with $\neq$ instead of $=$. One may think of a semigroup mixed equality as a positive atomic formula in one free variable in the language of semigroups (with equality and composition) expanded by constants for elements of $S$, and similarly of a mixed disequality as a negated atomic formula. We say that a semigroup mixed equality/disequality is \textbf{singular} if $n=m$. Otherwise, we say it is \textbf{regular}. {Our key insight for proving non-Hausdorffness will be that locally finite algebraic closure of $G$ implies that  open sets of $\zar$ defined by regular mixed disequalities are large (they are open and dense in $\pw$); and non-trivial central elements of $G$ cannot be distinguished from the identity  by open sets defined by singular mixed disequalities.}

\begin{notation} 
{Below, $S\subseteq \Omega^\Omega$ is a transformation monoid, and we consider elements $f \in S^2$. For $a \in \Omega^2$, we write $fa$ for the natural (componentwise) action: writing $f = (f', f'')$ and $a = (a', a'')$, we set $fa = (f'a', f''a'')$. For $a \in \Omega^2$ and $\delta \in S$, we write $\delta a$ for $(\delta a', \delta a'')$, the diagonal action of $\delta$ on $\Omega^2$.}
\end{notation}
\begin{definition}\label{def:freesequence} Let $G\acts\Omega$, and let $S\subseteq \Omega^\Omega$ be a {transformation monoid} with $G\subseteq S$. %
 Consider sequences $\overline{f}:=(f_{n}, \dots, f_0)$ where for each $i\leq n$, $f_i\in S^2$ and $\a:=(a_{n}, \dots, a_0)$ such that for each $i\leq n$, $a_i\in\Omega^2$. Further consider finite tuples $\b$ and $\b'$ from $\Omega$ such that
$\b\orq \b'$. We say that $\a$ is a \textbf{free} sequence with respect to $\f$ over $(\b, \b')$ if {$a_0$ does not contain any element in $\acl(\emptyset)$ and } the following two conditions hold:
\begin{itemize} 
\item  There exists some $\delta\in G$ such that $\delta(\overline{b})=\overline{b}'$ and $a_{i+1}=\delta f_i a_i$ for all $i<n$. That is,
\[ %
(f_{n-1} a_{n-1}) \dots (f_{0}a_{0})\b\orq %
a_n\dots a_1\b'\;;\] 

\item For all $i<n$,
\[f_{i+1}a_{i+1}\ind^{a}%
(f_ia_i)\dots (f_0a_0)\b\;.\]
\end{itemize} 
\end{definition}

\begin{lemma}\label{lem:buildfree} Let $G\acts \Omega$ have  locally finite algebraic closure. Let $S\subseteq \Omega^\Omega$ be a {transformation monoid} such that $G\subseteq  S\subseteq  \G$. Let $\a$ be a free sequence with respect to $\f$ over $(\b, \b')$ as in Definition~\ref{def:freesequence}, and let  $f_{n+1}\in {S}^2$. Then, there is $a_{n+1}\in\Omega^2$ such that $(a_{n+1},\a)$ is free with respect to $(f_{n+1},\f)$ over $(\b, \b')$.
\end{lemma}
\begin{proof} Let 
{$\b_{n}:=(f_{n-1} a_{n-1}, \dots, f_{0}a_{0},\b)$} and $\b_{n}':=(a_{n},\dots, a_1,\b')$.
 By freeness, there is an element $\delta\in G$ such that   $\delta \b_n=\b_n'$. Setting $c:=\delta(f_na_n)$, we then have $(f_{n}a_{n},\b_n)\orq (c,\b_n')$.  By Definition~\ref{def:freesequence}, we have  $f_{n}a_{n}{ \ind^{a} }%
 \b_{n}$, and hence $c \ind^{a} %
 \b_{n}'$. Consider the finite tuple  %
 $f_{n+1}^{-1}(f_na_n,\b_n)$; this is well-defined since both components of  $f_{n+1}$ are injective. By Neumann's Lemma, there is $a_{n+1}\orq_{\b_{n}'}c$ such that %
 $a_{n+1}\ind^{a}_{\b_{n}'} f_{n+1}^{-1}(f_{n}a_n,\b_n)$. Since also $a_{n+1}\ind^{a} \b_n'$, this implies $a_{n+1}\ind^{a}  f_{n+1}^{-1}(f_{n}a_n,\b_n)$. Then $S\subseteq \overline{G}$ implies 
 \[f_{n+1}a_{n+1}\ind^{a}(f_n a_n,\overline{b}_n)=( f_na_n,\dots, f_0a_0,\b),\]
as required in the second condition of the definition of freeness. 
Since $a_{n+1}\orq_{\b_{n}'}c$, we have $(f_{n}a_{n},\b_n)\orq (a_{n+1},\b_n')$, as required by the first condition.
\end{proof}

\begin{lemma}\label{lem:comeagre1} Let $G\acts\Omega$ have locally finite algebraic closure. Let $S\subseteq \Omega^\Omega$ be a transformation monoid such that $G\subseteq S\subseteq \G$. Then, the set of solutions to any regular semigroup mixed {disequality} %
for $S$ is %
{dense and open} in $S$ with respect to $\pw$.
\end{lemma}
\begin{proof} Consider a regular semigroup mixed %
{disequality} 
of the form of (\ref{eq:cequality}) and let $\U$ be {the set of elements of {$S$}
satisfying it. Then $\U$ is $\pw$-open since $\zar\subseteq \pw$.} %
To prove density of $\U$, consider a basic open set   $\nn_{(\b,\b')}$ of $S$.  Let 
\[\f:=((\lambda_n, 1),\dots, (\lambda_{m+1}, 1), (\lambda_m, \eta_m),(\lambda_{m-1}, \eta_{m-1}), \dots, (\lambda_0, \eta_0)).\] 
Select $e_0\ind^a f_0^{-1}(\b)$ with $e_0\not\in\acl(\emptyset)$, and set $a_0:=(e_0, e_0)$. Then $a_0$ is free with respect to $f_0$ over $(\bar b,\bar b')$. We can then use Lemma~\ref{lem:buildfree} to build a free sequence $\a$ with respect to $\f$ over $(\b, \b')$.  Since $n>m$, at the end of this process, 
\[f_n a_n\ind^a(f_m a_{m})\dots  (f_0 a_0)\b.\]

{In particular, writing $a_i=(a_i^1,a_i^2)$ for all $i<n$, we have that %
$\lambda_n a^1_n\neq \eta_ma_m^2$.} %
{Since $\a$ is a free sequence with respect to $\f$ over $(\b, \b')$,} there is $\delta\in G$ such that $\delta\b=\b'$ and $a_{i+1}=\delta {f_i}
a_i$ for all $i<n$. Since %
$\lambda_n a^1_n\neq \eta_m a_m^2$, we see that $\delta$ satisfies the given disequality.
\end{proof}

Lemma~\ref{lem:comeagre1} has some interesting consequences for the study of mixed identities in semigroups. Let $S$ be a semigroup. A semigroup mixed equality for $S$ is a \textbf{semigroup mixed identity} if its solution set is all of $S$.  \cref{lem:comeagre1} tells us that %
in the context of locally finite algebraic closure, regular semigroup {disequalities} have large sets of solutions:

\begin{corollary}\label{cor:mixedidentitiessingular} Let $G\acts\Omega$ have locally finite algebraic closure. Let $S\subseteq \Omega^\Omega$ be a transformation monoid such that $G\subseteq S\subseteq \overline{G}$. Then all semigroup mixed identities of $S$ are singular.
\end{corollary}
{We remark that the study of mixed identities in groups (where mixed equalities are given by group rather than semigroup words) has a long history, with some recent results focusing on oligomorphic groups or groups with no algebraicity, see e.g.~\cite{bodirsky2025mixed,etedadialiabadi2021dense,slanina2019laws, marimon2026all}. In particular, the conclusion of Corollary~\ref{cor:mixedidentitiessingular} has been shown recently for all groups with an action with no algebraicity in~\cite{marimon2026all}. This stronger assumption (compared to locally finite algebraic closure) is due to the fact that group words are more expressive than semigroup words.}\\

{In the following, we denote the centre of a semigroup $S$ by $Z(S)$.}

\begin{lemma}\label{lem:centretransfer} Let $G\acts\Omega$. Suppose that $\gamma\in Z(G)$. Then, $\gamma\in Z(\G)$. 
\end{lemma}
\begin{proof} Let $\alpha\in \G$. Take $a\in \Omega$. There is some $\alpha'\in G$ such that $\alpha'_{\upharpoonright \{a,\gamma(a)\}}=\alpha_{\upharpoonright \{a,\gamma(a)\}}$. Hence, 
\[\gamma\alpha(a)=\gamma\alpha'(a)=\alpha'\gamma(a)=\alpha\gamma(a),\]
where the first and last equalities follow from the choice of $\alpha'$ and the second equality from $\gamma\in Z(G)$. Since the equation above holds for all $a\in \Omega$, it follows that $\gamma\alpha=\alpha\gamma$, and so that $\gamma\in Z(\G)$.
\end{proof}

\begin{theorem}\label{thm:zar} Let $G\acts \Omega$ have locally finite algebraic closure and non-trivial centre. Let $S\subseteq \Omega^\Omega$ be a transformation monoid such that $G\subseteq S\subseteq \G$. Then, the semigroup Zariski topology $\zar$ on $S$ is not Hausdorff.
\end{theorem}
\begin{proof} {Let $\gamma \in Z(G)$ be a non-trivial central element of $G$. Let $\mathcal{O}$ and $\mathcal{U}$ be any $\zar$-basic open neighbourhoods of $1$ and $\gamma$, respectively; we show that $\mathcal{O} \cap \mathcal{U} \neq \emptyset$. 
We may write $\mathcal{O} = \mathcal{O}_{\mathrm{s}} \cap \mathcal{O}_{\mathrm{r}}$, where $\mathcal{O}_{\mathrm{s}}$ and $\mathcal{O}_{\mathrm{r}}$ are the intersection of finitely many $\zar$-open sets defined by singular and regular disequalities, respectively.} %

{We  argue that $\gamma\in\mathcal{O}_{\mathrm{s}}$. In fact, $\gamma$ satisfies all singular disequalities
    satisfied by $1$. To see this, suppose that, for 
    {$\lambda_0, \dots, \lambda_n, \eta_0, \dots, \eta_n\in {S}$},
    we have that
    \[{\lambda_n\cdots \lambda_0\neq \eta_n\cdots \eta_0\;.}\]
    Then, we can deduce that}
    \begin{align*}
        \lambda_n\gamma\dots \lambda_1\gamma\lambda_0 &=\gamma^n\lambda_n\cdots \lambda_0\\
        &\neq \gamma^{n} \eta_n\cdots \eta_0\\
        &= \eta_n\gamma\dots \eta_1\gamma \eta_0\;,
    \end{align*}
{where the first and last equalities follow from $\gamma \in Z(S)$ (by Lemma~\ref{lem:centretransfer}), and the intermediate inequality follows from the injectivity of $\gamma\in S\subseteq\overline{G}$, and hence of $\gamma^{n}$.}

{It follows that $\mathcal{O}_{\mathrm{s}}\cap \mathcal{U}$ is non-empty. %
Now Lemma~\ref{lem:comeagre1} implies that $\mathcal{O}_{\mathrm{r}}$, as the finite intersection of $\pw$-dense and $\pw$-open sets, is itself $\pw$-dense. Since $\zar\subseteq \pw$, it then is also $\zar$-dense. We conclude that $\mathcal{O}_{\mathrm{r}}\cap \mathcal{O}_{\mathrm{s}}\cap \mathcal{U}$ is non-empty, which we had to show.  %
}
\end{proof}

\section{Chains, rigidity, and dispersivity}\label{sec:chains}

  \subsection{Chains}
 \newcommand{\fts}[1]{(#1)^{\omega}} 
 
 In this section, we work again with a group $G\acts\Omega$ with locally finite algebraic closure, and the semigroup (under composition) given by its topological closure $\overline{G}$ in the space $\Omega^{\Omega}$ with the 
topology of pointwise convergence $\pw$. 
 We let  $S\in\{G, \overline{G}\}$.  \\

For $k \geq 0$, a \textbf{chain of length $k$} is a $(k+1)$-tuple
$\C = (\overline{c}_i)_{i=0}^{k}$ where $\overline{c}_i$'s are
themselves finite tuples. When $k$ is known or implicit from the context, we
refer to $\C$ simply as a \textbf{chain}. If $\acl(\overline{c}_i) =
\overline{c}_i$ for every $i \leq k$ (identifying each tuple
$\overline{c}_i$ with the set of its coordinates), we call $\C$ a
\textbf{chain of algebraically closed sets}.

Given a chain $\C = (\overline{c}_i)_{i=0}^{k}$, we write
$\nn_{(\overline{c}_i)_{i=0}^{k}}$, or simply $\nn_{\C}$, for the product
\begin{equation*}
  \nn_{\overline{c}_0}\,\nn_{\overline{c}_1}\cdots\nn_{\overline{c}_k},
\end{equation*}
where, for each $i \leq k$, $\nn_{\overline{c}_i}$ denotes the stabiliser
of $\overline{c}_i$ in $S$. Recall each $\nn_{\overline{c}_i}$ is a basic open set of the topology $\pw$.
   
 \begin{lemma}
 	\label{o:commutation}  For any 
   $C\in\fps{\Omega}$ 
and any $g\in G$, we have $\nn_{g(C)}g= g\nn_{C}$. 
 \end{lemma}
 \begin{proof}
 For the inclusion $\nn_{g(\C)}\,g \subseteq g\,\nn_{\C}$, let
$h \in \nn_{g(\C)}$. Then $s := g^{-1}hg$ stabilises $\C$, that is,
$s \in \nn_{\C}$, and $hg = gs \in g\,\nn_{\C}$. The reverse inclusion is
analogous: given $s \in \nn_{\C}$, set $h := gsg^{-1} \in \nn_{g(\C)}$, so
that $gs = hg \in \nn_{g(\C)}\,g$. Hence $g\,\nn_{\C} \subseteq
\nn_{g(\C)}\,g$, and the two inclusions give
$\nn_{g(\C)}\,g = g\,\nn_{\C}$.
 \end{proof}
 
  \begin{observation}
	   \label{o:chain} Assume that we are given a chain $\C=(\overline{c}_{i})_{i=0}^{k}$, where each $\overline{c}_{i}\in\fps{\Omega}$. Then  $\nn_{\mathcal{C}}$ is precisely the set of $s\in S$ such that there exists a chain %
       ${\mathcal{C}'}:=(\overline{c}_{i}')_{i=0}^{k}$ with $\overline{c}_0'=\overline{c}_0$, $\overline{c}'_{k}=s(\overline{c}_{k})$, and $\overline{c}'_{i}\overline{c}'_{i+1}\orq \overline{c}_{i}\overline{c}_{i+1}$ for every $i<k$.  
  \end{observation}
  \begin{proof}
  Recall that $\nn_{\C} = \nn_{\overline{c}_0}\nn_{\overline{c}_1}\cdots\nn_{\overline{c}_k}$. First assume that $s\in \nn_{\C}$, so $s=g_0\cdots g_k$ with $g_i\in \nn_{\overline{c}_i}$ for each %
  {$i\leq k$.}
  Then set $\overline{c}_i':=g_0\cdots g_i(\overline{c}_i)$ for every $i$. 
  For $i < k$, the element $g_0 \cdots g_i$ witnesses
  $\overline{c}_i\,\overline{c}_{i+1} \orq \overline{c}'_i\,\overline{c}'_{i+1}$: it sends
  $\overline{c}_i$ to $\overline{c}'_i$ by definition, while, using $g_{i+1} \in \nn_{\overline{c}_{i+1}}$, \begin{equation*}
    g_0 \cdots g_i(\overline{c}_{i+1})
      = g_0 \cdots g_i g_{i+1}(\overline{c}_{i+1})
      = \overline{c}'_{i+1}.
  \end{equation*}
Conversely, suppose such a chain 
${\mathcal{C}'}$ exists. For each $i < k$ pick,
  using $\overline{c}_i\,\overline{c}_{i+1} \orq \overline{c}'_i\,\overline{c}'_{i+1}$, an element
  {$g_i \in G$}
  with $g_i(\overline{c}_i\,\overline{c}_{i+1}) = \overline{c}'_i\,\overline{c}'_{i+1}$.
  Set $g_k := s$, so that $g_k(\overline{c}_k) = s(\overline{c}_k) = \overline{c}'_k$ as well. Note that $    s = g_0 \,(g_0^{-1}g_1)\,(g_1^{-1}g_2)\cdots(g_{k-1}^{-1}g_k),
$ and  $g_0 \in \nn_{\overline{c}_0}$ and  $g_{i-1}^{-1}g_i\in \nn_{\overline{c}_i}$ for every $1\leq i\leq k$.  Hence
  $s \in \nn_{\overline{c}_0}\nn_{\overline{c}_1}\cdots\nn_{\overline{c}_k} = \nn_{\C}$.

  \end{proof}
  
\begin{observation}\label{obs: subchains}
 Let $\C=(\overline{c}_i)_{i=0}^k$ be a chain, let $J\subseteq \{0, \dots, k\}$ and consider the subchain $\D:=(\overline{c}_i)_{i\in J}$. Then, $\nn_{\D}\subseteq\nn_{\C}$.

\end{observation}
\begin{proof} %
Let $s\in\nn_{\D}$. By definition, $s=\prod_{j\in J} g_j$, {where the $j$'s are in increasing order and} %
$g_j\in\nn_{\overline{c}_j}$. For $i\leq k$, let
\[g'_i:=\begin{cases} g_i &\text{ if } i\in J;\\
1 &\text{ otherwise.}
\end{cases}\]
Then, $s=\prod_{j\in J} g_j=g'_0\cdot \dots \cdot g_k'\in \nn_{\C}$,
concluding the proof. 
\end{proof}

  \begin{corollary}
  	\label{change chains} Assume that $A\in\fps{\Omega}$ and $k\geq 0$. {Let $\C=(\overline{c}_{i})_{i=0}^{k}$  and $\C'=(\overline{c}_{i}')_{i=0}^{k}$ be two chains such that $\overline{c}_{i}\overline{c}_{i+1}\orq_{A} \overline{c}'_{i}\overline{c}'_{i+1}$ for all $0\leq i\leq k-1$.}
    \begin{itemize}
     \item {If $\overline{c}_{0}=\overline{c}'_{0}$, there is $g\in G_{A}$ such that $\nn_{\C'}g=\nn_{\C}$, and $g(\overline{c}'_{k})=\overline{c}_k$;}
        \item {If $\overline{c}_{k}=\overline{c}'_{k}$, there is $g\in G_{A}$ such that $g\nn_{\C}=\nn_{\C'}$ and $g(\overline{c}_{0})=\overline{c}_0'$.}
    \end{itemize}

  \end{corollary}
  \begin{proof}
  {We argue by induction on the length $k$ of the chains. We begin by proving the first statement. The base case of $k=0$ is trivial since then $\C=\C'$ and we may take $g=\mathrm{Id}$. So, for the inductive hypothesis, suppose the claim holds up to $k\geq 0$ and consider chains $\C=(\overline{c}_{i})_{i=0}^{k+1}$ and $\C'=(\overline{c}'_{i})_{i=0}^{k+1}$ with $\overline{c}_{0}=\overline{c}_{0}'$ and $\overline{c}_{i}\overline{c}_{i+1}\orq_{A} \overline{c}_{i}'\overline{c}'_{i+1}$ for all $0\leq i\leq k$. Denote by $\tilde{\C}$ and $\tilde{\C'}$ the chains obtained from $\C$ and $\C'$ respectively by removing the final term of the chain.}
  
  {By inductive hypothesis, there is $h\in G_A$ such that $\nn_{\tilde{\C'}} h=\nn_{\tilde{\C}}$ and $h(\overline{c}_k')=\overline{c}_k$. Using \cref{o:commutation} for the last equation below, we obtain:}
\begin{equation}\label{eq:chainarg1}
\nn_{\C}=\nn_{\tilde{\C}}\nn_{\overline{c}_{k+1}}=\nn_{\tilde{\C'}}h\nn_{\overline{c}_{k+1}}=\nn_{\tilde{\C'}}\nn_{h^{-1}(\overline{c}_{k+1})}h\;.  
\end{equation}
Since $h\in G_A$,
\[\overline{c}'_k h^{-1}(\overline{c}_{k+1})\equiv_A h(\overline{c}'_k) \overline{c}_{k+1}=\overline{c}_k \overline{c}_{k+1}\equiv_A \overline{c}'_k \overline{c}'_{k+1}\;.\]
{Thus, there is $\ell\in G_{A \overline{c}_k'}$ such that $\ell h^{-1} (\overline{c}_{k+1})=\overline{c}'_{k+1}$. Note that by \cref{o:commutation},} 
\[\nn_{h^{-1}(\overline{c}_{k+1})}=\ell^{-1}\nn_{\ell h^{-1}(\overline{c}_{k+1})} \ell\;.\]
{Furthermore, since $\ell\in G_{\overline{c}'_k}$, $\nn_{\overline{c}_k'} \ell^{-1}=\nn_{\overline{c}_k'}$. (This is just because for $s\in \nn_{\overline{c}_k'}$, we can write $s=s\ell\ell^{-1}$, where $s\ell\in\nn_{\overline{c}_k'}$ too.) Hence, using the two equations above in the right-hand side of \eqref{eq:chainarg1}, we get:}
\[\nn_{\C}=\nn_{\tilde{\C'}}\nn_{h^{-1}(\overline{c}_{k+1})}h=\nn_{\tilde{\C'}}\ell^{-1}\nn_{\ell h^{-1}(\overline{c}_{k+1})} \ell h=\nn_{\tilde{\C'}}\nn_{\overline{c}_{k+1}'} \ell h\;.\]
  {Taking $g:=\ell h$, we can thus conclude the argument for the first part.}

  {The argument for the second part is identical, by carrying out the induction in reverse. That is, now for the inductive step we will denote by $\tilde{\C}$ and $\tilde{\C'}$ the chains obtained from $\C$ and $\C'$ respectively by removing their first term. Then, taking $h\in G_A$ with $h \nn_{\tilde{C}}=\nn_{\tilde{C}'}$ and $h(\overline{c}_1)=\overline{c}_1'$, just as in \eqref{eq:chainarg1}, we obtain} 
  \begin{equation}
    \label{nn equality}	\nn_{\C'}=\nn_{\overline{c}_0'}\nn_{\widetilde{\C'}}=\nn_{\overline{c}_0'}h\nn_{\widetilde{\C}}=h\nn_{h^{-1}(\overline{c}_0')}\nn_{\widetilde{\C}}\;.  
    \end{equation}
{By an analogous argument to the above, we can choose $h'\in G_{A\overline{c}_1}$ such that $h h' (\overline{c}_0)=\overline{c}_0'$, and similarly deduce that} 
\begin{equation*}
    \nn_{\C'}
      = h\,h'\,\nn_{\overline{c}_0}\,\nn_{\widetilde{\C}}
      = h\,h'\,\nn_{\C}.
  \end{equation*}
   {Finally, setting $g := hh'$ yields the desired element of $G_A$.}
  
  \end{proof}

  We remind the reader of the following standard consequence of Neumann's Lemma, which is referred to frequently throughout the paper:  
  \begin{fact}
  	\label{f:neumann} {Let $G\acts\Omega$ with locally finite algebraic closure.} Let $C\in\fps{\Omega}$ be algebraically closed. Then, for all $A, B\in\fps{\Omega}$ there is $A'\in \fps{\Omega} $ such that $A'\equiv_C A$ and $A'\cap B= A\cap B\cap C$.  
  \end{fact}
  
  An easy but useful consequence of this is the following: 
  \begin{lemma}
 	\label{l:intersections} Let $G\acts\Omega$ with locally finite algebraic closure.
 	Then for any algebraically closed $A,C,D\in\fps{\Omega}$  there is some algebraically closed $E\in\fps{\Omega}$ and $g\in G$ such that $E\cap A=C\cap D\cap A$ and the following inclusion holds:
  \begin{equation*}
  	\nn_{E}g\subseteq \nn_{C}\nn_{D}.
  \end{equation*}
 \end{lemma}
 \begin{proof}
 	By \cref{f:neumann}, we can find some element $g\in \nn_{C}\cap G$ such that $g(D)\cap A=C\cap D\cap A$. By applying \cref{o:commutation}, we have $\nn_{g(D)}g=g\nn_{D}\subseteq\nn_{C}\nn_{D}$, so the desired conclusion holds with $E=g(D)$. 
 \end{proof}
 
  \begin{lemma}\label{skipping terms}

        Let $\C=(\overline{c}_i)_{i=0}^{2k}$ be a  chain of algebraically closed sets of length $2k$ for some $k>0$. Assume that we are in one of the following two cases:
	    \begin{enumerate}[label=(\Roman*)]
        \item \label{case1} {There is some algebraically closed $A\in\fps{\Omega}$ such that  $\overline{c}_{2i}\cap \overline{c}_{2i+1}\cap \overline{c}_{2i+2}=A$ for all $0\leq i\leq k-1$.}   
	    	\item \label{case2} 
            {There is some $G$-invariant algebraically closed $\Delta\subseteq\Omega$ such that $\overline{c}_{2i}\cap \overline{c}_{2i+1}\subseteq\Delta$ for all $0\leq i\leq k-1$.} 
            
	    \end{enumerate}
	    Then there is a chain of algebraically closed sets $\D=(\overline{d}_{i})_{i=0}^{k}$ of length $k$ such that for some $g\in G$, 
	    \begin{equation*}
	    	\nn_{\mathcal{D}}{g}\subseteq \nn_{\C}.
	    \end{equation*}

         Moreover: 
	     \begin{itemize}
	    	\item In case \ref{case1} we can ensure $\overline{d}_{i}\cap\overline{d}_{i+1}=A$ for all $0\leq i\leq k-1$, and $g\in G_{A}$;
            \item In case \ref{case2}  we can ensure $\overline{d}_{i}\cap \overline{d}_{i+1}\subseteq\Delta$ for all $0\leq i\leq k-1$.
	    \end{itemize}
  \end{lemma} 
\begin{proof}
{First consider case~\ref{case1}. 
Our aim is to construct a chain of algebraically closed sets
$\C'=(\overline{c}'_{i})_{i=0}^{2k}$ such that 
$\overline{c}'_{0}=\overline{c}_{0}$, 
for all $0\leq i\leq 2k-1$, $\overline{c}_{i}\overline{c}_{i+1}
\equiv_{A}
\overline{c}'_{i}\overline{c}'_{i+1}$, and for all $0\leq i\leq k-1$,
$\overline{c}'_{2i}\cap\overline{c}'_{2i+2}=A$.
Given such a chain of algebraically closed sets $\C'$, we know from the first
part of \cref{change chains} that there is some $g\in G_{A}$ such that $\nn_{\C'}g=
\nn_{\C}$. In particular, letting $\D=(\overline{c}'_{2i})_{i=0}^{k}$, by
Observation~\ref{obs: subchains} we get $\nn_\D\subseteq\nn_{\C'}$, and so $\nn_\D g\subseteq \nn_{\C}$, thus completing the proof.}

{We now construct $\mathcal{C}'$ by induction. The base case is trivial: set $\overline{c}'_{0}=\overline{c}_{0}$. Suppose $(\overline{c}'_{i})_{i=0}^{2l}$ has already been constructed, for some $0\leq l<k$, with all the desired properties; we build $\overline{c}'_{2l+1}$ and $\overline{c}'_{2l+2}$.}

{First choose any $\overline{c}'_{2l+1}$ with $\overline{c}'_{2l}\overline{c}'_{2l+1}
\equiv_{A}
\overline{c}_{2l}\overline{c}_{2l+1}$. We claim there is $\overline{c}'_{2l+2}$ with $\overline{c}'_{2l+1}\overline{c}'_{2l+2}
\equiv_{A}
\overline{c}_{2l+1}\overline{c}_{2l+2}$
and
$\overline{c}'_{2l+2}\cap \overline{c}'_{2l}=A$. To see this, take $\overline{c}''_{2l+2}$ such that $\overline{c}'_{2l}\overline{c}'_{2l+1}
\overline{c}''_{2l+2}
\equiv_A
\overline{c}_{2l}\overline{c}_{2l+1}\overline{c}_{2l+2}$.
From this, and our assumption on $\C$, we get $\overline{c}'_{2l}\cap\overline{c}'_{2l+1}
\cap\overline{c}''_{2l+2}=A$. 
Now apply \cref{f:neumann} to
$\overline{c}''_{2l+2}$,
$\overline{c}'_{2l+1}$, and
$\overline{c}'_{2l}$ to obtain 
$\overline{c}'_{2l+2}\equiv_{A \overline{c}'_{2l+1}}
\overline{c}''_{2l+2}$ with}
\[
\overline{c}'_{2l+2}\cap \overline{c}'_{2l}
=
\overline{c}''_{2l+2}\cap
\overline{c}'_{2l+1}\cap
\overline{c}'_{2l}
=A,
\]
{as required. This completes the construction of $\C'$ in case~\ref{case1}.}

{In case~\ref{case2}, we construct $\C'$ in exactly the same way, but
requiring only that $\overline{c}_{i}\overline{c}_{i+1}
\equiv
\overline{c}'_{i}\overline{c}'_{i+1}$. In the inductive step, having found $\overline{c}_{2l+2}''$ such that $\overline{c}'_{2l}\overline{c}'_{2l+1}
\overline{c}''_{2l+2}
\equiv
\overline{c}_{2l}\overline{c}_{2l+1}\overline{c}_{2l+2}$, we use \cref{f:neumann} to
find $\overline{c}'_{2l+2}\equiv_{\overline{c}'_{2l+1}}
\overline{c}''_{2l+2}$,}
\[
\overline{c}'_{2l}\cap\overline{c}'_{2l+2}=\overline{c}''_{2l+2}\cap
\overline{c}'_{2l+1}\cap
\overline{c}'_{2l}\subseteq\Delta\;,
\]
{which yields the desired containments for $\D$.}
\end{proof}

 \subsection{Rigidity and dispersivity}
   For the entirety of this subsection we work with a closed permutation group $G\acts \Omega$ with locally finite algebraic closure. {We will} consider the semigroup $S$ to be either $G$ or $\overline{G}$.  We let $\tau$ denote a semigroup topology on $S$ coarser than $\pw$.  The arguments are carried out in parallel for two cases: when $S=G$ (needed for \cref{thm:onebased}) and when $S=\overline{G}$ with \ref{sat} holding (needed for \cref{t:minimality endomorphism monoids}).  Throughout the lemmas below, a reader interested in only one of these results may accordingly restrict their attention to the case corresponding to their preferred choice of $S$. In particular, generalised pointwise topologies are only used for \cref{thm:onebased}.\\
  
We introduce two properties that an element $t \in S$ may have relative to a
subset $X \subseteq \Omega$, given a semigroup topology $\tau$ on $S$ strictly
coarser than $\pw$: \textbf{$\bm{\tau}$-rigidity} and \textbf{$\bm{\tau}$-dispersivity}. As
shown in Item~\eqref{obs6} of Observation~\ref{moving tau dispersiveness around} and
Lemma~\ref{l:moving out}, these two properties are complementary in a weak sense.
We prove the following.
\begin{itemize}
  \item If every $s \in S$ is $\tau$-rigid on all of $\Omega$, then $\tau$ is not
   {Hausdorff} (Lemma~\ref{rigidity over all implies non T2}).
  \item Conversely, if some element fails to be $\tau$-rigid on some
    $a' \in \Omega$, then there is an element that is $\tau$-dispersive on some
    $a \in \Omega$ (Lemma~\ref{l:moving out}), and hence an element that is
    $\tau$-dispersive on the entire $G$-orbit of $a$
    (Lemma~\ref{tau-dispersive everywhere}).
\end{itemize}
In the settings of both \cref{thm:onebased} and
\cref{t:minimality endomorphism monoids}, the existence of some $s \in S$ that is
$\tau$-dispersive on a suitable orbit $\Sigma$ allows us to deduce the existence
of some $t \in S$ satisfying a stronger property, which we call
\textbf{deep $\bm{\tau}$-dispersivity}.
The existence of a deeply $\tau$-dispersive element on $\Sigma$ is what lets us
show that every suitable non-empty $\tau$-open set $\U$ contains all elementary
embeddings sending a fixed tuple to the final term of a chain of finite closed
sets of a prescribed type---that is, a chain in which the initial term and the
joint type of each pair of consecutive links are predetermined. This richness of
$\tau$-open sets is used, together with further arguments, in the proofs of the
main theorems of the next two sections.

 {From now until Lemma~\ref{rigidity over all implies non T2}, we allow for $\Omega$ to be uncountable.}

 \begin{definition}
  Let $s_0 \in S$.
  \begin{itemize}
    \item We say that $s_0$ is $\bm{\tau}$\textbf{-rigid} on a point
      $a \in \Omega$ if there exists $\V \in \mathcal{N}_\tau(s_0)$ such that
      $s(a) \in \cl(s_0(a))$ for all $s \in \V$.

    \item We say that $s_0$ is $\bm{\tau}$\textbf{-dispersive} on $a\in\Omega$ if for
      every finite $B \subseteq \Omega$ and every
      $\V \in \mathcal{N}_\tau(s_0)$ there is $s \in \V$ with
      $a \notin s(B)$.
  \end{itemize}
  For $X \subseteq \Omega$, we say that $s_0$ is $\tau$-rigid
  (resp.\ $\tau$-dispersive) \textbf{on $\bm{X}$} if it is $\tau$-rigid
  (resp.\ $\tau$-dispersive) on every $a \in X$. {(Note that, vacuously, every element of $S$ is $\tau$-rigid on $\emptyset$.)}
  Finally, we say that $s_0$ is \textbf{deeply
  $\bm{\tau}$\textbf{-dispersive}} on $X \subseteq \Omega$ if for every $n > 1$
  there exist $s_1, \dots, s_n \in S$, each $\tau$-dispersive on $X$, such that
  $s_1 \cdots s_n = s_0$. 
\end{definition}

  Note that any $s_{0}\in S$ is
  $\tau$-dispersive on every point $a$ outside of $\im(s_{0})$. {Moreover, no element of $S$ is $\tau$-dispersive on any $a\in\acl(\emptyset)$.}
  
  {For the topology of pointwise convergence, it is easy to see that
every $s_0 \in S$ is $\pw$-rigid on every point of $\Omega$ and is not
$\pw$-dispersive on any element of $\im(s_0)$.}

  We collect a few easy facts about these notions.  
  \begin{observation}\label{moving tau dispersiveness around} Let $S\in\{G, \overline{G}\}$, and let $\tau$ be a semigroup topology on $S$ coarser than $\pw$. 
  	Let $t,s\in S$ and $a,a'\in\Omega$. Then  
     \begin{enumerate}
     	\item \label{obs1} if $ts$ is $\tau$-rigid on $a$, then $s$ is also $\tau$-rigid on $a$;
     	\item \label{obs2}
        if $a'=s(a)$ and $ts$ is $\tau$-rigid on $a$, then $t$ is $\tau$-rigid on $a'$;
     	\item \label{obs3} if $t$ is $\tau$-dispersive on $a$,  then $ts$ is also $\tau$-dispersive on $a$;
        
      \item \label{obs4} {if $a=t(a')$ and $s$ is $\tau$-dispersive on $a'$, then $ts$ is $\tau$-dispersive on $a$;}
       
      \item \label{obs5}
      if $s_0$ is $\tau$-dispersive on $a$, then, given $\V \in \mathcal{N}_\tau(s_0)$
      and $B$ as in the definition, the witness $s \in \V$ for $a \notin s(B)$ may be taken in $G$;
      \item \label{obs6} if $s$ is $\tau$-rigid on some $G$-invariant $\Delta\subseteq\Omega$ and $a\in\cl(\Delta)\cap\im(s)$, then $s$ is not $\tau$-dispersive on $a$.  
     \end{enumerate}    
  \end{observation}
  \begin{proof}
   \eqref{obs1}. Since $ts$ is $\tau$-rigid on $a$, there is $\V \in \mathcal{N}_\tau(ts)$   with $r(a) \in \cl(ts(a))$ for all $r \in \V$. As $\tau$ is a semigroup topology, the map $\rho_t \colon S \to S$, $u \mapsto tu$, is continuous, so $\U := \rho_t^{-1}(\V)$ is a $\tau$-neighbourhood of $s$. For every $u \in \U$ we have $tu(a) \in \cl(ts(a))$; since elements of $S$ preserve algebraic closure, $u(a) \in \cl(s(a))$. Thus $s$ is $\tau$-rigid on $a$.
  Statement \eqref{obs2} follows in the same way using the map $\lambda_s:S\to S$ given by $u\mapsto us$.

\eqref{obs3}. Suppose $t$ is $\tau$-dispersive on $a$. Fix a finite $B \subseteq \Omega$ and
  $\V \in \mathcal{N}_\tau(ts)$; then $\U := \lambda_s^{-1}(\V)$ is a $\tau$-neighbourhood of $t$. Applying $\tau$-dispersivity to the finite set $s(B)$ and $\U$, we obtain $u \in \U$ with  $a \notin u(s(B))$. Since $us \in \V$ and $a \notin us(B)$, the element $us$ witnesses ({for this particular neighbourhood}) that  $ts$ is $\tau$-dispersive on $a$. Item \eqref{obs4} follows in the same way using $\rho_t$.\\

     \eqref{obs5}. Let $s\in\V$ witness $\tau$-dispersivity on $a$ with respect to $B$. As $\tau$ is coarser than $\pw$,  there is {a finite partial map $\phi$}
     with $s \in \nn_\phi \subseteq \V$. Now, take $\psi$ to be the extension of $\phi$ by $s_{\upharpoonright B}$. We still have that $s\in \nn_{\psi}\subseteq\V$. Since every element of $S$ agrees on each finite set with some element of $G$, we have  $\nn_\psi \cap G \neq \emptyset$. Any $r \in \nn_\psi \cap G$ satisfies $a\not\in r(B)=s(B)$.
  So it is the required witness in $G$.\\
     
 \eqref{obs6}. We provide a finite $B$ and $\V \in \mathcal{N}_\tau(s)$ with $a \in r(B)$ for   every $r \in \V$, which shows that $s$ is \emph{not} $\tau$-dispersive on $a$. Write   $a = s(a')$, and, {using $G$-invariance of $\Delta$,} choose a finite $A \subseteq \Delta$ with $a' \in \acl(A)$. For each $c \in A$, $\tau$-rigidity of $s$ on $c$ yields $\V_c \in \mathcal{N}_\tau(s)$ with $r(c) \in \cl(s(c))$ for   all $r \in \V_c$. Set $\V_A := \bigcap_{c \in A} \V_c \in \mathcal{N}_\tau(s)$; then $s \in \V_A$ and $r(A) \subseteq \cl(s(A))$ for every $r \in \V_A$. Since the elements of $S$ {preserve} algebraic closure and $\acl$ is locally finite, $\cl(A)$ is finite and   $r(\cl(A)) = \cl(s(A)) = s(\cl(A))$ for all $r \in \V_A$. Put $B := \cl(A)$. As  $a' \in \acl(A) = B$, we get $a = s(a') \in s(\cl(A)) = r(B)$ for every $r \in \V_A$, as required.

  \end{proof}
\begin{remark}\label{rem:deepstuff}
By Item~\eqref{obs3} of Observation~\ref{moving tau dispersiveness around}, if an element $g \in G$   is $\tau$-dispersive on some $X \subseteq \Omega$, then it is deeply $\tau$-dispersive on $X$:  taking $t = g$ and $s = g^{-1}$ in \eqref{obs3}, $gg^{-1} = 1$ is $\tau$-dispersive   on $X$, so for every $n > 1$ the factorisation $g = 1 \cdots 1\, g$ expresses
  $g$ as a product of elements of $S$ that are $\tau$-dispersive on $X$.
\end{remark}

{The lemma below is a routine exercise in topology, which we will need for the subsequent lemma. For basic facts about filters in topology, the reader may consult~\cite[$\S$ 12]{willard2012general}:}
 \begin{lemma}\label{lem:filterlemma} {Let $S$ be a closed submonoid of $\Omega^\Omega$. Let $\mathfrak{F}$ be a filter on $S$ and suppose that for every $A\in\fps{\Omega}$ there is $F_A\in\mathfrak{F}$ such that $\{f_{\upharpoonright A} \ \vert \ f\in F_A\}$ is finite. Then, $\mathfrak{F}$ has a cluster point in $S$ (with respect to $\pw$).}    
   \end{lemma}
   \begin{proof} {By the ultrafilter lemma, the filter $\mathfrak{F}$ extends to an ultrafilter $\mathfrak{U}$. For each finite $A\subseteq \Omega$, let $r_A:\Omega^\Omega\to \Omega^A$ be the restriction map $f\mapsto f_{\upharpoonright A}$. Now, the pushforward $r_A(\mathfrak{U})$ of the ultrafilter $\mathfrak{U}$ with respect to $r_A$ is still an ultrafilter. Since by hypothesis, $r_A(\mathfrak{U})$ contains a finite set, $r_A(\mathfrak{U})$ is principal, and let $p_A\in\Omega^A$ be the principal element of $r_A(\mathfrak{U})$. Note that for $A, B\in\fps{\Omega}$, we must have that ${p_A}_{\upharpoonright A\cap B}={p_B}_{\upharpoonright A\cap B}$, and so in particular, there is $p\in\Omega^\Omega$ given by $\bigcup_{A\in\fps{\Omega}} p_A$. By construction, $p$ is a cluster point of $\mathfrak{U}$ (and so of $\mathfrak{F}$).  Since $S$ is closed in $\pw$ and each $p_A$ agrees on $A$ with some element of $S$ by construction, $p\in S$, as desired.}
   \end{proof}

  The following lemma generalises the argument of~\cite[Lemma 2.7]{ghaddln24} from groups to semigroups. We include proofs for both groups and semigroups since even in the former case we remove some unnecessary and hidden assumptions of~\cite[Lemma 2.7]{ghaddln24}.

    \begin{lemma}\label{rigidity over all implies non T2}  Let $G\acts \Omega$ be a closed permutation group  with locally finite algebraic closure and let $S$ be either $G$ or $\overline{G}$. %
    Then any semigroup topology $\tau$ on $S$ that is strictly coarser than $\pw$ and for which every $s\in S$ is $\tau$-rigid on $\Omega$ %
    {is not Hausdorff.}
   \end{lemma}
   
   \begin{proof} 
  We give the proof for $\overline{G}$ and indicate the adaptation to $G$ at the end. Since $\tau\subsetneq\pw$, there are $r\in S$ and $\V\in\mathcal{N}_{\pw}(r)$ such that for any {$\U\in\mathcal{N}_\tau(r)$}, $\U$ is not a subset of $\V$. Without loss of generality let $\V=\nn_{r_{\upharpoonright A_0}}$, {where $A_0$ is finite and algebraically closed.}  {By the rigidity assumption, for each $a \in \Omega$, there is $\V_a \in \mathcal{N}_\tau(r)$ with  $s(a) \in \cl(r(a))$ for all $s \in \V_a$; 
now, for $A\supseteq A_0$ finite and algebraically closed, 
 set $\V_A := \bigcap_{a \in A} \V_a \in \mathcal{N}_\tau(r)$.
  As the elements of $\overline{G}$ send algebraically closed sets to algebraically closed sets,   $r \in \V_A$ and every $s \in \V_A$ restricts to a bijection $s_{\upharpoonright A}: A\to r(A)$. Now, let }
   \[\V_A':=\V_A\setminus(\V_A\cap\V)\;.\]
  {Note that $\V'_A$ is non-empty since otherwise $\V_A\subseteq \V$, contradicting the choice of $\V$. Let $\X^{A_0}$ be the set of finite algebraically closed subsets of $\Omega$ which contain $A_0$. For $A\in\X^{A_0}$ and $\mathcal{U}\in\mathcal{N}_\tau(r)$, let $\mathcal{W}_{A, \mathcal{U}}:=\V'_A\cap \mathcal{U}$. Note that $\mathcal{W}_{A, \mathcal{U}}$  is non-empty: $\U\cap \V_A$ is a $\tau$-neighbourhood of $r$ and so it is not entirely contained in $\V$ by choice of $\V$.
   Now, 
   consider the family of sets $\mathfrak{V}:=\{\mathcal{W}_{A, \U}\ \vert \ A\in \X^{A_0}, \U\in\mathcal{N}_\tau(r)\}$. Note that $\mathfrak{V}$ is a prefilter: firstly, as we argued above, all of its sets are non-empty.} 
     Secondly,
   given $A, B\in \X^{A_0}$,
   \[\V'_{\acl(A\cup B)}\subseteq \V'_A\cap \V'_B\;.\]

  {This is because $\V_{\acl(A\cup B)}\subseteq \V_A\cap \V_B$, and we are then removing the same set on both sides of the containment. Hence, for $A, B\in\X^{A_0}$, $\mathcal{U}, \mathcal{U}'\in\mathcal{N}_\tau(r)$,}
\[\mathcal{W}_{\acl(A\cup B), \U\cap \U'}=\U\cap\U'\cap \V'_{\acl(A\cup B)}\subseteq \U\cap\U'\cap \V'_A\cap \V'_B=\mathcal{W}_{A, \U}\cap\mathcal{W}_{B, \U'}\;.\]
  
In particular, $\mathfrak{V}$ generates a filter $\mathfrak{F}$. 
{By construction, for each finite $A\in \fps{\Omega}$, we have that $\{f_{\upharpoonright A} \, \ \vert \, f\in\mathcal{W}_{\acl(A\cup A_0), S}\}$ is finite. So, by \cref{lem:filterlemma},} $\mathfrak{F}$ has a $\pw$-cluster point $p\in S$. Now, by construction, $p\neq r$. {Moreover, $p$ lies in the $\pw$-closure (and so the $\tau$-closure) of every $\tau$-neighbourhood of $r$. Hence, $p$ and $r$ cannot be separated by $\tau$-open sets
and so $\tau$ is not Hausdorff.}

   {The argument  for $G$ is essentially identical, except that in this case, 
   composing with inverses, we may work at the identity ($r = 1$), so that $\V = \nn_{A_0}$, and for each algebraically closed $A\in\fps{\Omega}$,  any $f\in\V_A$ is such that $f_{\upharpoonright A}\in\Sym(A)$. Consequently, the map $p$ obtained by the above construction is a permutation of $\Omega$. To see this, observe that we can run the same argument of Lemma~\ref{lem:filterlemma} restricting our attention to algebraically closed sets, and if for $A\in\fps{\Omega}$ algebraically closed, $p_A\in\Sym(A)$, then the cluster point $p$ belongs to $\Sym(\Omega)$. Since
  $G$ is
  closed in $\mathrm{Sym}(\Omega)$, we can conclude that  $p \in G$.}
   \end{proof}
 The fact below follows from a union of chains argument which is standard in model theory:
\begin{fact}\label{extensions of isomorphisms}
	Let $\M$ be a countable saturated structure, $J$ a countable index set, and $\{\phi_{j}\}_{j\in J}$ a countable collection of elementary isomorphisms between countable substructures of $\M$. Then, there is a countable saturated elementary extension
    $\M'$ of $\M$ such that each $\phi_{j}$  extends to an automorphism of $\M'$.  
\end{fact}
\begin{proof} 
This is a standard union of chains argument (see~\cite[Chapters 2, 4, and 5]{tent2012course} for several arguments of this kind). For each $j\in J$, let $A_j$ and $B_j$ be, respectively, the domain and codomain of $\phi_j$. Let $\M_0:=\M$ and fix an enumeration $\overline{m}_0$ of $\M_0$.

By compactness, {the L\"{o}wenheim-Skolem theorem, and the existence of countable saturated elementary extensions of countable models,} we can build a countable saturated elementary extension $\M_1\succeq \M_0$ such that $\M_1$ contains tuples of countable length $(\overline{b}^1_j)_{j\in J}$ such that for each $j\in J$, 
\[\mathrm{tp}(\overline{b}^1_j/B_j)=\mathrm{tp}(\overline{m}_0/A_j)\;.\]
Note in particular that this correspondence of types yields elementary isomorphisms $(\phi^1_j)_{j\in J}$ where $\phi^1_j: \overline{m}_{0}\to \overline{b}^1_j$ extends $\phi_j:A_j\to B_j$  for each $j\in J$. Next we build countable and saturated $\M_2\succeq\M_1$ containing tuples $(\overline{a}^2_j)_{j\in J}$ such that 
\[\mathrm{tp}(\overline{a}^2_j/\overline m_{0})=\mathrm{tp}(\overline m_{1}/\overline{b}^1_j)\;\]
for each $j\in J$, where $\overline{m}_1$ is an enumeration of $\mathcal{M}_1$.
Again, we get an elementary isomorphism  $\phi^2_j: \overline{a}^2_j\to \overline m_{1}$  for each $j\in J$.

We proceed inductively by back-and-forth building a sequence of countable elementary extensions (each of which is saturated)
\[\M_0\preceq \M_1\preceq \cdots \M_{2i}\preceq \M_{2i+1}\preceq\cdots\;,\]
with enumerations $\overline{m}_i$ for $i\geq 0$, together with sequences $(\overline{a}_j^{2i})_{i\in\mathbb{N}, j\in J}$ and $(\overline{b}_j^{2i+1})_{i\in\mathbb{N}, j\in J}$, and elementary isomorphisms %
$\phi^n_j$ such that, for each $j\in J$, {$i\geq1$}, $\phi_j^{2i+1}: \overline{m}_{2i}\to\overline{b}^{2i+1}_j$ and $\phi_j^{2i}:\overline{a}^{2i}_j\to \overline{m}_{2i-1}$, %
and $\phi_j^{n+1}$ extends
  $\phi_j^n$ for all $n$.
Put $\M':=\bigcup_{n\in\mathbb{N}}\M_n$, and  $\phi_j^{\infty}:=\bigcup_{n\in\mathbb{N}}\phi_j^n$ for each $j\in J$. Since $\M'\succeq\M$ is a countable union of countable saturated models, it is still countable and saturated. Moreover, by construction, $\phi_j^\infty$ is an automorphism of $\M'$ extending $\phi_j$ for each $j\in J$, as desired.
\end{proof}

 \begin{lemma}
    \label{l:moving out} 
    Let $S\in\{G, \overline{G}\}$ and let $\tau$ be a semigroup topology on $S$ coarser than $\pw$.  Fix $s_{0}\in S$ and assume either $s_{0}\in G$ or hypothesis \ref{sat} 
    holds.  Let $a'\in\Omega$ and suppose  $s_0$ is not $\tau$-rigid on $a'$. Then there exists some $s_{1}\in S$ which is $\tau$-dispersive  on some $a\in  (G\cdot a')\cap \im(s_{1})$. 
    Moreover, if $s_{0}\in G$, %
    we can take $s_{1}=s_{0}$ and $a=s_{0}(a')$.   
 \end{lemma}
 \begin{proof}

We prove the two cases separately, arguing by contrapositive.
 
   \textbf{Case 1:} $s_0 \in G$.
  \begin{subproof}[Proof of Case 1]
  Suppose $s_0$ is not $\tau$-dispersive on {$a\in\Omega$.}
  Then there are $B\in[\Omega]^{<\omega}$ and $\V\in\mathcal{N}_{\tau}(s_{0})$ such that $a\in s(B)$ for all $s\in\V$.
  Clearly, we may assume $a=s_0(a')$ for some $a'\in B$.   

   By Neumann's Lemma, there is $h \in G_{a'}$ with
  $B\cap h(B)\subseteq \cl(a')$. Set $\tilde{h} := h^{s_0^{-1}}$ and
  \[
    \W := \V \cap \lambda_h^{-1}\rho_{\tilde{h}^{-1}}^{-1}(\V)
        = \V \cap \{\, s \in S \mid \tilde{h}^{-1} s h \in \V \,\},
  \]
  where $\rho_t \colon u \mapsto tu$ and $\lambda_t \colon u \mapsto ut$. Since
  $\tilde{h} s_0 h^{-1} = s_0$ we have $s_0 \in \W$, so by continuity of multiplication
  $\W \in \mathcal{N}_\tau(s_0)$; one checks $\tilde{h} \in G_a$.

    For $s \in \W$ we have $a \in s(B)$ and, as $\tilde{h}^{-1} s h \in \V$,
  $a \in \tilde{h}^{-1} s\, h(B)$; applying $\tilde h \in G_a$ gives $a \in s(h(B))$. Hence
  \[
    a \in s(B) \cap s(h(B)) \subseteq s\bigl(B \cap h(B)\bigr) 
   \subseteq s(\cl(a')) = \cl(s(a')).
  \]
  	The equality above implies that $\cl(a)=\cl(s(a'))$, so that $s(a')\in\cl(a)$ by \cref{closure transitivity and local finiteness}. As this holds for all $s \in \W$,
  $s_0$ is $\tau$-rigid on $a'$.
  \end{subproof}

 	  \textbf{Case 2:} $s_0 \notin G$ and \ref{sat} holds.
    \begin{subproof}[Proof of Case 2]
   The previous argument fails: for
  $s_0 \in S \setminus G$ and $h \in G$, an element $\bar{h} \in G$ with
  $\bar{h} s_0 h^{-1} = s_0$ need not exist, since $s_0$ is not invertible. We instead use
  \cref{extensions of isomorphisms}. Let $\mathcal{F}$ be the countable collection of pairs
  $(a', B)$ with $a' \in \Omega$ and $B \subseteq \Omega$ finite. For each such pair,
  Neumann's Lemma gives %
  {$g_{a',B} \in G_{\acl(a')}$}
  with $g_{a',B}(B) \cap B \subseteq \cl(a')$; let
  \[
    \phi_{a',B} := \{\, \langle s_0(c), s_0(d) \rangle \mid \langle c, d \rangle \in g_{a',B} \,\}
  \]
  be the push-forward of $g_{a',B}$ by $s_0$. Applying \cref{extensions of isomorphisms} to
  this countable family, and using uniqueness up to isomorphism of the countable saturated
  model, we obtain $t \in S$ and, for each $(a', B) \in \mathcal{F}$, an element
  $\hat{g}_{a',B} \in G$ with $t s_0\, g_{a',B} = \hat{g}_{a',B}\, t s_0$. {Let $s_1:=t s_0$, $a:=s_0(a')$, and $a^* := s_1(a') = t(a)$. Note that by construction, $\hat{g}_{a',B}\in G_{a^*}$.}

  {Suppose that $s_1$ is not $\tau$-dispersive on $a^*$:}
  there are $\V \in \mathcal{N}_\tau(s_1)$ and finite $B$ with $a^* \in s(B)$ for all
  $s \in \V$. Put
  \[
    \W := \V \cap \hat{g}_{a',B}\, \V\, g_{a',B}^{-1}\, \mbox{ } , \qquad
    \W' := \{\, s \in S \mid t s \in \W \,\}.
  \]
  Since $\hat{g}_{a',B}^{-1} s_1 g_{a',B} = s_1 \in \V$ we have $\W \in \mathcal{N}_\tau(s_1)$,
  and $\W' \in \mathcal{N}_\tau(s_0)$. As in the first case,  for all $s \in \W'$
  \[a^* \in ts\bigl(g_{a',B}(B) \cap B\bigr) \subseteq ts\,(\cl(a')),\] and thus $a\in \cl(sa')$, 
  whence $s_{0}$ is $\tau$-rigid on $a'$, by \cref{closure transitivity and local finiteness} again.
\end{subproof}\end{proof}
  
  \begin{lemma}\label{tau-dispersive everywhere}  Let $S\in\{G, \overline{G}\}$ and $\tau$ be a semigroup topology on $S$ coarser than $\pw$. 
  	Assume that either $s_{0}\in G$ or hypothesis \ref{sat} holds. Let $s_{0}$ be $\tau$-dispersive on $a=s_{0}(a')$ where $a'\in \Omega$. Then there is $s_{1}\in S$ which is $\tau$-dispersive on $G\cdot a$. If $s_{0}\in G$, then we can conclude that any $g\in G$ is $\tau$-dispersive on $G\cdot a$. 
  \end{lemma}
  \begin{proof}
   By \cref{extensions of isomorphisms} we obtain $t \in S$ and $g_b, \hat{g}_b \in G$
  (for $b \in G \cdot a'$) with $g_b(a') = b$ and $\hat{g}_b\, t s_0\, g_b^{-1} = t s_0$.
  Set $s_1 := t s_0$. Since $s_0$ is $\tau$-dispersive on $a$, Item~\eqref{obs4} of
  Observation~\ref{moving tau dispersiveness around} gives that $s_1$ is $\tau$-dispersive
  on $t(a)$. We show $s_1$ is $\tau$-dispersive on every $c \in G \cdot t(a) = G \cdot a$.
  Fix such $c$. We may assume that $c\in\im(s_{1})$, since otherwise $s_{1}$ is automatically $\tau$-dispersive on $c$, so let $b := s_1^{-1}(c)$. From $\hat{g}_b s_1 g_b^{-1} = s_1$ and
  Item~\eqref{obs4}, $s_1$ is $\tau$-dispersive on
  \[
    \hat{g}_b\bigl(t(a)\bigr) = s_1\bigl(g_b(a')\bigr) = s_1(b) = c.
  \]
  As $c$ was arbitrary, $s_1$ is $\tau$-dispersive on $G \cdot a$.
  If $s_0 \in G$, the above holds with $t = 1$, so $s_0$ is $\tau$-dispersive on $G \cdot a$;
  by Item~\eqref{obs3} of Observation~\ref{moving tau dispersiveness around}, every element of $G$
  is then $\tau$-dispersive on $G \cdot a$.
  \end{proof}

In the last two lemmas of this section, when working with $G$ we mention generalised pointwise topologies. Note that the parts of the statements mentioning such topologies are only needed for \cref{thm:onebased} and are not required for \cref{t:minimality endomorphism monoids}.

  \begin{lemma}
  	\label{l:aux} Let $S$ be either $G$ or $\overline{G}$, where in the latter case \ref{sat} holds. 
  Let $\tau$ be a semigroup topology on $S$ coarser than $\pw$. Let $s\in S$ be deeply $\tau$-dispersive on some $G$-invariant set $\Sigma\subseteq \Omega$. 
    Then, for every fixed $D\in\X$ and every $\V\in\mathcal{N}_{\tau}(s)$ there are $E\in\X$ and $h\in G$ such that $\nn_{E}h\subseteq \V$ and %
    $E\cap D\cap\Sigma=\emptyset$.  
  	Moreover, if the set $\Delta:=\Omega\setminus\Sigma$ is algebraically closed, $s=1$ and %
    {$\pw^{\Delta,K}\subseteq\tau$}
    {for some compact $K\leq \gc(\Delta)$, $K\unlhd G(\Delta)$}, then we may assume that {$h\in S^{K}_{D\cap\Delta}$}. 
  \end{lemma}
  
  \begin{proof}
{Suppose first $D \cap \Sigma=\emptyset$. It suffices to find $E\in\X$ and $h\in G$ such that $\nn_E h\subseteq\V$. Since $\tau\subseteq\pw$ and $\V\in\mathcal{N}_\tau(s)$, there is some $E'\in\X$ with $\nn_{s\upharpoonright E'}\subseteq \V$. Choose $h\in G$ with $h_{\upharpoonright E'}=s_{\upharpoonright E'}$ and set $E:=h(E')$.   By
\cref{o:commutation},
\[
  \nn_{E}\, h = \nn_{h(E')}\, h = h\, \nn_{E'} \subseteq \nn_{s \upharpoonright E'}
    \subseteq \V,
\]
as required. }

{Assume now that $D \cap \Sigma \neq \emptyset$, and let $d_1, \dots, d_m$ enumerate it.
Since $s$ is deeply $\tau$-dispersive on $\Sigma$, we may write $s = s_1 \cdots s_{2m}$
with each $s_i$ $\tau$-dispersive on $\Sigma$; by continuity of multiplication, choose
$\W_i \in \mathcal{N}_\tau(s_i)$ with $\W_1 \cdots \W_{2m} \subseteq \V$. Moreover, since
$\tau$ is coarser than $\pw$, for each $1 \leq l \leq m$ there are $\alpha_l \in G$ and
$B_l \in \X$ with $\nn_{B_l} \alpha_l \subseteq \W_{2l}$.} 
 We construct, by
  induction on $l$ for $1 \le l \le m$, elements $h_l \in G$ and $E_l \in \X$ with
  \begin{enumerate}
    \item\label{induction assumption1} $E_l \cap D \cap \Sigma \subseteq \{d_{l+1}, \dots, d_m\}$;
    \item\label{induction assumption2} $\nn_{E_l} h_l \subseteq \W_1 \cdots \W_{2l}$.
  \end{enumerate}
  Taking $l = m$ and $E := E_m$ gives the conclusion. %

    For the base case $l=1$, %
     using $\tau$-dispersivity of $s_1$ on $\Sigma$, choose
  $h'_1 \in \W_1 \cap G$ with $d_1 \notin h'_1(B_1)$. By \cref{o:commutation},
  $\nn_{h'_1(B_1)}\, h'_1 \alpha_1 = h'_1\, \nn_{B_1} \alpha_1 \subseteq \W_1 \W_2$, so the
  conditions hold with $E_1 := h'_1(B_1)$ and $h_1 := h'_1 \alpha_1$.

For the induction step, suppose $h_i, E_i$ have been constructed for all $i \le l$ where $1 \le l < m$. Using
  $\tau$-dispersivity of $s_{2l+1}$ on $\Sigma$, choose $g_{l+1} \in G \cap \W_{2l+1}$ with
  $h_l^{-1}(d_{l+1}) \notin g_{l+1}(B_{l+1})$. By \cref{o:commutation},
  \begin{equation}\label{eq:containment}
    \nn_{E_l}\, \nn_{h_l g_{l+1}(B_{l+1})}\, h_l g_{l+1} \alpha_{l+1}
      = (\nn_{E_l} h_l)\, g_{l+1}\, (\nn_{B_{l+1}} \alpha_{l+1})
      \subseteq (\W_1 \cdots \W_{2l})\, \W_{2l+1}\, \W_{2l+2}.
  \end{equation}
  {Now apply \cref{l:intersections} with $h_l g_{l+1}$ in place of $g$, with $E_l$ in place
of $C$, with $h_l g_{l+1}(B_{l+1})$ in place of $D$, and with $D$ in place of $A$.}
  This yields $E_{l+1} \in \X$ and $h'_{l+1} \in G$ with
  $\nn_{E_{l+1}} h'_{l+1} \subseteq \nn_{E_l}\, \nn_{h_l g_{l+1}(B_{l+1})}$ and
  $E_{l+1}\cap D = E_l \cap h_l g_{l+1}(B_{l+1})\cap D$, so that
  \[
    D \cap \Sigma \cap E_{l+1}\subseteq D\cap\Sigma\cap E_{l}\cap h_{l}g_{l+1}(B_{l+1})\subseteq \{d_{l+2}, \dots, d_m\}
  \]
  using $d_{l+1} \notin h_l g_{l+1}(B_{l+1})$ and \eqref{induction assumption1}. By
  \eqref{eq:containment},
  $\nn_{E_{l+1}} h'_{l+1} h_l g_{l+1} \alpha_{l+1} \subseteq \W_1 \cdots \W_{2l+2}$, so set
  $h_{l+1} := h'_{l+1} h_l g_{l+1} \alpha_{l+1}$.
 
    {Finally, suppose that $s = 1$ and $\pw^{\Delta, K} \subseteq \tau$ for some $K$ as in the
statement. Set $\V' := \V \cap \nn^K_{D \cap \Delta}\in \mathcal{N}_\tau(1)$. Running the argument above
with $\V'$ in place of $\V$ produces $E \in \X$ and $h \in G$ with
$E \cap D \cap \Sigma = \emptyset$ and $\nn_E h \subseteq \V'$. As $1 \in \nn_E$, we get
$h = 1 \cdot h \in \nn_E h \subseteq \V'$, and hence $h \in \nn^K_{D \cap \Delta}$, as
required.}
  \end{proof}

   \begin{lemma}\label{building chains}
  Let $S$ be either $G$ or $\overline{G}$, where in the latter case \ref{sat} holds, and let
  $\tau$ be a semigroup topology on $S$ coarser than $\pw$. Let $\Delta \subseteq \Omega$ be
  $G$-invariant and algebraically closed, and assume $w$ is deeply $\tau$-dispersive on
  $\Omega \setminus \Delta$. Fix $k > 0$ and pick $\U \in \mathcal{N}_\tau(w^{2k+1})$. Then
  there are $h \in \U$ and a chain of algebraically closed sets
  $\E = (\overline{e}_i)_{i=0}^{k}$ with
  \begin{enumerate}
    \item\label{propc} $\overline{e}_i \cap \overline{e}_{i+1} \subseteq \Delta$ for all $0 \le i \le k-1$;
    \item\label{propd} $\nn_{\E} h \subseteq \U$.
  \end{enumerate}
  If, moreover, $w = 1$ and {$\pw^{\Delta, K}\subseteq \tau$ for some compact $K\leq \gc(\Delta)$, $K\unlhd G(\Delta)$}, then we may assume {that for some $A \in \X_\Delta$ we have $h\in G_{A}^{K}$ and $\overline{e}_i \cap \overline{e}_{i+1} = A$ for all
  $0 \le i \le k-1$.}
\end{lemma}
   \begin{proof}

 Pick $\V \in \mathcal{N}_\tau(w)$ with $\V^{2k+1} \subseteq \U$. Since $\tau \subseteq \pw$
  there are $D \in \X$ and $h_0 \in \V$ with $\nn_D h_0 \subseteq \V$; we may assume
  $D \cap \Delta \in \X$.
  If $w = 1$ we may take $h_0 = 1$; if furthermore $\tau^{\Delta, K} \subseteq \tau$,
  {up to shrinking $\V$, we may assume that $\V\subseteq S^{K}_{D\cap\Delta}$. Note that, in particular, $s \in \V$ preserves $[b]$ for every $b\in D \cap \Delta$.}

   	  By \cref{l:aux} there are $E\in\X$ and $h_{1}\in G$ such that 
      \begin{itemize}
   	     \item {$\nn_{E}h_{1}\subseteq\V$; and}
   	     \item $E\cap h_{0}^{-1}(D)\subseteq \Delta$.
      \end{itemize}
   {In the case of $w=1$ and $\pw^{\Delta, K}\subseteq\tau$, we take $h_1\in G^K_{D\cap\Delta}$ as in \cref{l:aux}.}
     By repeated application of \cref{o:commutation} we obtain the {containment}
     below: 
      \begin{equation*}
   	   \nn_{\C}(h_{0}h_{1})^{k}h_{0}=\nn_{D}h_{0}(\nn_{E}h_{1}\nn_{D}h_{0})^{k}\subseteq\V^{2k+1},
      \end{equation*}
      where $\C=(\overline{c}_{i})_{i=0}^{2k}$ is the chain given by
      \begin{equation*}
         \overline{c}_{2i}=(h_{0}h_{1})^{i}(D),\, \text{ for }\,0\leq i\leq k\quad\quad \text{and} \quad\quad \overline{c}_{2i+1}=(h_{0}h_{1})^{i}h_{0}(E)\, \text{ for }\, 0\leq i< k.
      \end{equation*} 
      Notice that the inclusion $h_{0}^{-1}(D)\cap E\subseteq\Delta$ implies
      $D\cap h_{0}(E)\subseteq\Delta$ and thus $\overline{c}_{2i}\cap \overline{c}_{2i+1}\subseteq\Delta$ for all %
       {$0\leq i< k$}.\\
      
    \textbf{Claim:} Under the assumption that $w=1$ and $\tau^{\Delta,{K}}\subseteq\tau$ the intersection $\overline{c}_{2{i}}\cap \overline{c}_{2i+1}\cap \overline{c}_{2i+2}$ is a fixed $A\subseteq D\cap\Delta$ for all $0\leq i\leq k-1$.
      \begin{subproof}[Proof of Claim]
      	{As mentioned above, in this case we have taken $h_0=1$ and $h_1\in G^K_{D\cap\Delta}$. Consider the set $A:=D\cap E\cap h_{1}(D)$. By construction (and choice of $h_0$), $D\cap E\subseteq\Delta$, and so $A\subseteq D\cap\Delta$ and is algebraically closed. In particular, since $h_1\in S^K_{D\cap\Delta}$, $h_1$ fixes $A$ setwise. Hence, for $0\leq i\leq k-1$ we have that}
      	\begin{equation*}
      		\overline{c}_{2i}\cap \overline{c}_{2i+1}\cap \overline{c}_{2i+2}=h_{1}^{i}(D\cap E\cap h_{1}(D))=D\cap E\cap h_{1}(D),
      	\end{equation*}
        {as desired.}
      \end{subproof}
       The 
       {final part of the lemma} follows from \cref{skipping terms} {with $h = g(h_0 h_1)^k h_0$, where $g$ is the element appearing in the conclusion of \cref{skipping terms}}. Note that in the case {where} $w=1$, {and} $\pw^{\Delta,K}\subseteq\tau$ by case~\ref{case1} of 
       \cref{skipping terms} 
       we can {take $g\in G_A$.}
       {Meanwhile, along the proof, %
       we have chosen $h_{0}=1$ and} $h_{1}\in S_{A}^{K}$, so in this case $h\in S_{A}^{K}$.
   \end{proof}

\section{Independence relations and their calculus}\label{sec: calcindependence}

    \begin{definition}    	\label{d:indipendence basic} Let $G\acts\Omega$ be a group acting on an infinite 
        set with 
        locally finite algebraic closure.  An \textbf{independence relation} is a $G$-invariant ternary relation $\ind$ on  $[\Omega]^{<\omega}$, where we write $A\ind_{C}B$ whenever $(A,B,C)$ are related in $\ind$ and read ``$A$ is independent from $B$ over $C$'', which  satisfies the following axioms:  for all  $A,B,C,D\in[\Omega]^{<\omega}$ (writing unions as concatenations, so that e.g.~$BC:=B\cup C$, and writing $a$ instead of $\{a\}$ for $a\in\Omega$), 

        \begin{itemize}
            \item (symmetry) if $A\ind_C B$, then $B\ind_C A$;
            \item (normality)  if $A\ind_{C}B$, then $A\ind_{C}BC$;
    		\item (monotonicity) if $A\ind_{C} BD$,  then $A\ind_{C} B$;
    		\item (transitivity) if $A\ind_C B$ and $A\ind_{BC} D$, then $A\ind_{C} BD$;
    		\item (full existence) there is $A'\equiv_C A$
        such that $A'\ind_C B$;
        \item (anti-reflexivity) for $a\in\Omega$, if $a\ind_C a$, then $a\in\acl(C)$.
        \end{itemize}
        We write $A\ind B$ to abbreviate $A\ind_{\emptyset} B$.
    \end{definition}

     \begin{remark} It is well-known that $\ind^{a}$ is an independence relation%
     ~\cite[Proposition 3.16]{d2023axiomatic}.  For any independence relation, $\ind$, we have that if $A\ind_{C} B$, then $A\ind^{a}_{C} B$~\cite[Remark 1.2.11]{d2023axiomatic}. Note that Neumann's Lemma is just full existence for algebraic independence. 
    \end{remark}

    \begin{fact}[{\cite[Proposition 3.1.3]{d2023axiomatic}}] Let $G\acts\Omega$ be a group action with locally finite algebraic closure. Let $\ind$ be an independence relation. Then, $\ind$ satisfies the following additional properties for all $A,B,C\in[\Omega]^{<\omega}$:
    \begin{itemize}
          \item \normalfont{(existence)}   %
           $A\ind_{B}B$;
        \item \normalfont{(strong closure)} $A\ind_{C} B$ if and only if $\acl(A)\ind_{\acl(C)}\acl(B)$.  
    \end{itemize} 
    \end{fact}

\begin{notation}
     We will apply   independence relations to finite tuples, %
     thus writing $\overline{a}\ind_{\overline{c}}\overline{b}$,
     whenever the relation holds for the sets of all entries of these  tuples. 
\end{notation}

\begin{definition} 
Let $G\acts\Omega$ have locally finite algebraic closure. Let $\ind$ be an  independence relation. We define the following additional property: for all $A,B,C,D\in \fps{\Omega}$ with $B\subseteq C\subseteq D\in[\Omega]^{<\omega}$, 
\begin{itemize}
    \item (base monotonicity)  if $A\ind_{B} D$, then $A\ind_C D$.
\end{itemize}
\end{definition}

{We note below that base monotonicity of $\ind^{a}$ is equivalent to modularity of algebraic closure. This fact is well-known under the additional assumption of algebraic closure forming a pregeometry (cf.~\cite[Proposition 1.2.17]{d2023axiomatic}):}

\begin{prop}\label{prop:modularity} {Let $G\acts\Omega$ with locally finite algebraic closure. Then, the following conditions are equivalent:}
\begin{itemize}
    \item $\ind^{a}$ satisfies base monotonicity;
    \item algebraic closure is modular.
\end{itemize}
\end{prop}
\begin{proof} {%
{Suppose first that $\ind^{a}$ satisfies base monotonicity, and let $A\subseteq B$ and $C$
be algebraically closed sets in $\fps{\Omega}$; we must show
$\acl(A\cup(C\cap B))=\acl(A\cup C)\cap B$.}
Let $D:=\acl(A\cup(C\cap B))$. Since $C\cap B\subseteq D\subseteq B$, by base monotonicity, $B\ind^a_{D} C$. So $B\cap\acl(C\cup D)=D$. Note that $\acl(C\cup D)=\acl(C\cup A)$. Hence, from the earlier identity, we obtain $D=\acl(A\cup C)\cap B$ as desired.}

{{Conversely, suppose that algebraic closure is modular. Let
$A,B,C,D\in\fps{\Omega}$ with $B\subseteq C\subseteq D$ and $A\ind^{a}_{B}D$; we must show
$A\ind^{a}_{C}D$.}
We may assume without loss of generality that the sets involved are algebraically closed. Note that
\[B=\acl(A\cup B)\cap D=\acl(B\cup(A\cap D))\;,\]
where the first equality follows from $A\ind^{a}_B D$, and the second by modularity (for $B\subseteq D$). But then, since $\acl(B\cup (A\cap D))=B$, we must have that $A\cap D\subseteq B$. Now, }
\[\acl(C\cup A)\cap D=\acl(C\cup (A\cap D))=C\;,\]
{where the first equality follows by modularity and the second from the fact that $A\cap D\subseteq B\subseteq C$. In particular, from the equation above, we have that $A\ind^{a}_C D$, and so that algebraic independence satisfies base monotonicity, as desired.} 
\end{proof}

   \begin{definition}
    \label{d:independence more} We define  additional properties an independence relation $\ind$ might have:
    \begin{itemize}
    \item (independent $3$-amalgamation)  Let   $A,B_{1},B_{2},C_{1},C_{2}\in[\Omega]^{<\omega}$ be algebraically closed and such that $B_{1}\ind_{A}B_{2}$, $A\subseteq B_{i}$ and $C_{i}\ind_{A}B_{i}$ for $i=1,2$ and $C_{1}\orq_{A}C_{2}$. Then there exists $D\in[\Omega]^{<\omega}$ such that
    $D\orq_{B_{i}} C_{i}$ for $i=1,2$ and $D\ind_{A}B_{1}B_{2}$.
    \item (stationarity)
    Let $A, B_1, B_2, C\in[\Omega]^{<\omega}$ be  algebraically closed
    and such that  $B_{1}\orq_{A} B_{2}$ and $B_i\ind_A C$ for $i=1,2$. Then  $B_{1}\orq_{AC} B_{2}$.
    \end{itemize}
    \end{definition}

\section{Minimal semigroup topologies on automorphism groups of simple structures}
 \label{section groups}

      In this section, we prove \cref{thm:onebased}, generalising previous work of~\cite{ghaddln24}. The reader interested in \cref{t:minimality endomorphism monoids} may skip this section with no loss of understanding, though it may serve as a warm-up for that result. 
    \subsection{More on independence relations}

  \begin{definition}
  \label{d:narrow} Given $k\geq 1$,  an \textbf{independent chain} of length  $k$ is a chain of algebraically closed sets $\C=(\overline{c}_{0},\dots, \overline{c}_{k})$ such that for all $1\leq i\leq k-1$ we have 
  \begin{equation*}
  	\overline{c}_{0}\dots \overline{c}_{i-1}\id_{\overline{c}_{i}}\overline{c}_{i+1}\dots \overline{c}_{k}.
  \end{equation*}
  For $A \in \X$, we say an independent chain is \textbf{over $\bm{A}$} if
  $\overline{c}_i \cap \overline{c}_{i+1} = A$ for all
  $0 \leq i \leq k-1$. We impose no further conditions on how $A$ sits inside
  the tuples.
  \end{definition}	
 
  The following is a standard observation. 
  \begin{observation}
	\label{o:chains} For $(\overline{c}_{0},\overline{c}_{1},\dots, \overline{c}_{k})$ where $k\geq 2$ to be an independent chain, it suffices to assume 
  $\overline{c}_0 \cdots \overline{c}_{i-1} \id_{\overline{c}_i}
  \overline{c}_{i+1}$ for all $1 \leq i \leq k-1$, or 
  $\overline{c}_i \id_{\overline{c}_{i+1}}
  \overline{c}_{i+2} \cdots \overline{c}_k$ for all $0 \leq i \leq k-2$.

	It follows that for every chain $(\overline{c}_{i})_{i=0}^{k}$ (over $A$) there is an independent chain $(\overline{c}'_{i})_{i=0}^{k}$ with $\overline{c}_{i}\overline{c}_{i+1}\orq \overline{c}'_{i}\overline{c}'_{i+1}$  (respectively $\overline{c}_{i}\overline{c}_{i+1}\orq_{A} \overline{c}'_{i}\overline{c}'_{i+1}$) for all $0\leq i\leq k-1$, and $\overline{c}_{k}=\overline{c}'_{k}$; alternatively, one may require $\overline{c}_{0}=\overline{c}'_{0}$ instead. 
  \end{observation}

 \begin{definition} Let $\ind$ be an independence relation on $\Omega$ and $k\geq 1$. We say that $\ind$ is $\bm{k}$\textbf{-narrow} if:  
	\begin{itemize}
		\item ($k$-narrowness) for every $A\in\X$ and every independent chain $\C=(\overline{c}_{0},\dots, \overline{c}_{k})$ over $A$, we have $\overline{c}_{0}\ind_{A} \overline{c}_{k}$.
	\end{itemize}

\end{definition}

\begin{remark}\label{rem:onebased}
 Note that $1$-narrowness just says that for all $A, B\in[\Omega]^{<\omega}$, $A\ind_{\acl(A)\cap\acl(B)} B$. We call this property \textbf{strong one-basedness}. From the observation at the end of \cref{def:alg ind},  it is easy to deduce that an independence relation is strongly one-based if and only if it is $\ind^{a}$. The definition of $k$-narrowness is inspired by the definition of the failure of $k$-ampleness in the sense of~\cite{palacin2013ample}.
\end{remark}

 {Certain simple $\omega$-categorical Hrushovski constructions
 that are not one-based} provide
  examples with a $3$-narrow notion of independence satisfying independent
  $3$-amalgamation~\cite[Section~3E]{ghaddln24}.

  The usefulness of independence relations in our context stems from the following result. 
  \begin{lemma}
   	\label{l:reaching over the empty set}  Let $G\acts\Omega$ have locally finite algebraic closure and $\ind$ be an independence relation. Let $\C = (\overline{c}_i)_{i=0}^{L}$ be a chain
  of algebraically closed sets such that, for some $A \in \X$,
    \begin{itemize}
    	\item $A\subseteq \overline{c}_{i}$ for all $0\leq i\leq L$, and
    	\item $\overline{c}_{i}\ind_{A} \overline{c}_{i+1}$ for all $0\leq i<L$.
    \end{itemize} 
    Assume moreover that one of the following holds:
    \begin{enumerate}[label=(\Alph*)]
   		\item \label{sta case} $L=2$ and $\ind$ satisfies stationarity.
   		\item \label{independent amalgamation case} $L=3$ and $\ind$ satisfies  independent $3$-amalgamation. 
   	\end{enumerate}   
    Then $G_{\mathcal{C}}=G_{A}$.
  \end{lemma} 
  \begin{proof}
  	We wish to apply \cref{o:chain}, according to which $G_{\mathcal{C}}$ equals the collection of all elements $s$ of %
    {$G$}     for which there is a chain of algebraically closed sets $(\overline{c}_{i}')_{i=0}^{L}$ where $\overline{c}'_{0}=\overline{c}_{0}$, $\overline{c}'_{L}=s(\overline{c}_{L})$, and 
  	$\overline{c}_{i}\overline{c}_{i+1}\orq_{A}\overline{c}'_{i}\overline{c}'_{i+1}$ for all $0\leq i< L$. 
  	
  	Fix some $s\in G_{A}$. To begin with, notice that by full existence we can find some $\overline{c}'_{1}\in\X$ such that $\overline{c}'_{1}\orq_{\overline{c}_{0}}\overline{c}_{1}$ and 
  	$s(\overline{c}_{L})\ind_{\overline{c}_{0}} \overline{c}'_{1}$. Since by assumption $\overline{c}_{0}\ind_{A} \overline{c}_{1}$, by invariance and the assumption that $s\in G_{A}$ we also have $\overline{c}_{0}\ind_{A} \overline{c}'_{1}$. An application of transitivity yields $\overline{c}_{0}s(\overline{c}_{L})\ind_{A} \overline{c}_{1}'$, from which one can derive, using monotonicity and symmetry, the following relation
  	\begin{equation}\label{first indepenendence}
  		\overline{c}_{1}'\ind_{A} s(\overline{c}_{L}).
  	\end{equation}
  	
  	\textbf{Case \ref{sta case}.} In this case, we claim that $(\overline{c}_{0},\overline{c}'_{1},s(\overline{c}_{2}))$ is a suitable chain between $\overline{c}_{0}$ and $s(\overline{c}_{2})$. Indeed, by assumption
  	\begin{equation*}
  	  \overline{c}_{1}\ind_{A} \overline{c}_{2},\quad\quad \overline{c}_{1}'\orq_{A}\overline{c}_{1},\quad\quad s(\overline{c}_{2})\orq_{A}\overline{c}_{2},
  	\end{equation*}
  	which together with \eqref{first indepenendence} allows one to apply stationarity to conclude $\overline{c}'_{1}s(\overline{c}_{2})\orq_{A}\overline{c}_{1}\overline{c}_{2}$.   
  	
  	\textbf{Case \ref{independent amalgamation case}.} Here we take $L=3$. We claim we can find $\overline{c}'_{2}$ with $\overline{c}_{1}'\overline{c}_{2}'\orq_{A}\overline{c}_{1}\overline{c}_{2}$ and $\overline{c}_{2}'\overline{c}'_{3}\orq_{A}\overline{c}_{2}\overline{c}_{3}$, as a result of an application of independent $3$-amalgamation, which means $(\overline{c}_{0},\overline{c}'_{1},\overline{c}'_{2},s(\overline{c}_{3}))$ is the needed chain.
  	The corresponding assumptions are listed below: 
  	\begin{itemize}
  		\item $\overline{c}_{1}'\orq_{A}\overline{c}_{1}$ by construction, since $A\subseteq \overline{c}_{0}$;
  		\item $s(\overline{c}_{3})\orq_{A} \overline{c}_{3}$, since $s\in G_{A}$;
  		\item $\overline{c}_{2}\ind_{A} \overline{c}_{1}$, by assumption and symmetry;
  		\item $\overline{c}_{2}\ind_{A} \overline{c}_{3}$, by assumption;
  		\item $\overline{c}'_{1}\ind_{A} s(\overline{c}_{3})$, which is just \eqref{first indepenendence} with $L=3$.
  	\end{itemize}	
  \end{proof}
  
  \begin{corollary}
  	\label{c:reachability} In the setting of \cref{l:reaching over the empty set} (with the same $L\in\{2,3\}$), assume furthermore that $\ind$ is $k$-narrow for some $k\geq 1$. Then for every chain $\mathcal{C}=(\overline{c}_{i})_{i=0}^{kL}$ with $\overline{c}_{i}\cap \overline{c}_{i+1}=A$ for some fixed $A\in\X$ and all $0\leq i\leq kL-1$, we have $G_{A}\subseteq G_{\mathcal{C}}$. 
  \end{corollary}
  \begin{proof}
  	By \cref{o:chains}  there is an independent chain $\mathcal{C}'=(\overline{c}'_{i})_{i=0}^{kL}$ over $A$ with $\overline{c}'_{i}\overline{c}'_{i+1}\orq_{A}\overline{c}_{i}\overline{c}_{i+1}$ for all $0\leq i<kL$ and $\overline{c}'_{kL}=\overline{c}_{kL}$. 
    By \cref{change chains}, $G_{\mathcal{C}'}=gG_{\mathcal{C}}$ for some $g\in G_{A}$. It thus suffices to prove that $G_{A}\subseteq G_{\mathcal{C}'}$. 
  	 
  	Since $\ind$ is $k$-narrow, we have that $\overline{c}'_{ki}\ind_{A} \overline{c}'_{k(i+1)}$ for all $0\leq i\leq L-1$. The conclusion can now be obtained by applying \cref{l:reaching over the empty set} to the chain $(\overline{c}'_{ik})_{i=0}^{L}$, as $G_{(\overline{c}'_{ik})_{i=0}^{L}}\subseteq G_{\mathcal{C}}$. 
  \end{proof}

    \subsection{Semigroup topologies on automorphism groups of simple structures}

\subsubsection{More on generalised pointwise topologies}
Before delving into the proof of \cref{mainthm: c}, we need to prove some basic facts about generalised pointwise topologies (including some mentioned in Section~\ref{subsub:localised}). In this section, $G$ is a closed group of permutations of an infinite set $\Omega$ with locally finite algebraic closure. 

\begin{definition}
    {For $A\in\X$ write $G[A]$ for the image of $G_{\{A\}}$ in $\mathrm{Sym}(A)$ under the restriction map. Let also} 
	\begin{equation*}
		\gc[A]=\{\,g\in G[A]\,|\,\forall a\in A\,\,g(a)\in[a]\,\}.
	\end{equation*}
\end{definition}

\begin{observation}\label{ob:gcdelta} {The group $\gc(\Delta)$ is a profinite normal subgroup of $G(\Delta)$.}
\end{observation}
\begin{proof}
    {Consider the collection of groups $\gc[A]$ for $A\in\X_{\Delta}$, partially ordered with respect to inclusion between the sets in $\X_{\Delta}$ together with the restriction maps $r_{A}^{B}:\gc[B]\to \gc[A]$, for $A\subseteq B$ from $\X_{\Delta}$. Consider the inverse limit $G^\star$ of the system}
       \[(\,(\gc[A],{r_{A}^{B}})\,|\,A,B\in\X_{\Delta},A\subseteq B\,)\; .
   \]

   { An element of $G^{\star}$ is a tuple $\overline{g}=(g_{A})_{A\in\X_{\Delta}}$ with
  $r^{B}_{A}(g_{B})=g_{A}$ whenever $A\subseteq B$. Such a tuple determines a well-defined
  and faithful action of $G^{\star}$ on $\Delta$ by setting
  $\overline{g}\cdot a:=g_{A}\cdot a$ for any $A\in\X_{\Delta}$ with $a\in A$. Conversely, every
  $g\in\gc(\Delta)$ preserves each set $A\in\X_{\Delta}$ setwise and restricts on it to an
  element of $\gc[A]$, and these restrictions are compatible; the resulting map
  $\gc(\Delta)\to G^{\star}$ is an isomorphism. Since each $\gc[A]$ is finite,
  $\gc(\Delta)$ is profinite.
  It remains to show that $\gc(\Delta)$ is normal in $G(\Delta)$. By $G$-invariance of
  $\Delta$, the action of $G(\Delta)$ on $\Delta$ preserves the equivalence relation $\sim$,
  and hence induces an action on the set $\Delta/{\sim}$ of $\sim$-classes. Thus there is a
  group homomorphism $\pi\colon G(\Delta)\to\Sym(\Delta/{\sim})$ given by
  $\pi(g)([a])=[g(a)]$, and by definition $\gc(\Delta)=\ker\pi$. Being a kernel,
  $\gc(\Delta)$ is normal in $G(\Delta)$.}
\end{proof}

The following is immediate from the definition of a generalised pointwise topology.
   
\begin{observation} \label{largest Delta}
    Let $\Delta$ be a $G$-invariant {subset of $\Omega$}. Then $\pw^{\Delta,{\gc(\Delta)}}$ is %
    the coarsest  semigroup topology $\tau$ such that every $g\in G$ is $\tau$-rigid on $\Delta$. Moreover, given any semigroup topology $\tau$ on $G$ coarser than $\pw$, let $\Delta_{\tau}:=  \{\, a \in \Omega \mid \text{every } g \in G \text{ is }
  \tau\text{-rigid on } a \,\} $. %
  Then: 
    \begin{enumerate}[label=(\roman*)]
     \item \label{Delta c} $\Delta_{\tau}$ is $G$-invariant;
         \item \label{Delta b} $\Delta_{\tau}$ agrees with the set of elements $a\in\Omega$ such that some $g\in G$ is $\tau$-rigid on $a$;
     \item \label{Delta a} $\Delta_{\tau}$ is the largest {$G$-invariant} $\Delta\subseteq \Omega$ 
     such that $\pw^{\Delta, {\gc(\Delta)}}\subseteq \tau$.

    \end{enumerate}
   \end{observation}
   
  \begin{proof}  
  	Clearly, $\pw^{\Delta, {\gc(\Delta)}}$ is a {semigroup}
    topology in which every element is $\tau$-rigid on $\Delta$.
  	Since $G$ is a group, {by Item~\eqref{obs1} of \cref{moving tau dispersiveness around}, every $g\in G$ is $\tau$-rigid on $\Delta$ if and only if $1$ is. So,} for the first statement it is sufficient to prove that $\pw^{\Delta, {\gc(\Delta)}}$ is coarsest among semigroup topologies such that $1$ is $\tau$-rigid on $\Delta$. Let $\tau\subseteq\pw$ be such that {every $g\in G$}
    is $\tau$-{rigid}
    on $\Delta$ and pick $A\in\X_\Delta$. Consider the basic $\pw^{\Delta,{\gc(\Delta)}}$ open neighbourhood of the identity 
    \[{G_A^{\gc(\Delta)}=\{g\in G\;\vert \; g_{\upharpoonright A}\in \gc(\Delta)\upharpoonright A\}\;,}\]
    {where $A\in\X_\Delta$ and $\gc(\Delta)\upharpoonright A$ is the set of restrictions of permutations in $\gc(\Delta)$ to $A$. We shall make use of the description of $\gc(\Delta)$ as the inverse limit $G^\star$ provided in \cref{ob:gcdelta}. For $B\supseteq A$ from $\X_\Delta$, the finite sets $r_A^B(\gc[B])\subseteq\gc[A]$ form a decreasing directed family. In particular, since $\gc[A]$ is finite, by the inverse limit description of $\gc(\Delta)$, there is some $B\supseteq A$ from $\X_\Delta$ such that} 
    \[\gc(\Delta)\upharpoonright A=r_A^B(\gc[B])\;.\]
   {Now, by $\tau$-rigidity on $B\subseteq\Delta$, there is $\V\in\mathcal{N}_\tau(1)$ such that for all $b\in B$ and $g\in\V$, $g(b)\in[b]$. Hence,  for $g\in \V$, $g_{\upharpoonright B}\in\gc[B]$, and so $g_{\upharpoonright A}\in\gc(\Delta)\upharpoonright A$. In particular, $\V\subseteq G_A^{\gc(\Delta)}$, as desired. Since for any $\pw^{\Delta, \gc(\Delta)}$-open neighbourhood of $1$, we can find some $\V\in\mathcal{N}_\tau(1)$ contained in it, we can deduce that $\pw^{\Delta, \gc(\Delta)}\subseteq\tau$.}\\

  {We now move to proving the various items for $\Delta_\tau$. We begin by proving Item~\ref{Delta c}: take $a\in\Delta_\tau$ and consider $a':=h(a)$ for $h\in G$. For $g\in G$, $g h$ is $\tau$-rigid on $a$ by definition of $\Delta_\tau$. Hence, by Item~\eqref{obs2} of \cref{moving tau dispersiveness around} $g$ is $\tau$-rigid on $a'$. So, every $g\in G$ is $\tau$-rigid on $a'$, and so $\Delta_\tau$ is $G$-invariant. Item~\ref{Delta b} simply follows from the fact that $G$ is a group and Item~\eqref{obs1} of \cref{moving tau dispersiveness around}, as we pointed out earlier. Finally, we prove Item~\ref{Delta a}. Note that, by definition of $\tau$-rigidity, if $\tau'\subseteq \tau''$ and $g\in G$ is $\tau'$-rigid on $a$, then it is also $\tau''$-rigid on $a$. Now, for $G$-invariant $\Delta$, every $g\in G$ is $\pw^{\Delta, \gc(\Delta)}$-rigid on $\Delta$ by our first statement, and so $\tau$-rigid on $\Delta$. But then, by definition of $\Delta_\tau$,  $\Delta\subseteq\Delta_\tau$, as desired.} 
  \end{proof}

{Below, we will need a couple of facts about paratopological groups (see~\cite[Section 3]{tkachenko2013paratopological}):}

  \begin{fact}[{\cite[Lemma 4]{RavskyHClosed}}]\label{fact: paratopological quotient} Let $G$ be a paratopological group and $H\unlhd G$. If both $H$ and $G/H$ are topological groups with, respectively, the induced and quotient topology, then $G$ is a topological group.
  \end{fact}

  \begin{fact}[{\cite[Proposition 3.2]{tkachenko2013paratopological}}]\label{fact: paratopological compact} A compact paratopological group is a topological group.  
  \end{fact}

  \begin{corollary} \label{cor:T0} {Let $G\acts\Omega$ be a closed permutation group with locally finite algebraic closure. Let $\tau$ be a semigroup topology on $G$ which is strictly coarser than $\pw$ and such that every $g\in G$ is $\tau$-rigid on $\Omega$. Then, $\tau$ is a group topology and hence, not $T_0$.}
\end{corollary}
\begin{proof} {By Lemma~\ref{rigidity over all implies non T2} and the fact that, for a topological group, $T_0$ and Hausdorffness are equivalent, we only need to prove that $\tau$ is a group topology. In the setting of \cref{largest Delta}, $\Delta_\tau=\Omega$, and so $\pw^{\Omega, \gc(\Omega)}\subseteq\tau$. Note that by definition, $G(\Omega)=G$, and so, by \cref{ob:gcdelta}, $\gc(\Omega)$ is a compact normal subgroup of $G$. We will proceed by showing that both $\gc(\Omega)$ and $G/\gc(\Omega)$ are topological groups. This implies that $G$ is a topological group by \cref{fact: paratopological quotient}.}\\

{Firstly, by \cref{ob:gcdelta}, $\gc(\Omega)$ with subspace topology induced by $\pw$ is a compact subgroup of $G=G(\Omega)$. In particular, since $\tau\subseteq\pw$, $\gc(\Omega)$ with the subspace topology induced by $\tau$ is also compact, and so a topological group by \cref{fact: paratopological compact}.}\\

{Now, we need to prove that $G/\gc(\Omega)$ with the quotient topology $(\tau)/{\gc(\Omega)}$ is also a topological group. Note that in our setting, by definition, for $A\in\X$, $G_A^{\gc(\Omega)}=G_A \gc(\Omega)$. So, $\pw^{\Omega, \gc(\Omega)}$ is the pullback of the quotient topology induced by $\pw$ on $G/\gc(\Omega)$. Compare $(\tau)/{\gc(\Omega)}$ to the quotient topology $(\pw)/{\gc(\Omega)}$ induced by $\pw$. We have that
\[(\tau)/\gc(\Omega)\subseteq (\pw)/\gc(\Omega)=\left(\pw^{\Omega, \gc(\Omega)}\right)/{\gc(\Omega)}\subseteq (\tau)/{\gc(\Omega)}\;,\]
 where the first containment is given by the fact that $\tau\subseteq\pw$, the equality was observed above, and the second containment follows since $\pw^{\Omega, \gc(\Omega)}\subseteq\tau$. Hence, $(\tau)/\gc(\Omega)=(\pw)/\gc(\Omega)$. Hence, $G/\gc(\Omega)$ with $(\tau)/\gc(\Omega)$ is a topological group. This concludes our proof.}
\end{proof}

  \begin{lemma}\label{minimal K}
  	\label{Omega0 and K0} Let $G\acts\Omega$ be a closed permutation group with locally finite algebraic closure. Let $\tau$ be a semigroup topology on $G$ coarser than $\pw$, and $\Delta_{\tau}$ the $G$-invariant set given by \cref{largest Delta}. 
    Let $K_{\tau}$ be the set of  {$\alpha\in G(\Delta_{\tau})$}
    such that for every $\V\in\mathcal{N}_{\tau}(1)$ and every $A\in\X_{\Delta_{\tau}}$, there exists $h\in\V$ with $h_{\restriction A}=\alpha_{\restriction A}$.  
   Then $K_{\tau}$ is closed, and it is the smallest $K$ such that $K\unlhd G(\Delta_{\tau})$, $K\leq {\gc(\Delta_{\tau})}$, and $\pw^{\Delta_{\tau},K}\subseteq\tau$.
  \end{lemma}

  \begin{proof}    
  	{ 
       For any closed $K\leq {G(\Delta_\tau)}$ and $A\in\X_{\Delta_{\tau}}$ define 
      \begin{align*}
        K_{\upharpoonright A}:=&\{ \,\phi\in G[A]\,|\,\exists g\in K,\;g_{\restriction A}=\phi \}.
      \end{align*}    
      The fact that $K_\tau$ is closed is immediate from {its definition}. Moreover, every $g\in G$ is $\tau$-rigid on $\Delta_\tau$ by definition. So, {for each $a\in \Delta_\tau$,} there is $\V\in\mathcal{N}_\tau(1)$ such that for all $g\in\V$, $g(a)\in[a]$, and so, by definition of $K_\tau$, we have that it is contained in $\gc(\Delta_\tau)$. 
      
      We must show $K_\tau$ is a subgroup of $\gc(\Delta_\tau)$, normal in $G(\Delta_\tau)$. 
      For the first, note that continuity of multiplication at the identity easily implies that $(K_{\tau})_{\upharpoonright A}$ is a subsemigroup of $G[A]$: given $\phi, \psi \in (K_\tau)_{\upharpoonright A}$ and an arbitrary $\V \in \mathcal{N}_\tau(1)$, choose
      $\W \in \mathcal{N}_\tau(1)$ with $\W^2 \subseteq \V$ and $s, t \in \W$  with $s_{\restriction A} = \phi$ and $t_{\restriction A} = \psi$; then $st \in \V$ and $(st)_{\restriction A} = \phi \psi$. Since a subsemigroup of a finite group is necessarily a subgroup,  $(K_{\tau})_{\upharpoonright A}$ is a subgroup of $G[A]$ for every $A\in \X_{\Delta_{\tau}}$, and hence $K_{\tau}\leq {\gc(\Delta_{\tau})}$.
    Finally, invariance of   $\mathcal{N}_\tau(1)$ under conjugation by any fixed element of $G$
   shows $g\, (K_\tau)_{\upharpoonright A}\, g^{-1} = (K_\tau)_{\upharpoonright g(A)}$, from which we can deduce $K_\tau\unlhd G(\Delta_\tau)$.
   }
 
   { 
      We now show that $\pw^{\Delta_{\tau},K_{\tau}}\subseteq \tau$.
      Take 
     $A\in\X_{\Delta_{\tau}}$ and $\alpha\in \gc[A]\setminus (K_{\tau})_{\upharpoonright A}$. By definition of $K_\tau$, there is $\V_\alpha\in\mathcal{N}_\tau(1)$ such that no $g\in\V_\alpha$ agrees with $\alpha$ on $A$. By $\tau$-rigidity on $A$ further take $\V_0\in\mathcal{N}_\tau(1)$ such that for $g\in\V_0$, we have that for all $a\in A$, $g(a)\in[a]$.  Since $\gc[A]$ is finite,} 
     \[\V:={\V_0\cap}\bigcap_{\alpha\in{\gc[A]\setminus (K_{\tau})_{\upharpoonright A}}} \V_\alpha\in\mathcal{N}_\tau(1)\;\]
     {is a $\tau$-neighbourhood of the identity. Now, for $g\in\V$, by $g\in\V_0$ we know that $g_{\upharpoonright A}\in \gc[A]$, and by the definition of the $\V_\alpha$ we also know that $g_{\upharpoonright A}\in (K_\tau)_{\upharpoonright A}$. Hence, by definition of $\pw^{\Delta_\tau, K_\tau}$, $\V\subseteq G_A^{K_\tau}$. In particular, $G_A^{K_\tau}\in \mathcal{N}_\tau(1)$, and so we get the desired containment between topologies.}

      All that is left to show is the minimality property of the statement. 
      Fix now some closed $K\leq {\gc(\Delta_{\tau})}$, such that $K\unlhd G(\Delta_{\tau})$ and $\pw^{\Delta_{\tau},K}\subseteq\tau$.
      We need to show that $K_{\tau}\leq K$. Since $K_{\tau}$ is closed, it suffices to show that {for each $A\in\X_{\Delta_{\tau}}$,}
      \begin{equation*}
        (K_{\tau})_{\upharpoonright A}\leq K_{\upharpoonright A}.
      \end{equation*}
      {Note that from $\pw^{\Delta_{\tau},K}\subseteq\tau$, we know that $G_A^K\in\mathcal{N}_\tau(1)$. Let $\alpha\in K_\tau$. By the definition of $K_\tau$, there is $h\in G_A^K$ such that $h_{\upharpoonright A}=\alpha_{\upharpoonright A}$. So $\alpha_{\upharpoonright A}\in K_{\upharpoonright A}$. Hence, as desired, $(K_{\tau})_{\upharpoonright A}\leq K_{\upharpoonright A}$.}
  \end{proof}

\subsubsection{The proof of Theorem~\ref{mainthm: c}}

\begin{theorem}\label{thm:onebased} %

Let $G\acts\Omega$ be a closed permutation group with locally finite algebraic closure, and let $\ind$ be an independence relation satisfying:
	\begin{enumerate}[label=(\Roman*)]
        \item {$k$-narrowness for some $k\geq 1$; and}
		\item either stationarity or  independent $3$-amalgamation (from Definition \ref{d:independence more}) holds. 
	\end{enumerate}
  Then every {non-indiscrete} semigroup topology  $\tau$  on $G$  {coarser than} $\pw$ is {of the form $\pw^{\Delta_{\tau},K_{\tau}}$ for some {algebraically} closed $G$-invariant $\Delta_{\tau}\subseteq \Omega$ and some {closed} $K_{\tau}\unlhd G(\Delta_\tau)$ with $K_\tau\leq {\gc}(\Delta_\tau)$}.
  
\end{theorem}
\begin{proof}
  Let $\tau$ be a semigroup topology on $G$ coarser than $\pw = \pw^{\Omega, \{1\}}$, and let $\Delta_\tau$ and $K_\tau$ be as provided by \cref{largest Delta}
  and \cref{minimal K} respectively. It  suffices to prove:
  
  \begin{itemize}
     \item $\Delta_{\tau}$ is algebraically closed; and
     \item $\pw^{\Delta_{\tau},K_{\tau}}=\tau$.
  \end{itemize}

  Write $\Sigma:=\Omega\setminus\Delta_{\tau}$. By part \ref{Delta b} of \cref{largest Delta}, no element of $G$ is $\tau$-rigid on any $a\in \Sigma$.
  From \cref{l:moving out} and \cref{tau-dispersive everywhere} it follows that every element of $G$ is $\tau$-dispersive on all of $\Sigma$. 
  By Item~\eqref{obs6}  of \cref{moving tau dispersiveness around} no element $g\in G$ is $\tau$-dispersive on any
  $a\in\cl(\Delta_{\tau})$; hence, $\Delta_{\tau}$ is algebraically closed.

 Since $\pw^{\Delta_\tau, K_\tau} \subseteq \tau$ by \cref{minimal K}, it
  remains to show $\tau \subseteq \pw^{\Delta_\tau, K_\tau}$, which is
  equivalent to: every $\U \in \mathcal{N}_\tau(1)$ contains an element of
  $\mathcal{N}_{\pw^{\Delta_\tau, K_\tau}}(1)$.\\
  \textbf{Claim:} Every $\U\in\mathcal{N}_{\tau}(1)$ contains {$G_{A}h$ for some $A\in\X_{\Delta_{\tau}}$ and some $h\in G_{A}^{K_{\tau}}$}.
   \begin{subproof}[Proof of Claim]
   Since every element of $G$ is $\tau$-dispersive on $\Sigma$, every
    element of $G$ is deeply $\tau$-dispersive on $\Sigma$ {by \cref{rem:deepstuff}.}
    {Let $L\in\{2,3\}$ be the constant from \cref{l:reaching over the empty set}, and let $\U\in\mathcal{N}_{\tau}(1)$. Applying \cref{building chains}  with $w=1$ 
   {(and $k$ there replaced by $kL$)},  we obtain an algebraically closed set $A\subset \Delta_{\tau}$, an element {$h\in G_{A}^{K_{\tau}}$}, and a chain of algebraically closed sets $\mathcal{E}=(\overline{e}_{j})_{j=0}^{kL}$ such that 
   	 \begin{itemize}
   	 	\item $G_{\E}h\subseteq\U$; and 
   	 	\item $\overline{e}_{i}\cap \overline{e}_{i+1}=A$ for all $0\leq i\leq kL-1$. 
        	 \end{itemize}}

   By \cref{c:reachability} (applied to the first inequality) 
   \begin{equation*}
   	G_{A}h\subseteq G_{\mathcal{E}}h\subseteq\U.
   \end{equation*}
   \end{subproof}
{To conclude, let $\U\in\mathcal{N}_{\tau}(1)$ and choose $\V\in\mathcal{N}_{\tau}(1)$ with
$\V^{2}\subseteq\U$. By the claim, there are $A\in\X_{\Delta_{\tau}}$ and
$h\in G_{A}^{K_{\tau}}$ with $G_{A}h\subseteq\V$; in particular $h\in\V$, since
$1\in G_{A}$.
Now let $s\in G_{A}^{K_{\tau}}$, so that $s_{\restriction A}=\alpha_{\restriction A}$ for
some $\alpha\in K_{\tau}$. By \cref{minimal K} there is $g\in\V$ with
$g_{\restriction A}=\alpha_{\restriction A}$, whence $sg^{-1}\in G_{A}$ and so
$s\in G_{A}\,g$. As $h\in\V$ and $g\in\V$, we obtain}
\begin{equation*}
  G_{A}^{K_{\tau}}\subseteq G_{A}\,\V\subseteq (G_{A}h)(h^{-1}\V)\subseteq \V^2
    \subseteq\U ,
\end{equation*}
{where $G_{A}^{K_{\tau}}$ is as in \cref{d:generalized pointwise stabilizer}. Since
$G_{A}^{K_{\tau}}\in\mathcal{N}_{\pw^{\Delta_{\tau},K_{\tau}}}(1)$, this proves $\tau\subseteq\pw^{\Delta_{\tau},K_{\tau}}$, and the result follows.
}

\end{proof}

\begin{proof}[Proof of \cref{mainthm: c}]
    Note that \cref{mainthm: c} follows from \cref{thm:onebased} by taking $k=1$. In this context, $\ind=\ind^{a}$ as pointed out in \cref{rem:onebased}, and so the second assumption in the theorem is precisely that algebraic independence satisfies independent $3$-amalgamation. 
    
    Suppose now $\tau\subseteq\pw$ is a semigroup topology on $G$. If $G$ is transitive, then the only $G$-invariant sets are $\emptyset$ and $\Omega$. In the case of $\Delta_\tau=\emptyset$, $\tau$ is the indiscrete topology, which is not $T_{0}$. %
    If $\Delta_{\tau}=\Omega$ then every $g\in G$ is $\tau$-rigid on $\Omega$, so by \cref{cor:T0} either $\tau=\pw$ or $\tau$ is not $T_{0}$. In either case, $\pw$ is minimal amongst  $T_{0}$ semigroup topologies on $G$.

 Finally, suppose in addition $G$ is primitive. Then the only $G$-invariant equivalence relations on $\Omega$ are equality and the full relation; since the classes $[a]$ are finite by local  finiteness of algebraic closure and $\Omega$ is infinite, $\sim$ is equality, so   $\gc(\Omega)=\{1\}$ and $\pw^{\Omega,\gc(\Omega)}=\pw$. By the previous paragraph, a semigroup topology strictly coarser than $\pw$ must have $\Delta_{\tau}=\emptyset$, and  hence is the indiscrete topology.
\end{proof}

\section{Minimality of $\pw$ for monoids of elementary embeddings}
\label{sec:min monoids}

 \subsection{Sinks}
  
  \newcommand{\into}{\hookrightarrow}

  \begin{definition}\label{sinks}
  Let $G$ satisfy \ref{sat} and let $S=\overline{G}$. For $k\geq 1$, we say that an (infinite) algebraically closed set $\Omega'\subsetneq \Omega$ is a $k$\textbf{-sink}  if the following conditions are satisfied:
  \begin{enumerate}
  		\item \label{sink3}  for every chain of algebraically closed sets $\C=(\overline{c}_{i})_{i=0}^{k}$ such that {$\overline{c}_{i}\cap \overline{c}_{i+1}\subseteq\Omega'$},
        for %
         $0\leq i\leq k-1$, there is $u\in\nn_{\C}$ such that $\im(u)\subseteq \Omega'$;
  		\item \label{sink4} for every $G$-invariant and algebraically closed $\Delta\subseteq\Omega$, there is $v\in S$ such that $\im(v)\cap \Omega'=v(\Delta)$.
  	\end{enumerate}
  \end{definition}

\begin{lemma}\label{lem:forsink1}  %
	Let $G\acts\Omega$ be the automorphism group of a countable saturated structure and $X\subseteq\Omega$ be an algebraically closed set. Then, there are $\alpha, \beta\in\overline{G}$ such that%
    \begin{equation}\label{eq:sink2}
       {X=\{a\in\Omega\,|\,\alpha(a)= \beta(a)\}}\;. %
    \end{equation}
    
\end{lemma}
\begin{proof}
The argument is identical to that of Lemma~\ref{lem:saturatedcase}. 
Let $Y:=\Omega\setminus X$, and consider $\M$ as before. Let $\overline{b}:=(b_0, b_1, b_2, \dots)$ be a tuple of length $\omega$ enumerating $\M$. Consider the infinitary type $p(\overline{x}):=\mathrm{tp}(\overline{b})$. For the tuples of variables $\bar x$ and $\bar y$ of length $\omega$, consider the partial type 
\[\Gamma(\overline{x}, \overline{y}):=p(\overline{x})\cup p(\overline{y}) \cup\{x_{i}\neq y_{j}\vert b_i,b_{j}\in Y\}\cup\{x_i=y_i\vert\  b_i\in X\}.\]

It suffices to prove that $\Gamma(\overline{x}, \overline{y})$ is (finitely) satisfiable, and one then concludes as in \cref{lem:saturatedcase}. 

Let $J\subseteq\omega$ be finite,  $\Gamma_J(\overline{x}_J, \overline{y}_J)$  denote the restriction of $\Gamma$
to the variables $x_i, y_i$ with $i \in J$. Write
$B_J := \{\, b_i \mid i \in J \,\}$, and set $B_J^X := B_J \cap X$ and
$B_J^Y := B_J \cap Y$. Since $\acl(B_J^X) \cap Y = \emptyset$, Neumann's
Lemma provides $g \in G$ that fixes $B_J^X$ pointwise and satisfies
$g(B_J^Y) \cap B_J^Y = \emptyset$. The pair of tuples
$\bigl( (b_i)_{i \in J},\, (g(b_i))_{i \in J} \bigr)$ then satisfies
$\Gamma_J$: both tuples realise the restriction of $p$ to $J$, since
$g \in G$; the equalities $x_i = y_i$ for $b_i \in X$ hold because $g$ fixes
$B_J^X$ pointwise; and the inequalities $x_i \neq y_j$ for $b_i, b_j \in Y$
hold because $g(B_J^Y)$ is disjoint from $B_J^Y$. Hence $\Gamma$ is finitely
satisfiable, and so satisfiable.
\end{proof} 

  The following useful observation from \cite{pinskerschindler} follows from the fact that in a Hausdorff space $X$ the diagonal
$\{(x, x) \mid x \in X\} \subseteq X^2$ is closed.
  \begin{observation}\label{use sink}
  	In the situation of \cref{sinks}, since $\Omega'\subseteq\Omega$ is algebraically closed,  \cref{lem:forsink1} provides $\alpha, \beta\in S$ such that (\ref{eq:sink2}) holds. Hence, if $\tau$ is a Hausdorff semigroup topology on $S$
    then the set
    \begin{equation*}
       \{s\in S\,|\,\im(s)\nsubseteq \Omega'\,\}=\{s\in S\,|\,\alpha s\neq \beta s\}
    \end{equation*}  
  	is open in $\tau$. 
  \end{observation}

\begin{remark} If $\Omega'\subset\Omega$ is a $k$-sink, then it is an $n$-sink for every $n\geq k$. To see this, we only need to verify condition~(\ref{sink3}) in Definition~\ref{sinks}. %
Let $\C=(\overline{c}_i)_{i=0}^n$ be a chain of algebraically closed sets and consider the subchain $\D:=(\overline{c}_i)_{i=0}^k$. Since $\Omega'$ is a $k$-sink, there is $u\in\nn_\D=\nn_{\overline{c}_0}\cdots \nn_{\overline{c}_k}$ with $\im(u)\subseteq\Omega'$. %
 {Since the identity belongs to
  $\nn_{\overline{c}_l}$ for each $k < l \leq n$, we also have $u \in \nn_{\C}$,
  so condition~(\ref{sink3}) holds for $n$.}
\end{remark}

The following is a standard union of chains argument (cf.~\cite[Lemma 4.5]{marimon2025guide}):

   \begin{lemma}\label{many divisors}
  	Let $\M$ be a countable saturated structure and let $S$ be its monoid of elementary embeddings.
    Let $(s_{l})_{l\in\mathbb{N}}$  be a countable sequence of elements  of $S$. Then there is $w\in S$ such that for each $m\geq 1$ there is $t_m\in S$ with $w=t_{m}s_{m-1}s_{m-2}\cdots s_{0}$.  
  \end{lemma}
  \begin{proof}
   {We construct an elementary chain of countable models $\M=:\M_{0}\preceq \M_{1}\dots $ together with isomorphisms $\alpha_{l}:\M_{l}\to\M$ satisfying 
   \begin{equation}\label{eq:intertwine}
    \alpha_{l+1} \circ \iota_l = s_l \circ \alpha_l
    \qquad \text{for all } l \in \mathbb{N},
  \end{equation}
  where $\iota_l \colon \M_l \to \M_{l+1}$ denotes the inclusion. %
   Start with
  $\M_0 := \M$, and let $\alpha_0 $  be an arbitrary isomorphism of $\M_0$. For the inductive step, suppose $(\M_i)_{i \leq l}$ and $(\alpha_i)_{i \leq l}$ have been constructed. We can construct a countable elementary extension $\M_{l+1}$ of $\M_l$ such that the inclusion $\iota_l:\mathcal{M}_l\to\mathcal{M}_{l+1}$ is isomorphic over  $\alpha_l$ to the elementary embedding $s_l: \mathcal{M}\to\mathcal{M}$, that is $\alpha_{l+1}:\M_{l+1}\to \M$ is such that $\alpha_{l+1}\iota_{l}=s_{l}\alpha_{l}$.
  Let $\M' := \bigcup_{l \in \mathbb{N}} \M_l$. Then $\M'$ is a countable   elementary extension of $\M$, and it is saturated: any finite subset of $\M'$  is contained in some $\M_l$,  and for any finite $A\subseteq M'$ there is $l$ such that $A\subset M_{l}$, and since any type over $A$ is realised in $M_{l}$, it is also realised in $\M'$.
  By uniqueness of the countable saturated model,
  there is an isomorphism $\alpha \colon \M' \to \M$.
  For $l \in \mathbb{N}$, let $\iota_{l,\infty} \colon \M_l \to \M'$ be the
  inclusion, and set $w := \alpha\, \iota_{0,\infty}\, \alpha_0^{-1}$. We claim
  $w$ satisfies the statement of the Lemma. Fix $m \geq 1$ and let
  $t_m := \alpha\, \iota_{m,\infty}\, \alpha_m^{-1}$. This is an elementary
  self-embedding of $\M$ being a composition of elementary maps. Using
  $s_l = \alpha_{l+1}\, \iota_l\, \alpha_l^{-1}$ from \eqref{eq:intertwine}:
  \begin{equation*}
    t_m\, s_{m-1} \cdots s_0
      = \alpha\, \iota_{m,\infty}\, \alpha_m^{-1}
        (\alpha_m \iota_{m-1} \alpha_{m-1}^{-1})
        (\alpha_{m-1} \iota_{m-2} \alpha_{m-2}^{-1})
        \cdots
        (\alpha_1 \iota_0 \alpha_0^{-1})
      = \alpha\, \iota_{0,\infty}\, \alpha_0^{-1}
      = w,
  \end{equation*}
  since $\iota_{m,\infty}\, \iota_{m-1} \cdots \iota_0 = \iota_{0,\infty}$.
}
  \end{proof}

Our reason for caring about $k$-sinks is that their existence implies minimality of $\pw$:
  
\begin{lemma}\label{abstract minimality endomorphism monoids}
  	Let $G$ be the automorphism group of a saturated countable structure $\M$ with locally finite algebraic closure and let $S=\overline{G}$. If $\M$ admits a $k$-sink for some $k\geq 1$, then $\pw$ is a minimal Hausdorff semigroup topology on $S$. 
  \end{lemma}
  \begin{proof}
  	Suppose, for contradiction, that $\tau$ is a Hausdorff semigroup topology on
  $S$ strictly coarser than $\pw$. Let 
    \[\Delta:=\{a\in\Omega \ \vert \text{ every } s\in S \text{ is }\tau\text{-rigid on } a\}\;.\]
By \cref{rigidity over all implies non T2}, $\Delta\subsetneq\Omega$. {Moreover, by Item~\eqref{obs2} of \cref{moving tau dispersiveness around}, $\Delta$ is $G$-invariant.} Write $\Sigma:=\Omega\setminus \Delta$ and let $\{\Sigma_{j}\}_{j\in J}$ be the (necessarily countable) partition of $\Sigma$ into $G$-orbits.\\
  	
\textbf{Claim:}  	  There is $w'\in S$  that is deeply $\tau$-dispersive on all of $\Sigma$. 
  	\begin{subproof}[Proof of claim]
  	 Fix $j\in J$. Since $\Sigma_j\subseteq\Omega\setminus\Delta$, there is some $t_{j}\in S$ that is not $\tau$-rigid on some $a_{j}\in\Sigma_{j}$.  
     By \cref{l:moving out} there is $t'_{j}\in S$ that is $\tau$-dispersive on some $a'_{j}\in\Sigma_{j}$. Then by \cref{tau-dispersive everywhere} there {is}
     some $s_{j}\in S$ which is $\tau$-dispersive on every point of $\Sigma_{j}$. 
 
  	 Using \cref{many divisors} we construct an element $s_{J}\in S$ with the property that for every $j\in J$ there are $u_j, v_j\in S$ %
     such that $s_{J}=u_{j}s_{j}v_{j}$ (here the elements $v_j$ are just sequences of the form $s_{j_m}s_{j_{m-1}}\dots s_{j_0}$ under some ordering of $J$). By virtue of \cref{moving tau dispersiveness around}, we conclude that $s_{J}$ is $\tau$-dispersive on $\Sigma$. 
     To conclude, apply \cref{many divisors} to the constant sequence of countably many copies of $s_J$ to find $w'\in S$ with the property that $w'$ admits a decomposition of the form $w'=t_{m}s_{J}^{m}$ for every $m\geq 1$. By  \cref{moving tau dispersiveness around} applied to the term $t_{m}s_{J}$ we conclude that $w'$ is deeply $\tau$-dispersive on $\Sigma$.  
  	\end{subproof}

{We now argue that $\Delta$ is algebraically closed. Take $a\in\acl(\Delta)$. Then, since $w'\in S$, we can pick $g\in G$ such that $gw'(a)=a$.  By Item~\eqref{obs4} of \cref{moving tau dispersiveness around}, since $w'$ is $\tau$-dispersive on $\Sigma$ and $\Sigma$ is $G$-invariant, $g w'$ is $\tau$-dispersive on $\Sigma$. Now, $a\in \acl(\Delta)\cap \im(g w')$, and so by Item~\eqref{obs6} of \cref{moving tau dispersiveness around}, $g w'$ is not $\tau$-dispersive on $a$, meaning that $a\in\Delta$, and thus concluding the argument that $\Delta$ is algebraically closed.}

     Let $\Omega'$ be a $k$-sink for $S$, for some {$k\geq 1$.}
     By property~(\ref{sink4}) of the definition of a sink, there is $v\in S$ such that $\im(v)\cap\Omega'=v(\Delta)$.
   {   Since $\Delta$ is $G$-invariant, every
  element of $S$ maps $\Delta$ into $\Delta$ and $\Omega \setminus \Delta$
  into $\Omega \setminus \Delta$; hence the same property passes to every
  product $v v'$ with $v' \in S$.}
     In particular, the property holds for any power of $w=vw'$. This implies that for all $m$ the set 
     {    \begin{equation*}
  	\U:=\{s\in S\,|\,\im(s)\nsubseteq \Omega'\}\, 
  \end{equation*}}
{is $\tau$-open by \cref{use sink} and  we have
  $\U \in \mathcal{N}_\tau(w^m)$ for all $m \geq 1$. Using continuity of
  multiplication, take $\U' \in \mathcal{N}_\tau(w)$ and
  $\U'' \in \mathcal{N}_\tau(w^{2k+1})$ with $\U' \U'' \subseteq \U$.}

     { Since $w$ is deeply $\tau$-dispersive on $\Sigma = \Omega \setminus \Delta$
  and $\Delta$ is $G$-invariant and algebraically closed,
  \cref{building chains} applied to $\U''$ yields $h \in G$ and a chain of
  algebraically closed sets $\E' = (\overline{e}'_j)_{j=0}^{k}$ such that}
   \begin{itemize}
     	\item $\overline{e}'_{j}\cap \overline{e}'_{j+1}\subseteq\Delta$ for all $0\leq j\leq k-1$,
     	\item $\nn_{\E'}h\subseteq\U''$. 
     \end{itemize}
      Since %
      $\tau\subseteq\pw$, the $\tau$-neighbourhood $\U'$ of $w$ contains a  basic $\pw$-neighbourhood of $w$. Hence,
      we can find $h'\in\U'{\cap G}$ agreeing with $w$ on $\E'$. 
      
{Set  $\overline{e}_j := h'(\overline{e}'_j)$ and
  $\E := (\overline{e}_j)_{j=0}^{k}$. For $0 \leq j \leq k-1$,
  \[
    \overline{e}_j \cap \overline{e}_{j+1}
      = h'\bigl(\overline{e}'_j\bigr) \cap h'\bigl(\overline{e}'_{j+1}\bigr)
            \subseteq w(\Delta) = v w'(\Delta) \subseteq v(\Delta) \subseteq \Omega',
  \]
  using that $h'$ agrees with $w$ on $\E'$, that
  $\overline{e}'_j \cap \overline{e}'_{j+1} \subseteq \Delta$, that
  $w'(\Delta) \subseteq \Delta$, and the choice of $v$.   }   %
   { By \cref{o:commutation}, $\nn_{\E} h' = h' \nn_{\E'}$, so}
 
      \begin{equation*}
        \nn_{\E}h'h=h'\nn_{\E'}h\subseteq\U'\U''\subseteq\U.
      \end{equation*}
      Now apply Item~(\ref{sink3}) of the definition of a $k$-sink {to the chain $\E$}: we conclude the existence of $u\in\nn_{\E}$ with $\im(u)\subseteq \Omega'$. 
      { Then   $u h' h \in \nn_{\E}\, h' h \subseteq \U$, while
  $\im(u h' h) \subseteq \Omega'$, contradicting the
  definition of $\U$. Hence no Hausdorff semigroup topology on $S$ is strictly
  coarser than $\pw$, i.e., $\pw$ is minimal.}
  \end{proof}

 \subsection{{The method of universally embedded models}}

  In this section we define a natural way to obtain sinks through the existence of a sufficiently ``generic'' copy of $\M$ within itself, which we  call a universally embedded model. In Section~\ref{subsub:existence} we define universally embedded models and give examples. {We shall later show in Section~\ref{sub:universalfromlovely} that lovely pairs, a natural notion of a generic copy of a model~\cite{ben2003lovely}, are a source of universally embedded models.} 
Our main result in \cref{subsub:existence} is \cref{from base monotonicity to universally embedded}, where we show that
universally embedded models exist whenever {algebraic closure is modular.} In \cref{subsub:sinks}, we then show that under the same assumption universally embedded models yield sinks, completing the proof of \cref{t:minimality endomorphism monoids}.
  
  \subsubsection{On the existence of universally embedded models}\label{subsub:existence}
 \begin{definition}\label{def: absorbing configuration} Let $\Omega'\subseteq\Omega$ be a (possibly infinite) algebraically closed set. An \textbf{absorbing configuration} $(\overline{a}_1; \overline{a}_2; \overline{b})$ for $\Omega'$ consists of algebraically closed finite tuples $\overline{a}_1, \overline{a}_2$, and $\overline{b}$ from $\Omega$ such that 
\begin{enumerate}[label=(\alph*)]
\item\label{p1} $\overline{a}_1\subseteq \overline{a}_2$; and
          \item \label{p3}$\overline{b}\ind^{a}_{\overline{b}\cap\Omega'} \overline{a}_1$; and
          \item \label{p4}$\overline{b}\cap\Omega'\ind^{a}_{\overline{a}_1} \overline{a}_2$.
\end{enumerate}
 \end{definition}

  \begin{definition}\label{universally embedded models} Let $G\acts\Omega$ be the automorphism group of a countable saturated structure $\M$ with locally finite algebraic closure.
  	We say that {an algebraically closed set} $\Omega'\subsetneq\Omega$ is  \textbf{universally embedded} in $\Omega$ if:  

        \begin{itemize}
\item (absorption) whenever $(\overline{a}_1; \overline{a}_2; \overline{b})$ is an absorbing configuration for $\Omega'$, there are $\overline{a}_1' \overline{a}_2'\equiv_{\overline{b}} \overline{a}_1 \overline{a}_2$ with $\Omega'\cap \overline{a}_2'=\overline{a}_1'$.
        
  	\end{itemize}
  		 
We will see in \cref{obs:universally embedded models} below that if $\Omega'\subseteq\Omega$ is universally embedded, then it is the universe of an elementary submodel of $\M$; {accordingly, we call $\Omega'$ a
  \textbf{universally embedded model}.}
  We  say that $G\acts\Omega$ \textbf{admits universally embedded models} if a universally embedded $\Omega'\subseteq\Omega$ exists. \cref{fig:unversallyembedded} illustrates the behaviour of a universally embedded model. 
  \end{definition}

  Before turning to deeper structural results, the reader may verify the following are universally embedded models when $G=\Aut(\M)$:
  \begin{itemize}
  \item a dense codense subset of $\M:=(\mathbb{Q};<)$;
  \item a subspace of infinite dimension and infinite codimension, where $\M$ is a countably infinite vector space over a finite field.  
  \end{itemize}

 \begin{figure}[t]
     \begin{center} \includegraphics[width=0.5\textwidth]{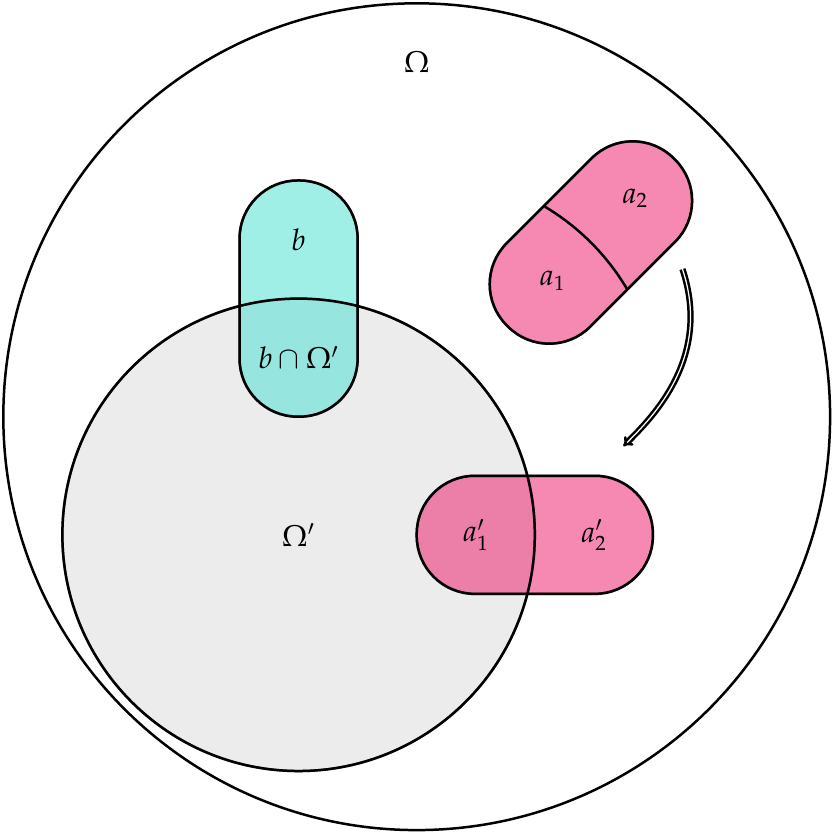}     
     \end{center}
     \caption{An illustration of the definition of a universally embedded model. In the figure, we have that $(\overline{a}_1; \overline{a}_2; \overline{b})$ is an absorbing configuration for $\Omega'$. In particular, $\overline{b}\ind^{a}_{\overline{b}\cap\Omega'} \overline{a}_1$ and $\overline{b}\cap\Omega'\ind^{a}_{\overline{a}_1} \overline{a}_2$. This allows us to ``absorb'' $\overline{a}_1$ into $\Omega'$: that is, to find $\overline{a}_1'\overline{a}_2'\equiv_{\overline{b}} \overline{a}_1 \overline{a}_2$ such that $\Omega'\cap \overline{a}_2'=\overline{a}_1'$.}
     \label{fig:unversallyembedded}
 \end{figure}

\begin{definition} Assume \ref{sat}. We say that an algebraically closed set $\Omega'\subseteq\Omega$ has the \textbf{coheir property}
 if for every absorbing configuration $(\overline{a}; \overline{a}; \bar b)$ for $\Omega'$, there is $\overline{a}'\equiv_{\overline{b}}\overline{a}$ with $\overline{a}'\subseteq\Omega'$. Equivalently, $\Omega'$ has the \textbf{coheir} property if and only if
 \begin{itemize}
     \item (coheir) for all finite and algebraically closed tuples $\overline{a}, \overline{b}$ from $\Omega$ with $\overline{b}\ind^{a}_{\overline{b}\cap\Omega'
     } \overline{a}$, there is $\overline{a}'\equiv_{\overline{b}} \overline{a}$ with $\overline{a}'\subseteq \Omega'$.
 \end{itemize}
In particular, every universally embedded model has the coheir property.
\end{definition}

\begin{observation}\label{obs:universally embedded models} Assume \ref{sat},  and 	let $\Omega'\subseteq\Omega$ be universally embedded. Then $\Omega'$ is the universe of
    an elementary submodel $\M'\preceq\M$ and $\M'$ is saturated.
\end{observation}
\begin{proof} This is a simple application of the Tarski--Vaught test~\cite[Theorem 2.1.2]{tent2012course} {and the coheir property (with $\overline{b}$ contained in $\Omega'$): take a formula $\phi(\overline{x}, \overline{b})$, where $\overline{b}$ is a tuple from $\Omega'$ and such that there is some $\overline{a}\in\Omega$ such that $\phi(\overline{a}, \overline{b})$ holds in $\M$. The coheir property applied to (the closure of) $\overline{b}$ tells us that there is some $\overline{a}'\subseteq\Omega'$ such that $\phi(\overline{a}', \overline{b})$ also holds in $\M$. By the Tarski-Vaught test, we then conclude that $\Omega'$ is the universe of an elementary submodel $\M'$ of $\M$. The same argument shows that every type over a finite  subset of $\Omega'$ which is realised in $\M$ is also realised in $\M'$; since $\M$ is  saturated, so is $\M'$.}
\end{proof}

 \begin{lemma}\label{from base monotonicity to universally embedded}
   Assume \ref{sat} and suppose {that algebraic closure is modular.} 
   Then $G\acts\Omega$ admits universally embedded models. In fact, the family of all universally embedded models is comeagre in the space of all algebraically closed subsets of $\Omega$. 
  \end{lemma}
  \begin{proof}
  	Let $\mathfrak{C}\subseteq 2^{\Omega}$ 
    be the space of all (possibly infinite) algebraically closed subsets of $\Omega$,
    {where $2^{\Omega}$ carries the Tychonoff 
    topology.}
    It is not hard to see that   $\mathfrak{C}$ is closed in $2^{\Omega}$. A basic open set of $2^\Omega$ has  the form
  $\U_{E_0, E} = \{\, F \subseteq \Omega \mid F \cap E = E_0 \,\}$ for finite   $E_0 \subseteq E \subseteq \Omega$, and $\U_{E_0, E} \cap \mathfrak{C} \neq \emptyset$ if and only if  $\cl(E_0) \cap E = E_0$. %

     Since countable intersections of comeagre sets are comeagre, it suffices to   show that for each triple $(\overline{a}_1; \overline{a}_2; \overline{b})$  of finite algebraically closed tuples there is a comeagre   $\V \subseteq \mathfrak{C}$ such that every $\Omega' \in \V$ satisfies one   of:
  \begin{itemize}
    \item $(\overline{a}_1; \overline{a}_2; \overline{b})$ is not an absorbing
      configuration for $\Omega'$; or
    \item there are $\overline{a}'_1 \overline{a}'_2 \equiv_{\overline{b}}
      \overline{a}_1 \overline{a}_2$ with
      $\overline{a}'_2 \cap \Omega' = \overline{a}'_1$.
  \end{itemize}

     {  For this, it suffices in turn to show that every non-empty basic open set
  $\mathfrak{C} \cap \U_{E_0, E}$ contains a non-empty open subset $\U'$ all
  of whose members satisfy one of the two alternatives.}

  So take $E_0 \subseteq E$ with $\cl(E_0) \cap E = E_0$. Passing to an open   subset of $\mathfrak{C} \cap \U_{E_0, E}$ (i.e.\ enlarging $E_0$ and $E$),   we may assume $\overline{b} \subseteq E$ and that $E_0$ and $E$ are algebraically closed. Write $\overline{b}_0 := \overline{b} \cap E_0$. Since   $\overline{b} \subseteq E$, every
  $\Omega' \in \mathfrak{C} \cap \U_{E_0, E}$ satisfies
  $\Omega' \cap \overline{b} = \overline{b}_0$. 
  Consequently $(\overline{a}_1; \overline{a}_2; \overline{b})$ is an absorbing   configuration for such an $\Omega'$ if and only if it is an absorbing configuration for
  $\overline{b}_0$ (a condition independent of $\Omega'$). 
  If it fails, the  first alternative holds on all of $\mathfrak{C} \cap \U_{E_0, E}$ and we are  done; so assume $(\overline{a}_1; \overline{a}_2; \overline{b})$ is an  absorbing configuration for $\overline{b}_0$.
    
     We  construct finite algebraically closed sets $E'_{0}\subseteq E'$ with $E\subseteq E'$, and:
     
     \begin{enumerate}[label=(\roman*)]
     	\item\label{item:clos1} $E'_{0}\cap E=E_{0}$;
     	\item\label{item:clos2} there is $\overline{a}_2'\overline{a}_1'\equiv_{\overline{b}} \overline{a}_2\overline{a}_1$ such that $\overline{a}_2'\subseteq E'$ and $\overline{a}_2'\cap E_0'=\overline{a}_1'$.
     \end{enumerate}
    Granted these, it follows that the open (and non-empty) set  $\U':=\mathfrak{C}\cap\U_{E'_{0},E'}\subseteq\mathfrak{C}\cap\U_{E_{0},E}$ satisfies the required properties: { it
  is non-empty since $E'_0$ is algebraically closed, so
  $\cl(E'_0) \cap E' = E'_0$; the containment
  $\U' \subseteq \mathfrak{C} \cap \U_{E_0, E}$ follows from
  \ref{item:clos1} and $E \subseteq E'$; and every $\Omega' \in \U'$ satisfies   $\Omega' \cap E' = E'_0$, hence
  $\Omega' \cap \overline{a}'_2 = E'_0 \cap \overline{a}'_2 =  \overline{a}'_1$, giving the second alternative.}
    
 For the construction, use full existence to pick
  $\overline{a}'_1 \overline{a}'_2 \equiv_{\overline{b}}
  \overline{a}_1 \overline{a}_2$ with
  $E \ind^{a}_{\overline{b}} \overline{a}'_1 \overline{a}'_2$, and set
  $E' := \cl(E\, \overline{a}'_1 \overline{a}'_2)$,
  $E'_0 := \cl(E_0\, \overline{a}'_1)$. Conditions \ref{item:clos1} and
  \ref{item:clos2} are then equivalent, respectively, to:
   \begin{enumerate}[label=(\roman*$^\star$)]
    	\item \label{item1} $\overline{a}_1'\ind^{a}_{E_0} E$;
    	\item \label{item2} $E_0'\ind^{a}_{\overline{a}_1'} \overline{a}_2'$. 
    \end{enumerate}

   In order to do this, let us look back first at our assumption that $(\overline{a}_1; \overline{a}_2; \overline{b})$ is an absorbing configuration for $\overline{b}_0$. Since $\overline{a}_2'\overline{a}_1'\equiv_{\overline{b}} \overline{a}_2 \overline{a}_1$, by invariance, $(\overline{a}_1'; \overline{a}_2'; \overline{b})$ is also an absorbing configuration for $\overline{b}_0$.
    Conditions \ref{p3} and \ref{p4} in \cref{universally embedded models} are then equivalent to:
     
     \begin{minipage}{.48\linewidth} %
    	\begin{equation}\label{cons1}
    	\overline{b}\ind^{a}_{\overline{b}_{0}}\overline{a}_1',
    	\end{equation}
    \end{minipage}%
    \hfill %
    \begin{minipage}{.48\linewidth} %
\begin{equation}\label{cons2}
    	\overline{b}_{0}\ind^{a}_{\overline{a}_1'}\overline{a}_2'.
    	\end{equation}
    \end{minipage}	
      
    We use the assumption that $\ind^{a}$ is a symmetric independence relation (the non-trivial assumption is that it satisfies base monotonicity) in a crucial way. To establish \ref{item1} we use monotonicity, transitivity together with  \cref{cons1}, and finally base monotonicity:
    \begin{align*}
    E\ind^{a}_{\overline{b}}\overline{a}_2' \quad &\Rightarrow \quad E\ind^{a}_{\overline{b}}\overline{a}_1' &\text{ by monotonicity,}\\
&\Rightarrow\quad E\ind^{a}_{\overline{b}_{0}}\overline{a}_1' & \text{by transitivity and \eqref{cons1},}\\
&\Rightarrow \quad E\ind^{a}_{E_{0}}\overline{a}_1' & \text{ by base monotonicity.}
    \end{align*}
    Thus, we conclude \ref{item1} by symmetry. Finally, for \ref{item2}, 
    \begin{align*}
        E\ind^{a}_{\overline{b}}\overline{a}_2' \quad & \Rightarrow\quad E_{0}\ind^{a}_{\overline{b}_{{0}}}\overline{a}_2' &{\text{by transitivity and \eqref{cons1},}}\\
        &\Rightarrow\quad E_{0}\ind^{a}_{\overline{b}_0 \overline{a}_1'} \overline{a}_2' &\text{ by base monotonicity},\\
        &\Rightarrow\quad E_{0}\ind^{a}_{\overline{a}_1'}\overline{a}_2' &\text{ by transitivity,  \eqref{cons2}, and } \overline{b}_0\subseteq E_0.
    \end{align*}
    Since we verified \ref{item:clos1} and \ref{item:clos2}, this completes our proof.
  \end{proof}

\subsubsection{Universally embedded models as sinks}\label{subsub:sinks}

We first verify that universally embedded models satisfy condition~\ref{sink4} of \cref{sinks}.

\begin{lemma}\label{properties of universally embedded sets}
  	Assume \ref{sat} and let $\Omega'\subseteq\Omega$ be a universally embedded model. Then for every  $G$-invariant algebraically closed $\Delta\subseteq\Omega$  there is $v\in \overline{G}$ with $\im(v)\cap\Omega'=v(\Delta)$.

  \end{lemma}

  \begin{proof}
 Fix such a $\Delta\subseteq\Omega$. Let $(A_{(i,2)})_{i\in\mathbb{N}}$ be a nested sequence of finite algebraically closed subsets of $\Omega$ with %
 { $A_{(0,2)} := \acl(\emptyset)$}
 and such that $\bigcup_{i\in\mathbb{N}} A_{(i,2)}=\Omega$. Set
  $A_{(i,1)} := A_{(i,2)} \cap \Delta$, for each $i\in\mathbb{N}$. %
  {It suffices to find a sequence
  $(g_i)_{i \in \mathbb{N}}$ in $G$ such that
  $g_{i+1} \upharpoonright_{A_{(i,2)}} = g_i \upharpoonright_{A_{(i,2)}}$ and
  \begin{equation}\label{eq: want for gi}
    g_i(A_{(i,1)}) = g_i(A_{(i,2)}) \cap \Omega'
  \end{equation}
  for every $i\in \mathbb N$: the limit $v$ of $(g_i)_{i\in \mathbb N}$ in the topology of pointwise
  convergence then satisfies $\im(v) \cap \Omega' = v(\Delta)$.}

{  We build the $g_i$ inductively, writing $A'_{(i,j)} := g_i(A_{(i,j)})$ for   $j \in \{1, 2\}$. %
For the base case, any   $g_0 \in G$ works as $A_{(0,2)} = \acl(\emptyset) \subseteq \Delta$, so  $A_{(0,1)} = A_{(0,2)}$, and any 
  $g_0$ preserves $\acl(\emptyset)$ setwise.} 

{ For the inductive step, suppose $g_n$ is given. Set
  $A''_{(n+1,j)} := g_n(A_{(n+1,j)})$ for $j \in \{1,2\}$ and let $\overline{a}^{k}_{(i,j)}$ denote an enumeration of  $A^{k}_{(i,j)}$, for $k \in \{\prime, \prime\prime\}$.}
  
\textbf{Claim:} $(\overline{a}''_{(n+1, 1)}; \overline{a}''_{(n+1, 2)}; \overline{a}'_{(n,2)})$ is an absorbing configuration for $\Omega'$.
\begin{subproof}[Proof of Claim] Firstly, \ref{p1} is true by construction (and the fact that both ${A_{(n+1, 2)}}$
and $\Delta$ are algebraically closed). For \ref{p3}, note that
\[A_{(n,2)}\cap A_{(n+1, 1)}=A_{(n,2)}\cap A_{(n+1, 1)}\cap\Delta=A_{(n,1)}\;.\]
From this and the definition of algebraic independence we get that 
\[A_{(n,2)}\ind^{a}_{A_{(n,1)}} A_{(n+1, 1)}\;.\]
Finally, by invariance (multiplying all sides by {$g_n$}), 
\[A'_{(n,2)}\ind^{a}_{A'_{(n,1)}} A''_{(n+1, 1)}\;.\]
This is precisely condition \ref{p3} for $(\overline{a}''_{(n+1, 1)}; \overline{a}''_{(n+1, 2)}; \overline{a}'_{(n,2)})$.

For condition \ref{p4}, note that by %
{existence} and monotonicity
\[A_{(n,1)}\ind^{a}_{A_{(n+1, 1)}} A_{(n+1, 2)}\;.\]
Again, by invariance, 
\[A'_{(n,1)}\ind^{a}_{A''_{(n+1, 1)}} A''_{(n+1, 2)}\;,\]
which is precisely \ref{p4}.
\end{subproof}

Since $(\overline{a}''_{(n+1, 1)}; \overline{a}''_{(n+1, {2})}; \overline{a}'_{(n,2)})$ is an absorbing configuration and $\Omega'$ is a universally embedded model, there are $A_{(n+1,2)}'$ and 
$A_{({n+1},1)}'$ such that
\begin{equation}\label{eq: where to send for 3}
\overline{a}_{(n+1,2)}''\overline{a}''_{(n+1,1)}\equiv_{A'_{(n,2)}} \overline{a}_{(n+1,2)}' 
{\overline{a}_{(n+1,1)}'}
\end{equation}
and $A_{(n+1,2)}'\cap \Omega'=A'_{(n+1,1)}$. {Let $h \in G$ witness
  \eqref{eq: where to send for 3}; then $g_{n+1} := h g_n$ agrees with $g_n$
  on $A_{(n,2)}$ (as $h$ fixes $\overline{a}'_{(n,2)}$ pointwise) and
  satisfies \eqref{eq: want for gi} for $i = n+1$, since
  $g_{n+1}(A_{(n+1,j)}) = h(A''_{(n+1,j)}) = A'_{(n+1,j)}$. This completes
  the induction, and the proof.}
  \end{proof}

  {We conclude by proving that universally embedded models are $1$-sinks whenever algebraic closure is modular.}

  \begin{lemma}\label{universally embedded are sinks}
Let $G\acts\Omega$ satisfy \ref{sat} {and be such that algebraic closure is modular. Let $\Omega'\subsetneq\Omega$ be a universally embedded model and let $S:=\overline{G}$.
Then $\Omega'$ is a $1$-sink for $S$.}
   	
  \end{lemma}
  \begin{proof} %
    By definition, $\Omega' \subsetneq \Omega$ is algebraically closed. Meanwhile,
  property~\ref{sink4} was established in
  \cref{properties of universally embedded sets}. It remains to verify
  property~\ref{sink3}.

  	{Recall from \cref{obs:universally embedded models} that $\Omega'$ is the universe of a saturated elementary submodel $\M'\preceq \M$. Let $\C=(\overline{c}_{0}, \overline{c}_1)$ be a chain of algebraically closed sets.
    By \cref{o:chain}, in order to show that $\nn_{\C}$ contains an element whose image is contained in $\Omega'$, it suffices to show the existence of a chain of algebraically closed sets $\C'=(\overline{c}_{0}', \overline{c}_{1}')$ where $\overline{c}_{0}=\overline{c}'_{0}$, with $\overline{c}_{1}'$ a tuple from $\Omega'$,  and $\overline{c}'_{0}\overline{c}'_{1}\equiv \overline{c}_{0}\overline{c}_{1}$.  In fact, given $\C'$, since $\Omega'$ is the universe of a saturated elementary submodel of $\M$, a back and forth argument yields that the partial map $\phi:\overline{c}_1\to\overline{c}_1'$ extends to an elementary map $u:\M\to\M'$. From this, \cref{o:chain} yields that $u\in\nn_{\C}$. By construction $\im(u)\subseteq\Omega'$ as desired.}\\

    {
  So let $\C = (\overline{c}_0, \overline{c}_1)$ satisfy the hypothesis of \ref{sink3} (from \cref{sinks}), i.e.\ %
$\overline{c}_0 \cap \overline{c}_1 \subseteq {\Omega'}$.}
  	We claim that there is $\overline{c}_1'\equiv_{\overline{c}_0} \overline{c}_1$ such that $\overline{c}_1'\subseteq\Omega'$. Since $\Omega'$ is universally embedded, it suffices to show that $(\overline{c}_1; \overline{c}_1; \overline{c}_0)$ is an absorbing configuration for $\Omega'$. Write $\overline{d}_0:=\overline{c}_0\cap \Omega'$. Condition \ref{p1} is trivial, and conditions 
    \ref{p3} and \ref{p4}  are equivalent, respectively, to:
  	\begin{equation*}
\overline{c}_{0}\ind^{a}_{\overline{d}_0} \overline{c}_{1},\quad\text{ and } \quad \overline{d}_{0}\ind^{a}_{\overline{c}_{1}}\overline{c}_{1}\;.
  	\end{equation*}
    The first of the two conditions above follows from symmetry, base monotonicity, and the fact that $\overline{c}_0\ind^{a}_{\overline{c}_1\cap \overline{c}_0} \overline{c}_1$, since $\overline{c}_1\cap \overline{c}_0  \subseteq{\Omega'\cap \overline{c}_{0}=}\overline{d}_{0}
    $. 
    {The second condition is just an instantiation of existence.}
  \end{proof}

  Combining \cref{abstract minimality endomorphism monoids} with \cref{from base monotonicity to universally embedded} and \cref{universally embedded are sinks} we obtain:
  \begin{theorem}[\cref{mainthm: b}]\label{t:minimality endomorphism monoids}
  	Let $G$ be the automorphism group of a  countable saturated structure $\M$ with %
    locally finite {and modular} algebraic closure. 
    Then $\pw$ is a minimal Hausdorff semigroup topology on $\overline{G}$. 
  \end{theorem}

\section{Model-theoretic applications and examples}\label{sec:discussion}
{In this section, we discuss more concretely applications of our results in model-theoretic contexts. We begin
by providing some additional model-theoretic background in \cref{sub:modelbackground}. We then discuss applications of \cref{mainthm: b} and \cref{mainthm: c} in \cref{sub:mainapplications}. We conclude with \cref{sub:universalfromlovely}, where we discuss connections between universally embedded models and the more standard model-theoretic notion of a lovely pair~\cite{berenstein2010lovely}.}

\subsection{Model-theoretic background}\label{sub:modelbackground}

\subsubsection{Imaginaries}

{Let $\M$ be a countable saturated structure with domain $\Omega$. Let $(\sim_i)_{i\in I}$ be the collection of equivalence relations on finite powers of $\M$ which are definable without parameters. Letting $G=\Aut(\M)$, it is easy to see that $G$ acts naturally on the disjoint union}
\[{\M^{eq}:=\bigsqcup_{i\in I} \M^{n_i}/\sim_i}\;.\]
We denote by $\dcl^{eq}$ and $\acl^{eq}$ the notions of definable and algebraic closure arising from this action. {We identify $\M$ with $\M/=$. %
We call %
$\M$ the \textbf{real sort} of this action of $G$ and the other $\M^{n_i}/\sim_i$ 
\textbf{imaginary sorts}. We say that $G$ has \textbf{weak elimination of imaginaries} if for every %
$e\in\M^{eq}$
there is some finite tuple $b$ from %
$\M$ such that $e\in\dcl^{eq}(b)$ and $b\in\acl^{eq}(e)$. For more on imaginaries and weak elimination of imaginaries, the reader may consult~\cite[$\S$ 16]{poizat2000course}.}

\subsubsection{Stability and simplicity}

Independence relations play a central role in the study of \textbf{stable} theories and of the strictly larger class of (model-theoretically) \textbf{simple} theories. Stability and simplicity are combinatorial properties of a first-order theory which can be equivalently characterised by the existence of an independence relation (in this case defined on sufficiently small subsets of a sufficiently saturated model) satisfying certain axioms (see \cref{rem:independenceaxioms}). We refer the reader to~\cite{tent2012course} for their formal definitions, as well as that of \textbf{non-forking independence}, a well-behaved notion of independence in the simple context, and of one-basedness, which corresponds to non-forking independence coinciding with an analogue of algebraic independence after having expanded the original structure by hyperimaginary elements (a generalisation of imaginaries). What will be important to us is that in the contexts we study these model-theoretic properties guarantee the existence of independence relations satisfying certain tameness conditions:

\begin{fact}\label{fact:modeltheoretic properties} Let $G\acts\Omega$ be the automorphism group of a countable saturated structure $\M$ with locally finite algebraic closure and weak elimination of imaginaries. Then, 
\begin{itemize}
    \item if the theory of  $\M$ is stable, there is an independence relation on $\Omega$ (known as non-forking independence) satisfying base monotonicity and stationarity;
     \item if the theory of $\M$ is simple, then there is an independence relation on $\Omega$ (known as non-forking independence) satisfying base monotonicity and independent $3$-amalgamation.  If the theory of $\mathcal{M}$ is, moreover, one-based, then non-forking independence is algebraic independence (and so, strongly one-based). 
\end{itemize} 
\end{fact}
\begin{proof} This fact is mainly a consequence of the characterisation of non-forking independence in simple (and stable) theories (see, for example,~\cite[Section 3]{kim2013simplicity}), together with the fact that in simple and small theories, Lascar types agree with strong types (i.e., type over imaginary algebraically closed sets)~\cite[Proposition 2.1.19(2) and Corollary 5.3.5]{kim2013simplicity}, and that strong types agree with types {over algebraically closed sets} under the assumption of weak elimination of imaginaries. We refer the reader to~\cite[Section 3C]{ghaddln24} for a more detailed discussion. 
\end{proof}

\begin{remark}\label{rem:independenceaxioms}
As mentioned in the proof of \cref{fact:modeltheoretic properties}, model-theoretic simplicity and stability can be characterised by the existence of an independence relation on (sufficiently) small subsets of a sufficiently saturated model satisfying  certain axioms. These are essentially those of Definition~\ref{d:indipendence basic} together with base monotonicity, two other axioms (finite character and local character), and a generalisation of either stationarity or independent $3$-amalgamation, depending on whether we are characterising, respectively, stable, or simple theories. For more on model-theoretic stability and simplicity, the reader may also consult~\cite{pillay1996geometric, baldwin2017fundamentals, kim2013simplicity}. 
\end{remark}

\subsection{On the scope of \cref{mainthm: b} and \cref{mainthm: c}}\label{sub:mainapplications}

Both Theorem~\ref{thm:onebased} and Theorem~\ref{t:minimality endomorphism monoids} apply to a wide range of structures. We discuss these in this subsection.\\

Firstly, from \cref{fact:modeltheoretic properties}, both Theorem~\ref{thm:onebased} and Theorem~\ref{t:minimality endomorphism monoids} apply to one-based countable saturated structures with locally finite algebraic closure and weak elimination of imaginaries.

\begin{theorem}\label{thm: simple theorem} Let $G=\Aut(\M)$ where $\M$ is the countable saturated model of a simple one-based theory with {locally finite algebraic closure and} weak elimination of imaginaries. Then, 
\begin{itemize}
    \item $\pw$ is minimal amongst Hausdorff semigroup topologies on $\overline{G}=\EEmb(\M)$;
    \item every semigroup topology on $G$ coarser than $\pw$ is a generalised pointwise topology. Moreover, if $G$ is transitive, $\pw$ is minimal amongst {$T_0$} semigroup topologies on $G$. 
\end{itemize}
\end{theorem}
\begin{proof}%
We shall use \cref{fact:modeltheoretic properties}. Since we assumed that $\M$ is one-based, we have that non-forking independence is algebraic independence. Hence, algebraic independence satisfies base monotonicity, it is $1$-narrow (i.e., what we call strongly one-based), and satisfies independent $3$-amalgamation. From base monotonicity of algebraic independence we know that algebraic closure is modular. Hence, both \cref{t:minimality endomorphism monoids} and \cref{thm:onebased} apply. In the second case, note that,
as we argued in the proof of \cref{mainthm: c}, if $G$ is transitive, $\pw$ is minimal amongst {$T_0$} semigroup topologies on $G$.
\end{proof}

For many natural examples of $\omega$-categorical structures $\mathcal{M}$, whilst the natural action of $\mathrm{Aut}(\M)$ on $\mathcal{M}$ itself does not have weak elimination of imaginaries, we can find an expansion $\mathcal{M}^\star$ of $\mathcal{M}$ by finitely many imaginary sorts for which $\mathrm{Aut}(\M)\acts\mathcal{M}^\star$ has weak elimination of imaginaries. This property is sometimes called \textbf{very weak elimination of imaginaries}. Since for $\omega$-categorical structures simplicity, one-basedness, and locally finite algebraic closure are all maintained when adding finitely many imaginary sorts, both theorems apply to simple one-based $\omega$-categorical structures with very weak elimination of imaginaries. %

\begin{example} %
Smoothly approximable structures~\cite{cherlin2003finite}, which include all $\omega$-categorical $\omega$-stable structures~\cite{cherlin1985}, are a large class of simple one-based $\omega$-categorical structures~\cite[Proposition 5.5.2]{cherlin2003finite}. These structures are known to include various linear, projective, semiprojective, and affine geometries over finite fields (arising from vector spaces with forms)~\cite[Definitions 2.1.4, 2.1.7, and 2.1.8]{cherlin2003finite}, and  can be ``coordinatised'' by such Lie geometries~\cite[Theorem 2]{cherlin2003finite}. Some interesting structures in this class have weak elimination of imaginaries, such as basic linear and semiprojective geometries~\cite[Lemmas 2.3.5 and 2.3.6]{cherlin2003finite}.
This includes the infinite-dimensional versions of the geometries giving rise to the classical groups, namely the infinite-dimensional vector,  symplectic, unitary, and orthogonal spaces over finite fields %
(cf.~\cite{evans1991small}). \\
{Other smoothly approximable structures, such as basic affine geometries (not of quadratic type), have %
very weak elimination of imaginaries~\cite[Lemma 2.3.12]{cherlin2003finite}.
Our understanding is that it remains an open problem whether every smoothly approximable structure has very weak elimination of imaginaries.}%
\end{example}

\begin{example} %
From the recent proof of Koponen's conjecture~\cite{koponen2016binary} in~\cite{baldwin2024simplehomogeneousstructuresindiscernible}, we know that all simple finitely homogeneous structures are one-based. (This was already known for stable finitely homogeneous structures from~\cite{cherlin1985}.) So both Theorem~\ref{thm:onebased} and Theorem~\ref{t:minimality endomorphism monoids} apply to all such structures with   very weak elimination of imaginaries. It is an open problem whether all finitely homogeneous structures have very weak elimination of imaginaries~\cite[Problem 18]{cherlin2022homogeneous}. Standard examples of stable finitely homogeneous structures are $(\mathbb{N}; =)$, all countable structures in a finite unary language, and countably infinitely many disjoint copies of $K_n$. 
Examples of simple non-stable finitely homogeneous structures are the random graph, the random bipartite graph, and the generic tetrahedron-free $3$-hypergraph (the Fra\"{i}ss\'{e} limit of all finite uniform $3$-hypergraphs omitting a clique on four vertices). %
\end{example}

\begin{example} We know from~\cite[Theorem 4.1.1]{onshuus2021dependent} that dependent finitely homogeneous rosy structures with weak elimination of imaginaries have modular algebraic closure. Accordingly, \cref{t:minimality endomorphism monoids} applies to all dependent finitely homogeneous rosy structures with very weak elimination of imaginaries. Dependent (a.k.a. $\mathrm{NIP}$) theories form another heavily studied class generalising stable theories. Roughly, these are defined by the lack of a definable relation $\phi(\overline{x}; \overline{y})$ coding all finite graphs. Prototypical examples of $\omega$-categorical structures with $\mathrm{NIP}$ unstable theories are $(\mathbb{Q}; <)$ and Droste’s almost-homogeneous semilinear orders~\cite{droste1985structure}. Stable theories correspond to the class of theories which are both dependent and simple. We refer the reader to~\cite{simon2015guide} for more on dependent theories. Rosiness is an additional model-theoretic property that is defined in terms of a certain generalisation of forking (known as \th-forking, where ``\th'' is read as ``thorn'') yielding an independence relation (known as \th-independence) on a sufficiently saturated model of a theory (see~\cite{adler2009geometric} for details). In the dependent setting, rosiness excludes certain tree-like configurations (cf.~\cite{SimonPierre2022N}). Dependent finitely homogeneous rosy structures include all dependent binary homogeneous structures and are studied in~\cite{onshuus2021dependent, SimonPierre2022N}. As usual, we require the additional condition of very weak elimination of imaginaries. Nevertheless, as we pointed out above, it is not known whether the latter holds for all finitely homogeneous structures.
\end{example}

\begin{remark}Note that \cref{thm:onebased} also applies to simple $\omega$-categorical structures beyond the one-based context. In particular,~\cite[Section 3E]{ghaddln24} shows that a large class of simple $\omega$-categorical Hrushovski constructions satisfies $3$-narrowness. The method of universally embedded models does not straightforwardly apply to such structures; though we would expect adequate generalisations of our methods to yield minimality of $\pw$ for monoids of elementary embeddings of adequately large classes of $\omega$-categorical Hrushovski constructions for which algebraic independence does not satisfy base monotonicity. 
\end{remark}

 \begin{example} {The main advantage of using universally} embedded models 
 in our proof of \cref{t:minimality endomorphism monoids} is that, unlike \cref{thm:onebased}, we are not relying on stationarity or amalgamation properties that would constrain us to some particular area of the model-theoretic universe (such as that of simple theories). Firstly, algebraic closure is modular in any countable saturated structure whose automorphism group has no algebraicity, thus recovering the result of~\cite{pinskerschindler} that $\pw$ is minimal amongst Hausdorff semigroup topologies in this setting. Moreover, algebraic closure is also modular when it is degenerate (i.e., the algebraic closure of any set is the union of the algebraic closures of its points). This is, for example, the case for infinitely many copies of $K_n$. Finally, for a more interesting example, let $\M$ be a homogeneous structure with no algebraicity in the language $\mathcal{L}$, and let $\mathcal{N}$ be a countably infinite vector space over the finite field $\mathbb{F}_q$.
 Consider the free superposition $\M'$ of $\M$ and $\mathcal{N}$. This is just the Fra\"{i}ss\'{e} limit of finite-dimensional vector spaces over $\mathbb{F}_q$ with an underlying $\mathcal{L}$-structure in the age of $\M$.
{Then, algebraic closure for $\M'$ is modular since it agrees with algebraic closure in $\mathcal{N}$. Hence, 
\cref{from base monotonicity to universally embedded} and \cref{t:minimality endomorphism monoids} apply to $\M'$.}
 \end{example}

\subsection{Universally embedded models and lovely pairs}\label{sub:universalfromlovely}

{In this section, we note that lovely pairs are a natural source of universally embedded models. Lovely pairs are the standard model-theoretic notion of a ``generically embedded'' submodel. 
Our definition of a lovely pair given below in the context of \ref{sat} agrees with the standard definition of a lovely pair for a geometric theory~\cite[Definition 2.3]{berenstein2010lovely}; we clarify this in \cref{lem: lovely correspondence}. Lovely pairs of models are  studied extensively in simple theories~\cite{ben2003lovely} (where the definition uses non-forking rather than algebraic independence), and in particular in stable theories, where they appear as \textbf{belles paires} from work of Poizat~\cite{poizat1983paires}. In the o-minimal context they are studied as \textbf{dense pairs}~\cite{van1998dense}.}

\begin{definition}\label{def:lovely} Assume \ref{sat}. We say that $(\Omega'; \Omega)$ forms a \textbf{lovely pair} if $\Omega'\subseteq\Omega$ is an algebraically closed subset of $\Omega$ satisfying the coheir property {and the following property:}
\begin{itemize}
    \item (extension) for all finite algebraically closed tuples $\overline{a}, \overline{b}$ from $\Omega$, %
      there is $\overline{a}'\equiv_{\overline{b}} \overline{a}$ such that
   \[{\acl(\overline{a}'\overline{b})\cap \acl(\Omega'\overline{b})=\overline{b}\;.}\]
\end{itemize}
 (If algebraic independence were extended to infinite algebraically closed sets, the conclusion could be phrased as $\overline{a}' \ind^{a}_{\overline{b}} \Omega'$.)
\end{definition}

\begin{lemma}\label{lem: lovely pairs} Assume \ref{sat} {and that algebraic closure is modular. Let $(\Omega';\Omega)$ be a lovely pair.}
Then $\Omega'$ is a universally embedded model.
\end{lemma}
\begin{proof}
    Let $(\overline{a}_1; \overline{a}_2; \overline{b})$ be an absorbing configuration for $\Omega'$. By the coheir property, there is  $\overline{a}_1'\equiv_{\overline{b}}\overline{a}_1$ with $\overline{a}_1'\subseteq\Omega'$. 
    Find $\overline{a}_2'$ such that $\overline{a}_1'\overline{a}_2'\equiv_{\overline{b}}\overline{a}_1\overline{a}_2$. {Let $\b_0:=\Omega'\cap \b$. By invariance, we have $\overline{a}_2'\ind^{a}_{\overline{a}_1'}\overline{b}_0$.}
    
{ Let $C:=\acl(\overline{a}_1'\overline{b})$. By the extension property (Definition \ref{def:lovely}),  }
there is $\overline{a}_2''\equiv_C\overline{a}_2'$ with
    \begin{equation}\label{eq:intersection for d}
        \acl(\overline{a}_2'' C)\cap\acl(\Omega'C)=C\;.
    \end{equation}

    { We verify that $\overline{a}''_2 \cap \Omega' = \overline{a}'_1$. For the
  inclusion $\subseteq$, let $d \in \overline{a}''_2 \cap \Omega'$. %
  By
  \eqref{eq:intersection for d}, $d \in C$, so $d \in C \cap \Omega'$. 

  {Let $E:=\acl(\a'_1 \b_0)$, $F:=\acl(\a'_1 b_0 d)$. Now,}
  \[C\cap F=\acl(E\b)\cap F=\acl(E\cup(\b\cap F))=\acl(E\cup\b_0)=E\;,\]
  {where for the intermediate equality we applied modularity from Proposition~\ref{prop:modularity} (for $E\subseteq F$ in the place of $A$ and $B$ and $\b$ in the place of $C$). Hence, since $d\in C\cap F$, $d\in E$.}
  {Since $\overline{a}''_2 \equiv_C \overline{a}'_2$ and $\overline{a}'_2 \ind^{a}_{\overline{a}'_1} \overline{b}_0$
  invariance gives  $\overline{a}''_2 \ind^{a}_{\overline{a}'_1} \overline{b}_0$, and  monotonicity gives $d \ind^{a}_{\overline{a}'_1} \overline{b}_0$. As   $d \in E$,  the definition of $\ind^{a}$  forces $d \in \overline{a}'_1$.}} The inclusion $\supseteq$ is straightforward: from
  $\overline{a}'_1 \overline{a}'_2 \equiv_{\overline{b}}
  \overline{a}_1 \overline{a}_2$ and $\overline{a}_1 \subseteq \overline{a}_2$
  we get $\overline{a}'_1 \subseteq \overline{a}'_2$; since
  $\overline{a}'_1 \subseteq C$ and $\overline{a}''_2 \equiv_C \overline{a}'_2$, we conclude that  $\overline{a}'_1 \subseteq \overline{a}''_2$ and
  $\overline{a}'_1 \subseteq \Omega'$.  Finally, $\overline{b} \subseteq C$ gives
  $\overline{a}'_1 \overline{a}''_2 \equiv_{\overline{b}}
  \overline{a}'_1 \overline{a}'_2 \equiv_{\overline{b}}
  \overline{a}_1 \overline{a}_2$, completing the proof.
\end{proof}

Since our setting differs slightly from that
of~\cite{berenstein2010lovely}, the next lemma reassures the reader that
countable saturated models of geometric theories with locally finite
algebraic closure do have lovely pairs. Recall that a first-order theory is \textbf{geometric} if it eliminates the quantifier $\exists^{\infty}$ (in the sense of~\cite[Definition 5.5.6]{tent2012course}) and model-theoretic algebraic closure~\cite[Definition 5.6.1]{tent2012course} forms a pregeometry~\cite[Definition 5.6.5]{tent2012course}. We remark that every $\omega$-categorical theory eliminates $\exists^\infty$.

\begin{lemma}\label{lem: lovely correspondence} Let $\M$ be a countable saturated model of a geometric theory $T$ and suppose that $\M$ has locally finite algebraic closure. Then there is some $\M'\preceq\M$ such that $(\M'; \M)$ forms a lovely pair (with respect to the action $\Aut(\M)\acts \Omega$).
\end{lemma}
\begin{proof} Since $\M$ is countable and saturated we may assume the language $\mathcal{L}$ of $T$ is countable {and that $T$ has quantifier elimination (see, for example,~\cite[Lemma 2.6]{pech2018reconstructing}}). By~\cite{berenstein2010lovely}, $T$ has lovely pairs in the sense of~\cite[Definition 2.3]{berenstein2010lovely}, and the theory $T_P$ of lovely pairs of models of $T$  (in the language $\mathcal{L}_P$ adding to $\mathcal{L}$ a predicate $P$ for the smaller model) is  axiomatised in~\cite[Theorem 2.10]{berenstein2010lovely}. Inspecting the proof of the second claim of~\cite[Theorem 2.10]{berenstein2010lovely}, one sees that in our context (locally finite algebraic closure) every countable saturated model of $T_P$ is itself a lovely pair. {It therefore suffices to show that  $T_P$ has a countable saturated model, which for a countable theory is  equivalent to having only countably many $n$-types for each  $n$~\cite[Lemma~4.3.9]{tent2012course}.}%

For this, we move to the expansion $T'_P$ in the language $\mathcal{L}_P'$ of $T_P$ which is obtained by adding, for each $\mathcal{L}$-formula $\phi(\overline{x}; \overline{y})$, a relation symbol $R_{\phi(\overline{x}; \overline{y})}$ together with the axiom:
\[\forall \overline{y}\left( R_{\phi(\overline{x}; \overline{y})}(\overline{y})\leftrightarrow \exists \overline{x} P(\overline{x})\wedge \phi(\overline{x}; \overline{y})\right)\;.\]
{By~\cite[Corollary~3.2]{berenstein2010lovely},
  $T'_P$ has quantifier elimination. 
  {We will show that} for a finite algebraically
  closed tuple $\overline{a}$ in a sufficiently saturated model
  $(\mathcal{N}'; \mathcal{N}) \models T'_P$, the $\mathcal{L}'_P$-type of
  $\overline{a}$ is determined by its $\mathcal{L}$-type together with the information of which coordinates of $\overline{a}$ lie in $P$. {To see this, consider an $\mathcal{L}$-formula $\phi(\overline{x};\overline{y})$ and a tuple $\a$ from $\mathcal{N}$. Now, take $\b$ such that $\phi(\b; \a)$ holds in $\mathcal{N}$, and let $\a_P:=\a\cap\mathcal{N}'$. If $\b\ind^a_{\a_P}\a$, by the coheir property we can deduce that $R_{\phi(\overline{x};\overline{y})}(\a)$ holds. So, we may assume that for any $\b$ such that $\phi(\b; \a)$ holds, there is some $c\in(\acl(\b\a_P)\cap\a)\setminus\a_P$. In particular, $c\not\in\mathcal{N}'$. But then we know that $\neg R_{\phi(\overline{x};\overline{y})}(\a)$ holds: if there was any $\b$ such that $\b\subseteq \mathcal{N}'$, since $\mathcal{N}'$ is algebraically closed (being a model), we would also have $c\in\mathcal{N}'$. Thus, we can conclude that whether $R_{\phi(\overline{x}; \overline{y})}(\overline{y})$ holds of a given tuple is entirely determined by its quantifier-free $\mathcal{L}$-type together with the information regarding which of its coordinates lie in $P$.} 
  
  By local finiteness,  every $n$-tuple lies in a finite algebraically closed tuple, so each $\mathcal{L}$-type of an $n$-tuple in $T$ admits only finitely many   extensions to an $\mathcal{L}'_P$-type in $T'_P$. Hence $T'_P$, and   therefore $T_P$, has only countably many $n$-types for each   $n \in \mathbb{N}$, as required.}
\end{proof}

\section{ Minimal semigroup topologies in Urysohn spaces }\label{sec:uri}

So far, we obtained results showing that $\pw$ is minimal amongst Hausdorff semigroup topologies in several monoids of elementary embeddings. Indeed, various sources of failures of minimality of $\pw$ for groups do not work in the context of monoids. In this section, we study semigroup topologies on the monoids of isometric embeddings of Urysohn spaces. These are homogeneous metric spaces which will yield examples of semigroup topologies for which $\pw$ is not minimal (amongst Hausdorff semigroup topologies). In fact, the minimal Hausdorff semigroup topology for these structures naturally arises from the presence of a metric and is known as the \textbf{metric pointwise topology} $\mpw$. We will show minimality of $\mpw$ for the spaces of isometric embeddings of various Urysohn spaces by proving these coincide with the Zariski topology $\zar$. Our techniques can be seen as adapting the ideas of~\cite{pinskerschindler} to a metric context.\\

The Urysohn universal metric space $\U_{\mathbb R}$, also known as the \textbf{real Urysohn space},  constructed by Urysohn in 1925, is a complete separable metric space uniquely characterised by {\bf universality} (every separable metric space embeds isometrically into $\U_{\mathbb R}$) and {\bf %
homogeneity} (every partial isometry between finite subspaces extends to a global isometry). 
It was originally constructed by taking the completion of what we now call the rational Urysohn space  $\U_{\mathbb{Q}_{\geq 0}}$ . This countable rational metric space, although it can be identified today as the Fraïssé limit of the class of finite metric spaces with rational distances, was obtained by Urysohn through a direct constructive method. 
There is also a bounded version of the Urysohn space, known as the \textbf{real Urysohn sphere} $\U_{[0,1]}$, which is universal and homogeneous with respect to the class of  finite metric spaces of diameter at most $1$. 
Again, this can be obtained as the completion of the \textbf{rational Urysohn sphere} $\U_{\mathcal{Q}_1}$, which is the Fra\"{i}ss\'{e} limit of finite metric spaces with rational distances bounded by $1$. {For the aforementioned Urysohn spaces $\U$,} it is easy to see that the spaces of isometries
{$\mathrm{I}(\U):=\Aut(\U)$ and {of} isometric embeddings $\ie(\U):=\overline{\Aut(\U)}$}
are naturally endowed with a metric pointwise topology induced by the product topology on $\U^\U$, where $\U$ is endowed with the topology arising from its metric.\\

The isometry groups $\mathrm{I}(\U_{\mathbb R})$ and $\mathrm{I}(\U_{[0,1]})$, equipped with $\mpw$, are Polish groups whose algebraic and dynamical properties have been extensively studied. Notably, every Polish group embeds as a closed subgroup of $\mathrm{I}(\U_{\mathbb R})$ \cite{uspenskij1990group}, and $\mathrm{I}(\U_{\mathbb R})$  is extremely amenable~\cite{pestov2002ramsey}. The spaces of isometric embeddings have also been studied in a topological context, with, for example,~\cite{kubis2017katetov} showing that every transformation monoid on a countable set embeds in the space of isometric embeddings of $\U_{\mathbb{Q}_{\geq 0}}$.\\

The model-theoretic construction of  $\U_{\mathbb R}$ and $\U_{[0,1]}$  extends to a broader framework based on so-called distance monoids: for %
{several distance monoids (including all countable ones),} the associated class of finite metric structures admits a Fra\"iss\'e limit, yielding a universal and homogeneous structure. Details of this setting are provided below and may also be found in~\cite{con15}.\\

The minimality of $\mpw$ on $\mathrm{I}(\U_{[0,1]})$  was established in~\cite{uspenskij2008subgroups} {using Roelcke precompactness of the isometry group of the Urysohn sphere.}
In~\cite{gonzalesghadernezhad} the minimality of $\mpw$ on $\mathrm{I}(\U_{\mathbb R})$ and for other generalised Urysohn spaces was obtained through a different approach. In what follows, we study the monoids of isometric embeddings of generalised Urysohn spaces. These are also the monoids of elementary embeddings of these structures. We prove that $\mpw=\zar\subsetneq\pw$ for the monoids of isometric embeddings of $\U_{\mathbb Q_{\geq 0}}$, $\U_{\mathbb R}$, $\U_{[0,1]}$, and $\U_{\mathcal{Q}_1}$. We will give most proofs in the broader context of generalised Urysohn spaces~\cite{con15}.

\subsection{Generalised Urysohn spaces}

We work within the framework of generalised metric spaces introduced by~\cite{con15}. This setting encompasses both the rational numbers and bounded rationals, which are used in the construction of  $\U_{\mathbb R}$  and $\U_{[0,1]}$.

A {\bf distance monoid} $\rr=(R,\leq,\oplus,0)$ is a commutative monoid $(R, \oplus, 0)$ endowed with a total order $\leq$ which is compatible with the monoid structure by satisfying the following axioms: 

\begin{itemize}
	\item (positivity) for all  $s, r\in R, \ r\leq r\oplus s $;
	\item (order) for all $s, r, t\in R$, if $s\leq r$, then $s\oplus t\leq r\oplus t$.
\end{itemize}

\begin{definition} A distance monoid $\rr$ is \textbf{metrically complete} if for any $s\leq r\in R$, the set $\{t\in R\,|\,r\leq t\oplus s\}$ is of the form $\{t\,|\,t_{0}\leq t\}$ for some $t_{0}\in R$, which we denote by $r\ominus s$.
\end{definition}

\begin{fact}\label{fact:ominus} Suppose that $\rr$ is a metrically complete distance monoid. For all $r,s,t,u\in R$ with $s\leq r$, we have
\begin{enumerate}
    \item $r\ominus s\leq t$ if and only if $r\leq s\oplus t$;
    \item suppose that $s\leq t\oplus u$. Then, $r\leq t\oplus u\oplus (r\ominus s)$.
   
\end{enumerate}
\end{fact}
\begin{proof} The first item is just the definition. For (2), note that by the first item, $r\leq s\oplus (r\ominus s)$. Hence, we have
\[r\leq s\oplus (r\ominus s)\leq t\oplus u\oplus (r\ominus s)\;.\]
\end{proof}

\begin{definition}
Given a distance monoid $\rr:=(R,\leq,\oplus,0)$, an $\bm{\rr}$\textbf{-metric space} $(\mathcal{X},d)$ consists of a set $X$ together with  a  map $d:X^{2}\to R$ such that for all $x,y,z\in X$:
\begin{enumerate}
	\item $d(x,y)=0$ if and only if  $x=y$;
	\item $d(x,y)=d(y,x)$;
	\item $d(x,z)\leq d(x,y)\oplus d(y,z)$.
\end{enumerate}
\end{definition}
\begin{definition}
     Let $(\mathcal{X}, d)$ and $(\mathcal{Y}, d)$ be $\rr$-metric spaces. An \textbf{isometric embedding}
    is a map $\phi:X\to Y$ such that for all $a,b\in X$
 \[d(\phi(a), \phi(b))=d(a,b)\;.\] 
If an isometric embedding is a bijection we call it an \textbf{isometry}. 
We denote by $\ie(\mathcal{X})$ the space of \textbf{isometric embeddings} of $\mathcal{X}$ into itself and by $\mathrm{I}(\mathcal{X})$ its space of isometries (into itself). 
\end{definition}

It is easy to see that the space of isometric embeddings $\ie(\mathcal{X})$ of an $\rr$-metric space $\mathcal{X}$ equals the space of elementary embeddings of $\mathcal{X}$ whenever the first-order theory of $\mathcal{X}$ has quantifier elimination in the language $\mathcal{L}_{\rr}:=\{D_r(x,y)\ \vert r\in\rr\}$ where each binary relation $D_r(x,y)$ is interpreted as the set $\{(a,b)\in X^2\ \vert d(a,b)\leq r\}$. Clearly, every elementary embedding in the above language is an isometric embedding. And quantifier elimination ensures that every isometric embedding is indeed elementary. We recall that a first-order theory $\mathcal{T}$ has quantifier elimination whenever every first-order formula is equivalent modulo $\mathcal{T}$ to a quantifier-free formula.

\begin{definition}
	\label{def-un}
Suppose $\rr$ is a distance monoid. An $\rr$-{\bf Urysohn space} is an $\rr$-metric space $\U$ satisfying the following two properties:
\begin{itemize}
	\item  (universality)\label{universality} any finite $\rr$-metric space embeds in $\U$; 
	\item (homogeneity) \label{ultrahomogeneity} every isometry between finite subspaces of $\U$ extends to an isometry of $\U$.
\end{itemize}

\end{definition}
\begin{remark}{By standard Fra\"{i}ss\'{e} theory, it is easy to see that universality and homogeneity of an $\rr$-Urysohn space $\U$ jointly imply the \textbf{extension property}:}
	\begin{enumerate}
		\item[(EP)] \label{extensionproperty} For any finite $\rr$-metric space $B$ and $A\subseteq B$, any isometric embedding
		$h:A\to\U$ extends to some isometric embedding $\bar{h}:B\to\U$. 
	\end{enumerate}
   {Indeed, the extension property is equivalent to the conjunction of universality and homogeneity under the assumption that $\U$ is countable.}
\end{remark}

\begin{fact}[{\cite[Theorem 2.7.7]{con15}}] \label{fact:existenceUri} Let $\rr$ be a countable distance monoid. Then, there is a countable $\rr$-Urysohn space $\U_{\rr}$. Moreover, $\U_{\rr}$ is unique up to isomorphism.
\end{fact}

\begin{definition} By a \textbf{standard} distance monoid $\rr$ we mean a distance monoid 
such that whenever $r\in R\setminus\{0\}$ and $s\in R$ is such that $s<r$, there is $t\in R\setminus\{0\}$ such that $s\oplus t<r$. A \textbf{standard $\rr$-metric space} is an $\rr$-metric space $(\mathcal{X}, d)$, where $\rr$ is a standard distance monoid.
\end{definition}

Note that standardness of $\rr$ yields a form of density of the order of our distance monoid around the origin: for any $0<r\in\rr$, we can find some $t\in\rr$ such that $0<t<r$. This follows from taking $s=0$ in the definition above.\\

\cite[Theorem 2.8.11]{con15} also gives necessary and sufficient conditions for the first-order theory of an $\rr$-Urysohn space to have quantifier elimination. Since these conditions are slightly technical, we just note that the Urysohn spaces we give in the examples below are such that their first-order theories have quantifier elimination. 

\begin{examples}\label{examples:uri} We proceed here by giving some examples of distance monoids $\rr$ with their respective Urysohn spaces $\U_\rr$. 
\begin{itemize}
    \item Clearly, {$(\mathbb{Q}_{\geq 0}, \leq, \oplus, 0)$} is a metrically complete standard distance monoid. By \cref{fact:existenceUri}, there is a rational Urysohn space $\U_{\mathbb{Q}_{\geq 0}}$. Its completion is the  real Urysohn space $\U_{\mathbb{R}}$. Clearly, $\mathbb{R}_{\geq 0}$ is still metrically complete and standard. Moreover, both $\U_{\mathbb{Q}_{\geq 0}}$ and $\U_{\mathbb{R}}$ have quantifier elimination;
    \item Consider $\mathcal{Q}_1:=(\mathbb{Q}\cap[0,1], \leq, +_1, 0)$, where $+_1$ is truncated addition. Again, this is a metrically complete standard distance monoid. Its Urysohn space is known as the rational Urysohn sphere $\U_{\mathcal{Q}_1}$. Its completion is the Urysohn sphere $\U_{[0,1]}$, for the metrically complete standard distance monoid $([0,1], \leq +_1, 0)$. Again, both of these Urysohn spaces have quantifier elimination;
\item We can work in the following more general setting. Let $(G, \leq, +, 0)$ be an ordered Abelian group. Then, $(G_{\geq 0}, \leq, +, 0)$ , where $G_{\geq 0}:=\{g\in G\vert g\geq 0\}$ is a distance monoid. A \textbf{convex} distance monoid $\rr$ is a distance monoid $(I\cup\{0\}, \leq, \oplus, 0)$ whose domain is $I\cup\{0\}$, where $I$ is a convex subset of $G_{\geq 0}$ and where for $r,s\in R$, we define $r\oplus s:=\min\{r+s, \infty_R\}$, where $\infty_R$ is the $\leq$-maximal element of $R$ (if there is one). For any countable convex distance monoid $\rr$, $\U_{\rr}$ has quantifier elimination~\cite[Proposition 2.9.4]{con15}. Choosing $G$ to be $\mathbb{Q}$ and $\rr:=\mathbb{Q}_{\geq 0}$ or $\mathbb{Q}\cap[0,1]$, we get, respectively, the rational Urysohn space and the rational Urysohn sphere. However, note that we can pick $G$ to also be $\mathbb{Q}\times \mathbb{Q}$ ordered lexicographically (and $\rr:=G_{\geq 0}$) and again, obtain a metrically complete standard distance monoid. In particular, we are able to work with $\rr$-metric spaces where the elements of $\rr$ may not form a single Archimedean class, where $r,s\in R$ are in the same Archimedean class if there is some positive integer $n$ such that $s\leq n\cdot r$ and $r\leq n\cdot s$.
\item We say that the distance monoid $\rr$ is \textbf{ultrametric} if for all $r,s\in R$, we have that $r\oplus s=\max\{r,s\}$. For every countable ultrametric distance monoid $\rr$, $\U_{\rr}$ has quantifier elimination~\cite[Proposition 2.9.4]{con15}. {So, the rational ultrametric Urysohn space $\U_{(\mathbb{Q}_{\geq 0}, \leq,  \max,0)}$ has quantifier elimination.}
 {We also know that the real ultrametric Urysohn} space {$\U_{(\mathbb{R}_{\geq 0}, \leq, \max,0)}$} has quantifier elimination.  Note that the theory of $\U_{\rr}$ where $\rr$ is an ultrametric space is essentially the theory of infinitely refining equivalence relations indexed by $(R, 0, \leq)$. These are important theories from the point of view of model theory~\cite{baldwin2017fundamentals}. {We note that the fact that we allow for $r\oplus s=s$ in our arguments will require us to be especially careful when choosing elements in $\rr$ as radii of open balls for our proofs.}
\end{itemize}  
\end{examples}

\subsection{Relating the Zariski and metric pointwise topology}\label{sub:relating zar and met}

In this section we give general conditions for a 
{submonoid}
of a space of isometric embeddings $\ie(\mathcal{X})$ to be such that the Zariski topology agrees with the metric pointwise topology $\mpw$. 
This is a natural topology on $\ie(\mathcal{X})$ which we introduce below and which is usually coarser than the topology of pointwise convergence. An interesting feature of our method is that, {besides} being the key to proving that $\mpw=\zar$ for the monoids of elementary embeddings of various Urysohn spaces, it also subsumes the technique underlying the result of Pinsker and Schindler~\cite{pinskerschindler} that $\zar=\pw$ for the monoids of elementary embeddings of $\omega$-categorical structures with no algebraicity.

\begin{remark} Note that any standard $\rr$-metric space $(\mathcal{X}, d)$ can be topologised with the generalised metric topology $\tau_{\mathrm{m}}$, whose basis of open sets is given by the set of open balls of the form 
\[B_\epsilon(a):=\{b\in X\ \vert \ d(a,b)<\epsilon\}\;,\]
for $a\in X$ and $\epsilon\in R\setminus\{0\}$. It is easy to see that $\tau_{\mathrm{m}}$ is a Hausdorff topology:  whenever $d(c,d)>0$ for $c,d\in\mathcal{X}$, since $\rr$ is standard, we can pick $\delta, \gamma \in R\setminus\{0\}$
 such that $\delta\oplus\gamma<d(c,d)$, and so $B_\delta(c)\cap B_\gamma(d)=\emptyset$. This can be seen  since if there was any $e\in B_\delta(c)\cap B_\gamma(d)$, we would have
 \[d(c,d)\leq d(c,e)\oplus d(e,d)<\delta\oplus\gamma<d(c,d)\;,\]
 leading to a contradiction. 
 \end{remark}

 \begin{definition}The \textbf{metric pointwise topology} $\mpw$ on $\ie(\mathcal{X})$ is obtained by taking the product topology on $X^X$, where each $X$ is endowed with $\tau_{\mathrm{m}}$, and then taking the induced subspace topology on $\ie(\mathcal{X})$. It is generated by open sets of the following form:
\[W_{(a, b, \epsilon)}:=\{s\in {\ie(\mathcal{X})}
\;
\vert \; s(a)\in B_\epsilon(b)\},\]
where $a,b\in X$  and  $\epsilon\in R \setminus\{0\}$. 
\end{definition}

\begin{lemma}\label{lem:hausdorff} Let $(\mathcal{X}, d)$ be a standard $\rr$-metric space. Then $\ie(\mathcal{X})$ with $\mpw$ is a Hausdorff topological semigroup. 
\end{lemma}
\begin{proof} Since $\mathcal{X}$ is Hausdorff, Hausdorffness of $\ie(\mathcal{X})$ follows from the fact that products of Hausdorff spaces are Hausdorff, and  subspaces of Hausdorff spaces are Hausdorff. Let $X$ be the underlying set of $\mathcal{X}$.
For continuity of multiplication, let $s,t\in \ie(\mathcal X)$ with $st\in W_{(a,b,\epsilon)}$ where $a,b\in X$ and $\epsilon\in R\setminus\{0\}$.
Put $c:=t(a)$ and $\zeta:=d(s(c),b)$. {Since $\rr$ is standard, we can find $\rho\in R\setminus\{0\}$ such that $\zeta\oplus\rho<\epsilon$. Then choose $\delta,\eta\in  R\setminus\{0\}$}
{such that $\eta\oplus\delta<\rho$. We now claim $W_{(c, s(c), \eta)}\cdot W_{(a, c, \delta)} \subseteq W_{(a,b,\epsilon)}$. {This suffices to establish continuity of multiplication at the pair $(s,t)$,}
since clearly $t\in W_{(a, c, \delta)}$ and $s\in W_{(c, s(c), \eta)}$. 
Let $t'\in W_{(a, c, \delta)}$ and $s'\in W_{(c, s(c), \eta)}$. Then
\begin{align*}
  d(s't'(a), b) &\leq d(s't'(a), s'(c))\oplus d(s'(c), s(c))\oplus d(s(c),b) \\
  &= d(t'(a), c)\oplus d(s'(c), s(c))\oplus d(s(c),b)\\
  &<\delta \oplus \eta\oplus \zeta<\rho\oplus\zeta<\epsilon\;,
\end{align*}
where the equality holds because $s'$ is an isometric embedding, and the final inequality follows from the choices of $\eta$,$\delta$ and $\rho$. Hence $s't'\in W_{(a,b,\epsilon)}$.}
\end{proof}

\begin{notation} {Let $(\mathcal{X}, d)$ be a standard $\rr$-metric space. For $A\subseteq X$ and $\epsilon\in\rr\setminus\{0\}$, we write}
\[B_\epsilon(A):=\bigcup_{a\in A}B_\epsilon(a)\;.\]   
\end{notation}

\begin{definition} Let $(\mathcal{X}, d)$ be a standard $\rr$-metric space. Let $\mathcal{T}\subseteq \rr$ be a dense submonoid of $\rr$ with respect to the order. Let $S$ be a closed submonoid of $\ie(\mathcal{X})$ with respect to $\mpw$.  
\begin{itemize}
    \item We say that $S$ is \textbf{pinching} (with respect to $\mathcal{T}$) if for every %
    $\epsilon \in \mathcal{T}\setminus\{0\}$ and every $a\in X$, there are $\phi, \psi\in S$ such that 
\begin{equation}\label{niceUri}
\phi\upharpoonright_{X\setminus B_{\epsilon}(a)}=\psi\upharpoonright_{X\setminus B_{\epsilon}(a)} \text{ and } \ \phi(b)\neq\psi(b) \text{ for all } b\in B_{\epsilon}(a).
\end{equation}
See Figure~\ref{fig:pinch} for an illustration of the behaviour of $\phi$ and $\psi$.
\item We say that $S$ is \textbf{spreading} (with respect to $\mathcal{T}$) if for every %
$\epsilon\in \mathcal{T}\setminus\{0\}$ and every $a\in X$, there are $\sigma, \theta\in S$ such that 
\begin{equation}\label{eq:ballcontainments}
a\in\im(\sigma)\cap \im(\theta) \text{ and } \im(\sigma)\cap B_\epsilon(\im(\theta))\subseteq B_{\epsilon}(a). 
\end{equation}
See Figure~\ref{fig:spread} for an illustration of the behaviour of $\sigma$ and $\theta$.
\end{itemize}
    We say that $S$ is pinching (respectively, spreading) when it is pinching with respect to $R$ itself. 
\end{definition}

\begin{figure}
    \centering
    \includegraphics[width=0.7\textwidth]{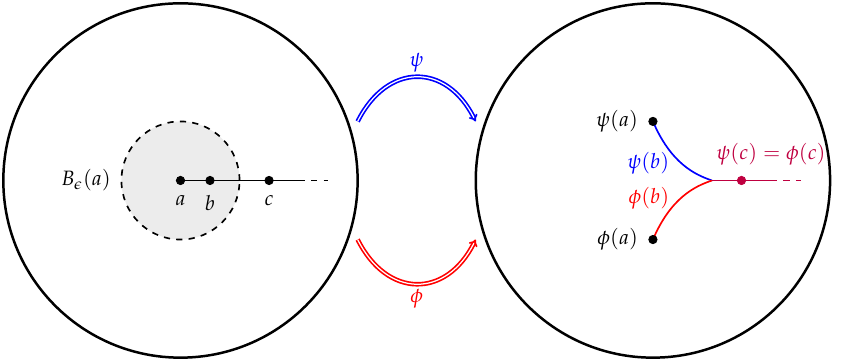}
    \caption{Illustration of {two isometric embeddings} $\phi$ and $\psi$ witnessing pinching at $a$. Note that $\phi$ and $\psi$ disagree everywhere in the ball $B_\epsilon(a)$, but agree everywhere else. }
    \label{fig:pinch}
\end{figure}
\begin{figure}
    \centering
    \includegraphics[width=0.4\textwidth]{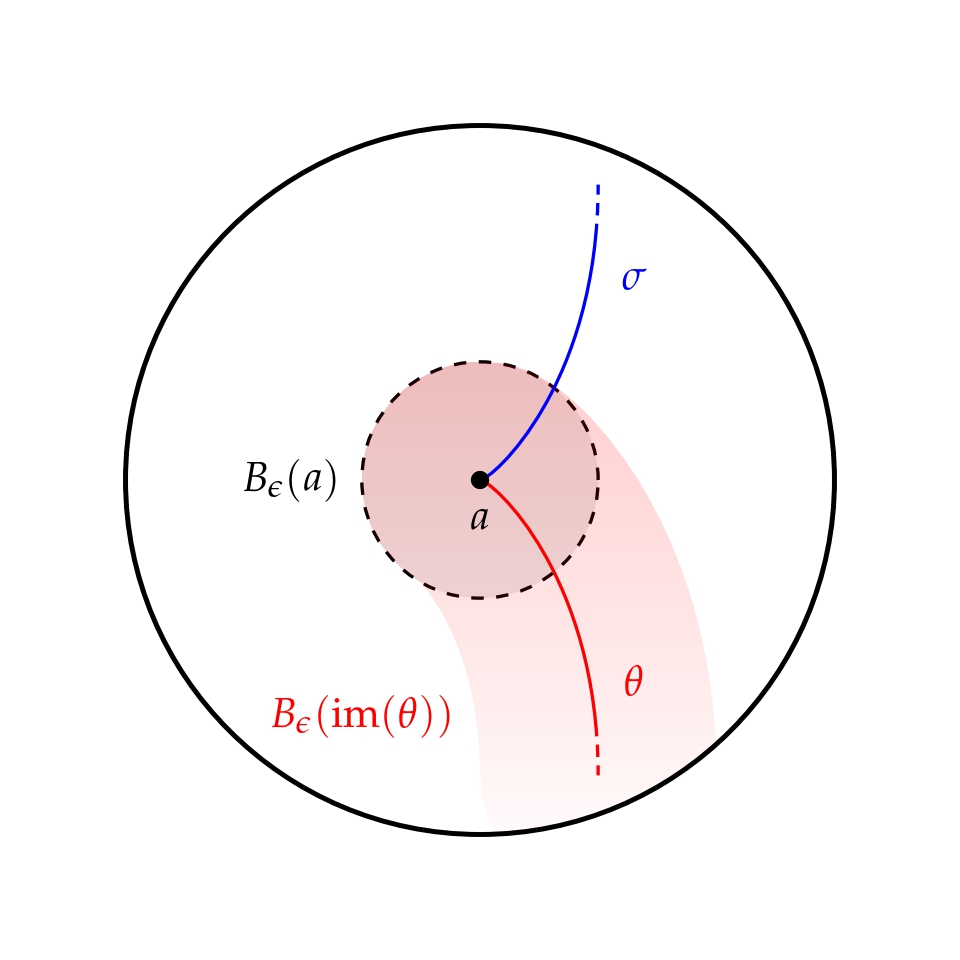}
    \caption{Illustration of {two isometric embeddings} $\sigma$ and $\theta$ witnessing spreading at $a$. Note that $a$ is in the images of both $\sigma$ and $\theta$ and that the ball around the image of $\theta$ of radius $\epsilon$ only contains points in the image of $\sigma$ inside of the ball $B_\epsilon(a)$.}
    \label{fig:spread}
\end{figure}

\begin{remark}  Note that any countable structure $\M$ can be considered as a standard metric space with the discrete metric. With this point of view in mind the metric pointwise topology $\mpw$ on $\EEmb(M)$ is just the topology of pointwise convergence $\pw$. The proof that $\zar=\pw$ for monoids of elementary embeddings of $\omega$-categorical structures with no algebraicity of~\cite{pinskerschindler} relies precisely on showing that $\EEmb(M)$ is pinching \cite[Lemma 3.6]{elliott2023polish} and spreading \cite[Lemma 3.6]{pinskerschindler} {in the discrete metric setting.}
Obviously, in this context open balls (of diameter $<1$) collapse to points, making proofs substantially simpler. 
\end{remark}

\begin{corollary}\label{Uaopen} Let $(\mathcal{X}, d)$ be a standard $\rr$-metric space. Let $S$ be a closed submonoid of $\ie(\mathcal{X})$ with respect to $\mpw$, and let $\mathcal{T}$ be a dense submonoid of $\rr$.  For $a\in X$ and $\epsilon\in \mathcal{T}\setminus\{0\}$, consider the set 
\[\mathcal{O}_{a,\epsilon}:=\left\{s\in S\vert B_{\epsilon}(a)\cap\im(s)\neq \emptyset\right\}.\]
If $S$ is pinching with respect to $\mathcal{T}$, then $\mathcal{O}_{a,\epsilon}\in\zar$.
\end{corollary}
\begin{proof} Since $S$ is pinching (with respect to $\mathcal{T}$), there are $\phi, \psi\in S$ such that (\ref{niceUri}) holds. Since $\phi$ and $\psi$ agree outside of $B_\epsilon(a)$,  if $s\in S\setminus \mathcal O_{a,\epsilon}$  then $\phi s=\psi s$. 
Consider the following set $\mathcal{V}\in\zar$:\[\mathcal{V}=\{s\in\ S \ \vert \phi s\neq \psi s\}.\]
By the above observation,  $\mathcal V\subseteq \mathcal O_{a,\epsilon}$. The other direction $\mathcal O_{a,\epsilon}\subseteq \mathcal V$ is clear since $\phi (c)\neq \psi (c)$ for any $c\in B_{\epsilon}(a)$.
\end{proof}
\begin{corollary}\label{Zopen} Let $(\mathcal{X}, d)$ be a standard $\rr$-metric space. Let $S$ be a closed submonoid of $\ie(\mathcal{X})$ with respect to $\mpw$, and let $\mathcal{T}$ be a dense submonoid of $\rr$. Let $a\in X$, $\sigma, \theta\in S$, and $\zeta,\eta\in \mathcal T\setminus\{0\}$. Consider the following set
\[Z_{(a, \zeta,\eta)}:=\left\{s\in S \ \vert B_{\zeta}(a)\cap\im(s\sigma)\neq \emptyset \text{ and }  B_{\eta}(a)\cap\im(s\theta)\neq \emptyset \right\}.\]
If $S$ is pinching with respect to $\mathcal{T}$, then  $Z_{(a, \zeta,\eta)}\in\zar$.

\end{corollary}
\begin{proof}
Consider the operations $\rho_\sigma: s\mapsto s\sigma$ and $\rho_{\theta}:s\mapsto s\theta$. Since $\zar$ is semitopological, we know that these are $\zar$-continuous. Moreover, 
\[Z_{(a, \zeta,\eta)}=\rho_\sigma^{-1}(\mathcal{O}_{a, \zeta})\cap \rho_{\theta}^{-1}(\mathcal{O}_{a, \eta}),\]

where $\mathcal{O}_{a, \zeta}, \mathcal{O}_{a, \eta}\in \zar$ by Corollary \ref{Uaopen}. Hence, $Z_{(a, \zeta,\eta)}\in\zar$.
\end{proof}
 
\begin{corollary}\label{containments} Let $(\mathcal{X}, d)$ be a standard $\rr$-metric space. Let $S$ be a closed submonoid of $\ie(\mathcal{X})$ with respect to $\mpw$, and let $\mathcal{T}$ be a dense submonoid of $\rr$.  Suppose that $S$ is pinching and spreading with respect to $\mathcal{T}$. {Let $a, b\in X$  and  $\epsilon, \zeta, \eta,\rho\in \mathcal T \setminus\{0\}$  where $\zeta<\eta$, $\zeta\oplus\eta<\epsilon$ and $\epsilon\oplus\eta<\rho$. Then, 
\[W_{(a, b, \zeta)}\subseteq Z_{(b, \zeta,\eta)}\subseteq W_{(a, b, \rho)}.\]}
\end{corollary}

\begin{proof} %
Since $S$ is spreading with respect to $\mathcal T$, there are $\sigma,\theta\in S$  such that $a\in\im(\sigma)\cap \im(\theta)$ and
$\im(\sigma)\cap B_\epsilon(\im(\theta))\subseteq B_{\epsilon}(a)$.
For the first inclusion, note that $a\in\im(\sigma)\cap\im(\theta)$. If $s(a)\in B_{\zeta}(b)$  then clearly $s(a)\in B_{\eta}(b)$ and $ s\in Z_{(b, \zeta,\eta)}$.
For the second inclusion, suppose $s\in Z_{(b, \zeta,\eta)}$. 
Then, there are $c\in\im(\sigma)$ and $ e\in\im(\theta)$ such that $d(s(c), b)<\zeta$ and $ d(s(e), b)<\eta$. %
{Hence,} $d(s(c), s(e))\leq d(s(c),b)\oplus d(b,s(e))\leq \zeta\oplus\eta< \epsilon$. 
Since $s$ is an isometric embedding, $d(c,e)<\epsilon$ and therefore $c\in B_\epsilon(\im(\theta))$. Hence, we have $c\in \im(\sigma)\cap B_\epsilon(\im(\theta))\subseteq B_\epsilon(a)$, that is $d(a,c)<\epsilon$. Finally, 
\[d(s(a), b)\leq d(s(a), s(c))\oplus d(s(c), b)\leq d(a, c)\oplus\zeta\leq \epsilon\oplus \zeta\leq \epsilon\oplus\eta<\rho\;.\]
From the above inequality, we can deduce the inclusion $Z_{(b,\zeta, \eta)}\subseteq W_{(a,b, \rho)}$.
    
\end{proof}

\begin{theorem} \label{Urirealproof} Let $(\mathcal{X}, d)$ be a standard $\rr$-metric space. Let $S$ be a closed submonoid of $\ie(\mathcal{X})$ with respect to $\mpw$, and let $\mathcal{T}$ be a dense submonoid of $\rr$.  Suppose that $S$ is pinching and spreading with respect to $\mathcal{T}$. Let $a, b\in X$ and $\epsilon\in \mathcal{T} \setminus\{0\}$. Let $s_0\in W_{(a,b,\epsilon)}$. Then, there is some $\mathcal{V}\in\zar$ such that 
\[s_0\in \mathcal{V}\subseteq W_{(a,b,\epsilon)}.\]
In particular, $\mpw=\zar$ on $S$, and so $\mpw$ is minimal amongst Hausdorff semigroup topologies on $S$. 
\end{theorem}

\begin{proof} Let $c:=s_0(a)$ and $\xi:=d(c,b)<\epsilon$. Since $\rr$ is standard and $\mathcal{T}$ is dense in $R$, there is $\eta\in \mathcal{T}\setminus\{0\}$ such that $\eta\oplus\xi<\epsilon$. 

Moreover, since 
$\mathcal{T}$ is dense in $\rr$, by iterated applications of standardness of $\rr$, we can find $\zeta, \eta_0, \epsilon_0, \rho\in \mathcal T\setminus\{0\}$, all strictly smaller than $\eta$ and such that 
\[\zeta<\eta_0, \quad \zeta\oplus \eta_0<\epsilon_0, \quad \text{ and } \quad \epsilon_0\oplus\eta_0<\rho<\eta\;.\]

Applying Corollary \ref{containments}, we get
\[ s_0\in W_{(a,c,\zeta)}\subseteq Z_{(c,\zeta,\eta_0)}\subseteq W_{(a,c,\rho)}\subseteq W_{(a, b, \epsilon)}\;.\]

{Now, by setting $\V:=Z_{(c, \zeta, \eta_0)}$, and applying Corollary \ref{Zopen}, we obtain the desired conclusion.}
Note that we just proved that for every $\mpw$-open set $\W$ and any $s_0\in W$, there is a $\zar$-open set $\V$ such that $s_0\in\V\subseteq \W$. In particular, this means that $\mpw\subseteq \zar$. But since $\zar$ is contained in every Hausdorff semigroup topology on $S$, we also have $\zar\subseteq\mpw$ by Lemma~\ref{lem:hausdorff}. Hence, $\mpw=\zar$.
\end{proof}

\subsection{Minimal semigroup topologies on the 
{monoids of isometric} embeddings of Urysohn spaces}

In this section we prove that the monoids of isometric embeddings of various Urysohn spaces are such that the topology $\mpw$ agrees with the Zariski topology $\zar$, and is strictly coarser than the topology of pointwise convergence $\pw$. In particular, we prove under general conditions on $\rr$ that $\ie(\U_{\rr})$ is spreading and pinching, which by \cref{Urirealproof} implies that $\mpw=\zar$. Proving the above two properties is the main challenge of this section.

\begin{lemma}\label{Urilemma1}  Let $\rr$ be a countably infinite metrically complete, standard distance monoid, and let $\U$ be $\U_\rr$. Then, $\ie(\U)$ is pinching.  
\end{lemma}
\begin{proof} 
 
{Let $a\in\U$, and  $\epsilon\in \rr\setminus\{0\}$.} 
 Choose $c\in \U$ so that $d(c,a)=\epsilon$. Fix an  enumeration $a_0, a_1, \dots$ of elements of $\U$ starting from $a_0=a$. We build by induction two sets $B=\{b_i:i\in \mathbb N, b_i\in \U\}$ and $C=\{c_i:i\in \mathbb N, c_i\in \U\}$ such that
\begin{align}
b_0=a_0 \mbox{ and } c_0=c \label{9}\\
    \delta_{i,j}:=d(a_i, a_j)=d(b_i, b_j)=d(c_i, c_j)\label{isoemb} \text{ for all } i,j\in \mathbb N\\
     b_i=c_i \text{ if }\ \alpha_i:=d(a_0, a_i)\geq \epsilon \label{eq:equalitybc}\\ d(b_i,c_0)=d(c_i,b_0)=\max\{\alpha_i, \epsilon\ominus\alpha_i\} \text{ for all } i\in \mathbb N\setminus\{0\} \label{withzero}\\
     d(b_i, c_j)=\mathrm{max}\left\{\delta_{i,j}, \epsilon\ominus\alpha_i, \epsilon\ominus\alpha_j\right\} \text{ for }\ \alpha_i,\alpha_j< \epsilon \text{ when } i,j\neq 0 
     .\label{max}
\end{align}
If this can be done for all $n\in \mathbb N$, then the maps $a_i\mapsto b_i$ and $a_i\mapsto c_i$ will yield the desired $\phi$ and $\psi$. In particular, condition (\ref{isoemb}) ensures that $\phi$ and $\psi$ are isometric embeddings and (\ref{eq:equalitybc}), the fact that $a\neq c$, and condition (\ref{max}) jointly ensure that property (\ref{niceUri}) is satisfied.

The base case $n=0$ is trivial. For the inductive step, we may assume there is a sequence that satisfies the conditions (\ref{9})--(\ref{max}) for $n>0$, and we prove that we can find the appropriate $b_{n+1}$ and $c_{n+1}$. Let $p(x_0, \dots, x_{n+1})$
{be the quantifier-free type of the tuple $(a_0, \dots, a_{n+1})$.}\\

\textbf{Case 1:} $\alpha_{n+1}\geq\epsilon$. In this case we need to show \begin{equation}\label{typesat1}
   p(b_0, \dots, b_n, x_{n+1})\cup p(c_0, \dots, c_n, x_{n+1}) 
\end{equation}
is satisfiable in $\U$. Then, $b_{n+1}=c_{n+1}$ will be any element of $\U$ satisfying the type above. By the universality of $\U$ and the extension property (EP), we need to show there is an $\rr$-metric space $D\cup \{e\}$ where $D$ is isometric to $\{b_0,\cdots,b_n\}\cup \{c_0,\dots,c_n\}$, and the distances of $e$ to $D$ are as suggested by the type $p$. Note that in order to show $D\cup \{e\}$ is an $\rr$-metric space, it is sufficient to check the triangle inequality holds for new triplets $(b_i,c_j,e)$, where $i,j\leq n$.
Based on our notation $d(b_i,e)=\delta_{i,n+1}$ and $d(c_j,e)=\delta_{j,n+1}$. Then, by the triangle inequality for $(a_i, a_j, a_{n+1})$,
\begin{equation}\label{a1}
    d(b_i, e)=\delta_{i, n+1}\leq \delta_{i,j}\oplus\delta_{j, n+1}=\delta_{i,j}\oplus d(c_j,e)
\end{equation}
and similarly 
\begin{equation}\label{a2} d(c_j, e)\leq \delta_{i,j}\oplus d(b_i,e).\end{equation}

When $b_i=c_i$ or $b_j=c_j$ then $d(b_i,c_j)=\delta_{i,j}$. From the inequalities above and the fact that $\delta_{i,j}\leq \delta_{i,n+1}\oplus\delta_{j,n+1}$, we immediately have all the triangle inequalities that we want. 
Now consider the case where $b_i\neq c_i$ and $b_j\neq c_j$. By (\ref{eq:equalitybc}), this implies that $\alpha_i,\alpha_j< \epsilon$. 
According to (\ref{max}), if $i,j\neq 0$, we have that $d(b_i,c_j)=\max\{\delta_{i,j}, \epsilon\ominus\alpha_i,\epsilon\ominus\alpha_j\}$. This, combined with (\ref{a1}) and (\ref{a2}) immediately implies that    $d(b_i, e)\leq d(b_i,c_j)\oplus d(c_j,e)$ and $d(c_j, e)\leq d(b_i,c_j)\oplus d(b_i,e)$. 
Again, if $d(b_i,c_j)=\delta_{i,j}$, then the above inequalities, together with $\delta_{ij}\leq\delta_{i, n+1}\oplus\delta_{j, n+1}$, establish the remaining triangle inequality. Assume $d(b_i,c_j)=\epsilon\ominus\alpha_i$. Then since $\alpha_{n+1}\geq\epsilon$ we have
\[\epsilon \leq \alpha_{n+1}\oplus\delta_{j, n+1}\leq \alpha_i\oplus\delta_{i, n+1}\oplus\delta_{j, n+1}, \]
and so by Fact~\ref{fact:ominus}.1,
\[ \epsilon\ominus\alpha_i \leq \delta_{i, n+1}\oplus\delta_{j, n+1}\;.\]
This settles all the triangle inequalities we need to check. The case of $d(b_i,c_j)=\epsilon\ominus\alpha_j$ is identical. When $i=j=0$ the required triangle inequalities are obvious. 
Now consider the case of $i\neq 0$ and $j=0$. According to assumption (\ref{withzero}), $d(b_i,c_0)=\max\{\alpha_i,\epsilon\ominus\alpha_i\}$. If $d(b_i,c_0) = \alpha_i$, since we know that $\alpha_i< \epsilon$,  we obtain $\alpha_i\leq \alpha_{n+1}\oplus\delta_{i,n+1}$ (note that $d(e,c_0)=\alpha_{n+1}$).
The %
remaining inequalities also hold. \\
\textbf{Case 2:} $\alpha_{n+1}<\epsilon$.  Let $\Gamma(x_{n+1}, y_{n+1})$ be the quantifier-free type asserting
\begin{align*}
    \left\{d(b_i, y_{n+1})=\mathrm{max}\{\delta_{i, n+1}, \epsilon\ominus\alpha_i, \epsilon\ominus\alpha_{n+1}\}= d(c_i, x_{n+1}) \ \vert \ 0< i\leq n, \alpha_i<\epsilon \right\}\cup\\
\{d(x_{n+1},c_0)=d(y_{n+1},b_0)=\max\{\alpha_{n+1}, \epsilon\ominus\alpha_{n+1}\} \}
\end{align*}
It is enough to show that the following quantifier-free type $\Theta(x_{n+1}, y_{n+1})$ is satisfiable in $\U$:
\begin{equation}\label{typesat2}
     p(b_0, \dots, b_n, x_{n+1})\cup p(c_0, \dots, c_n, y_{n+1}) \cup \Gamma(x_{n+1}, y_{n+1})\cup\{d(x_{n+1}, y_{n+1})=\epsilon\ominus\alpha_{n+1}\}.
\end{equation}
Note that $\Theta(x_{n+1}, y_{n+1})$ entirely specifies the distance relations between $x_{n+1}$, or $y_{n+1}$ and any of the elements in $\{b_i\vert 0\leq i\leq n\}\cup\{c_i\vert 0\leq i\leq n\}$, and it also specifies the distance between $x_{n+1}$ and $y_{n+1}$. So, again by universality, in order to show that $\Theta(x_{n+1}, y_{n+1})$ is satisfiable in $\U$, we only need to show that the finite structure on $\{b_i\vert 0\leq i\leq n+1\}\cup\{c_i\vert 0\leq i\leq n+1\}$ whose distance-relations are given by $\Theta(b_{n+1}, c_{n+1})$ is still an $\rr$-metric space.
And for this, we only need to check the triangle inequality for all triplets of distinct points of the form $(b_{n+1}, c_i, b_j)$ and $(b_{n+1}, c_i, c_j)$ for $i\leq n+1, j\leq n$. %
  For $(b_{n+1}, c_i, b_j)$  for $i\leq n+1, j\leq n$ the required triangle inequalities {clearly hold} if $c_i= b_i$ or $ b_j=c_j$. Thus we only need to deal with the cases of $c_i\neq  b_i$ and $ b_j\neq c_j$.\\ 

We begin by considering $(b_{n+1}, c_i, b_j)$ for $i\leq n+1, j\leq n$. Firstly according to the type $\Gamma$, $d(c_i, b_{n+1})$ may be any  value $\delta_{i, n+1}, \epsilon\ominus\alpha_i$ or $\epsilon\ominus\alpha_{n+1}$ when $i\neq 0$.  
Consider the triangle inequality starting from $d(b_j, b_{n+1})$. We get
\[d(b_j, b_{n+1})=\delta_{j, n+1}\leq \delta_{i, j}\oplus\delta_{i, n+1}\leq d(c_i, b_j)\oplus d(c_i, b_{n+1}).\]
When $i\neq n+1$ and  $d(c_i, b_{n+1})=\delta_{i, n+1}$ we get 
\[d(c_i, b_{n+1})=\delta_{i, n+1}\leq \delta_{i,j}\oplus\delta_{j, n+1}\leq d(c_i, b_j)\oplus d(b_j, b_{n+1}).\]
When $i=n+1$ or $d(c_i,b_{n+1})=\epsilon \ominus\alpha_i$, we get
\[d(c_i, b_{n+1})=\epsilon\ominus\alpha_i\leq d(c_i, b_j)\leq d(c_i, b_j)\oplus d(b_j, b_{n+1}).\]

Now we only need to deal with  $d(c_i,b_{n+1})=\epsilon \ominus\alpha_{n+1}$ when $i\neq n+1$. Note that $\alpha_j\leq \alpha_{n+1}\oplus \delta_{j,n+1}$. Then, using Fact~\ref{fact:ominus}.2 for the first inequality,
\[d(c_i,b_{n+1})=\epsilon\ominus\alpha_{n+1}\leq \epsilon \ominus\alpha_j\oplus\delta_{j,n+1}\leq d(c_i,b_j)\oplus d(b_j,b_{n+1}).\]

Finally, the case of starting the triangle inequality from $d(c_i, b_j)$ also splits into three cases. The only case not identical to previously treated ones is that of $d(c_i, b_j)=\epsilon\ominus\alpha_j$. To deal with this case, note that $\alpha_{n+1}\leq \alpha_j\oplus \delta_{j, n+1}$, yielding (again, using Fact~\ref{fact:ominus}.2)
\[\epsilon\ominus\alpha_j\leq (\epsilon\ominus\alpha_{n+1}){\oplus}\delta_{j, n+1}\leq d(c_i, b_{n+1})\oplus d(b_j, b_{n+1}).\]
When $i=0$ similar inequalities hold. Hence, all triplets $(b_{n+1}, c_i, b_j)$ for $i\leq n+1, j\leq n$ satisfy the required triangle inequalities.\\

We now move to the case of triplets of the form $(b_{n+1}, c_i, c_j)$ for $i\leq n+1, j\leq n$. Firstly, note that if $i\neq n+1$
\[d(c_i, c_j)=\delta_{i,j}\leq\delta_{i, n+1}\oplus \delta_{j, n+1}\leq d(c_i, b_{n+1})\oplus d(c_j, b_{n+1}).\]
The triangle inequalities starting with $d(c_i, b_{n+1})$ and $d(c_j, b_{n+1})$ can be dealt with analogously to previous cases.\\

The only remaining triangle inequalities to check are those for the triplet $(b_{n+1}, c_{n+1}, c_j)$ for $j\leq n$. Clearly,
\[d(c_j, c_{n+1})=\delta_{j,n+1}\leq\delta_{j, n+1}\oplus (\epsilon\ominus\alpha_{n+1})\leq d(c_j, b_{n+1})\oplus d(c_{n+1}, b_{n+1}).\] also
\[d(c_{n+1},b_{n+1})=\epsilon\ominus\alpha_{n+1}\ \leq d(c_{n+1},c_j)\oplus d(c_j,b_{n+1})\] holds simply because $d(c_j,b_{n+1})=\max\{\delta_{j,n+1},\epsilon\ominus\alpha_j,\epsilon\ominus\alpha_{n+1}\}$. Finally, we need to consider three separate cases, depending on whether  $d(c_j, b_{n+1})$ takes the value  $\delta_{j, n+1}, \epsilon\ominus\alpha_j$ or $\epsilon\ominus\alpha_{n+1}$. 
In the first case 
\[d(c_j, b_{n+1})=\delta_{j,n+1}\leq\delta_{j, n+1}\oplus (\epsilon\ominus\alpha_{n+1})= d(c_j, c_{n+1})\oplus d(c_{n+1}, b_{n+1}).\]
In the second case, from $\alpha_{n+1}\leq \alpha_j\oplus \delta_{j,n+1}$ and Fact~\ref{fact:ominus}.2, we get
\[d(c_j, b_{n+1})=(\epsilon\ominus\alpha_j)\leq (\epsilon\ominus\alpha_{n+1})\oplus \delta_{j,n+1}= d(c_{n+1}, b_{n+1})\oplus d(c_j, c_{n+1}).\]
For the third and final case,
\[d(c_j, b_{n+1})=\epsilon\ominus\alpha_{n+1}= d(c_{n+1},b_{n+1})\leq d(c_{n+1}, b_{n+1})\oplus d(c_j, c_{n+1}). \]
This concludes our proof of the lemma.
\end{proof}

\begin{corollary}\label{realcor1} Let $\rr$ be either {$\mathbb{R}_{\geq 0}$} or $[0,1]$ (with $\oplus$ being truncated addition in the latter case), and let $\mathcal{T}$ be {$\mathbb{Q}_{\geq0}$} or $\mathcal{Q}_1:=\mathbb{Q}\cap[0,1]$ respectively. Let $\U^\star:=\U_{\rr}$. Then, {$\ie(\U^\star)$} is pinching with respect to $\mathcal{T}$.
    
\end{corollary}
\begin{proof}
Let $\U$ be $\U_{\mathbb{Q}_{\geq 0}}$, or $\U_{\mathcal{Q}_1}$ depending on whether $\U^\star$ is $\U_{\mathbb{R}_{\geq_0}}$ or $\U_{[0,1]}$ respectively. Let $a\in\U^\star$ and fix $\epsilon\in \mathcal T\setminus\{0\}$. By transitivity of $\mathrm I(\U^\star)$ and density of  $\U$  in $\U^\star$, we may assume $a\in\U$.
Let $\phi, \psi\in\mathrm{IE}(\U)$ be as constructed in \cref{Urilemma1}, so that, for an enumeration $(a_i)_{i\in\mathbb{N}}$ of $\U$, and writing $b_i:=\phi(a_i)$ and $c_i:=\psi(a_i)$, conditions (\ref{9})-(\ref{max}) are satisfied. 
Given $e\in\U^\star$, pick a  Cauchy sequence $(d_i)_{i\in\mathbb{N}}$ in $\U$ converging to $e$,  and set $\phi^\star:\U^\star\to\U^\star$  to be defined as 
\[\phi^\star(e):=\lim_{i\to\infty} \phi(d_i)\;.\]
{Note that since $\phi$ is an isometric embedding $(\phi(d_i))_{i\in\mathbb{N}}$ is Cauchy and the limit does not depend on the choice of $(d_i)_{i\in\mathbb{N}}$.}
Hence $\phi^\star$ is indeed well-defined and continuous. {By continuity of the distance function, it is an isometric embedding extending
$\phi$.} %
We can similarly define $\psi^\star$. {Since $\phi^\star$ and $\psi^\star$ are continuous, for any Cauchy sequence $(d_i)_{i\in\N}$ from $\U$ converging to $e$,}
\[d(\phi^\star(e), \psi^\star(e))=\lim_{i\to\infty} d(\phi(d_i), \psi(d_i))\;.\]
{By \eqref{withzero} and \eqref{max}, the right hand side of the above equals $\epsilon \ominus d(a,e)$ whenever $d(a, e)<\epsilon$ and $0$ otherwise.}
{Hence,}
\[\phi^\star\upharpoonright_{\U^\star\setminus B_{\epsilon}(a)}=\psi^\star\upharpoonright_{\U^\star\setminus B_{\epsilon}(a)}\;,\]
{and $\phi^\star(e)\neq\psi^\star(e)$ for every $e\in B_\epsilon (a)$, as desired. This concludes the proof.}
\end{proof}

\begin{lemma}\label{Urilemma2} Let $\rr$ be a countably infinite standard distance monoid and let $\U$ be $\U_{\rr}$.  Then, $\ie(\U)$ is spreading.
\end{lemma}
\begin{proof} {Let $a\in\U$, and  $\epsilon\in \rr\setminus\{0\}$.} Fix an enumeration {$(a_i)_{i\in \mathbb N}$} of $\U$ where $a_0=a$. We build inductively {$(b_i)_{i\in \mathbb N}$ and $(c_i)_{i\in \mathbb N}$} such that $b_i,c_i\in \U$, writing $\delta_{i,j}:=d(a_i, a_j)$ and $\alpha_i:=d(a_0, a_i)$,
\begin{align}
    b_0=a_0=c_0, \label{boh1}\\
   d(b_i,b_j)= \delta_{i,j} =d(c_i,c_j),\label{boh2}\\
    d(b_i, c_j)=\alpha_i\oplus \alpha_j \;. \label{boh3}
\end{align}
Note that if we can find such sequences, then the maps $\sigma:a_i\mapsto b_i$ and $\theta:a_i\mapsto c_i$ are isometric embeddings of $\U$ satisfying the requirements of the Lemma. In particular, for 
    \[d\in\im(\sigma)\cap B_\epsilon(\im(\theta)),\]
there is some $i\in \mathbb N$ such that $d=b_i$ and some $j\in \mathbb N$ such that $d(b_i, c_j)<\epsilon$ where $c_j\in \im(\theta)$. Then, 
\[\alpha_i=d(b_i, a)\leq d(b_i, a)\oplus d(c_j, a)=\alpha_i\oplus \alpha_j=d(b_i, c_j)<\epsilon,\]
implying $d=b_i\in B_{\epsilon}(a)$, as desired.\\

We {now proceed to} prove by induction that we can find sequences {$(b_i)_{i\in \mathbb N}$ and $(c_i)_{i\in \mathbb N}$} satisfying conditions (\ref{boh1}), (\ref{boh2}) and (\ref{boh3}). The base case $i=0$ is trivial. For the induction step, by {the extension property,}
we just need to check that for $b_{n+1}$ and $c_{n+1}$ satisfying the conditions of (\ref{boh2}) and (\ref{boh3}), we do not get any triplet breaking the triangle inequality. Without loss of generality, we may only check triplets of the form $(b_{n+1}, c_i, b_j)$ and $(b_{n+1}, c_i, c_j)$ for $i\leq n+1, j\leq n$.\\

For triplets of the form $(b_{n+1}, c_i, b_j)$ for $i\leq n+1, j\leq n$, consider first the triangle inequality starting from $d(b_j, b_{n+1})$. We get
\[d(b_j, b_{n+1})=\delta_{j, n+1}\leq \alpha_j\oplus \alpha_{n+1}\leq (\alpha_i\oplus \alpha_j)\oplus (\alpha_i\oplus \alpha_{n+1})=d(c_i, b_j)\oplus d(c_i, b_{n+1}).\]
The case of $d(c_i, b_{n+1})$ is identical to the above.
For $d(c_i, b_j)$, we get
\[d(c_i, b_j)=\alpha_i\oplus \alpha_j\leq \alpha_i\oplus \delta_{j, n+1}\oplus \alpha_{n+1}=\alpha_i\oplus \alpha_{n+1} \oplus \delta_{j, n+1}= d(c_i, b_{n+1})\oplus d(b_j, b_{n+1}).\]

For triplets $(b_{n+1}, c_i, c_j)$ with $i\leq n+1, j\leq n$, consider first $d(c_i, b_{n+1})$. We get
\[d(c_i, b_{n+1})=\alpha_i\oplus \alpha_{n+1}\leq \delta_{i,j}\oplus\alpha_j\oplus\alpha_{n+1}=d(c_i, c_j)\oplus d(c_j, b_{n+1}).\]
The case of $d(c_j, b_{n+1})$ is identical.
For $d(c_i, c_j)$ we have
\[d(c_i, c_j)=\delta_{i,j}\leq \alpha_i\oplus\alpha_j\leq (\alpha_i\oplus\alpha_{n+1})\oplus (\alpha_j\oplus \alpha_{n+1})=d(c_i, b_{n+1})\oplus d(c_j, b_{n+1}).\]
Thus, all triplets of vertices still satisfy the triangle inequality and so we can find $b_{n+1}$ and $c_{n+1}$ satisfying (\ref{boh2}) and (\ref{boh3}). This concludes the induction proof and so the proof of the Lemma.
\end{proof}

\begin{corollary} \label{realcor2} Let $\rr$ be either $\mathbb{R}_{\geq 0}$ or $[0,1]$ (with $\oplus$ being truncated addition in the latter case), and let $\mathcal{T}$ be $\mathbb{Q}_{\geq 0}$ or $\mathcal{Q}_1:=\mathbb{Q}\cap[0,1]$ respectively. Let $\U^\star:=\U_{\rr}$. Then, $\ie(\U^\star)$ is spreading with respect to $\mathcal{T}$.
\end{corollary}
\begin{proof}
    The proof follows the same structure as that of Corollary~\ref{realcor1}: let $\U$ stand for $\U_{\mathbb{Q}_{\geq 0}}$ or $\U_{\mathcal{Q}_1}$ depending on whether $\U^\star$ stands for $\U_\mathbb{R}$ or $\U_{[0,1]}$ respectively. 
    {Let $a\in\U^\star$ and fix $\epsilon\in \mathcal T\setminus\{0\}$. By transitivity of $\mathrm{I}(\U^\star)$ and the fact that $\U^\star$ is the completion of $\U$, we may assume $a\in \U$. Let $\sigma, \theta\in \mathrm{IE}(\U)$ be as constructed in Lemma~\ref{Urilemma2} in order to witness (\ref{eq:ballcontainments}). Then, we can extend them to $\sigma^\star$ and $\theta^\star$ in $\mathrm{IE}(\U^\star)$ just as in Corollary~\ref{realcor1}. Since $a\in\im(\sigma)\cap \im(\theta)$, we have $a\in\im(\sigma^\star)\cap \im(\theta^\star)$. So it only remains to prove the containment
    \[\im(\sigma^\star)\cap B_\epsilon(\im(\theta^\star))\subseteq B_\epsilon(a).\]
    Consider $e\in\U^\star$ such that  $d(\sigma^\star(e), \theta^\star(f))<\epsilon'<\epsilon$ for some $f\in \U^\star$ and $\epsilon'\in \mathcal T$. }
    Then, pick $\delta, \gamma\in \mathcal{T}\setminus\{0\}$ such that $\epsilon'\oplus\delta\oplus\gamma\oplus\gamma<\epsilon$. 
    By density of $\U$ in $\U^\star$, pick $b,c\in \U$ such that {$d(b,e)<\gamma$ and $d(c,f)< \delta$.}
    Thus, since $\sigma^\star$ and $\theta^\star$ are isometric embeddings, {$d(\sigma(b), \sigma^\star(e))=d(b,e)$ and $d(\theta(c), \theta^\star(f))=d(c,f)$. }
    By the triangle inequality, {$d(\theta(c), \sigma^\star(e))\leq d (\sigma^\star(e),\theta^\star(f))\oplus d(\theta^\star(f),\theta(c))<\epsilon'\oplus\delta$ } and 
\[d(\sigma(b), \theta(c))< \epsilon'\oplus\delta\oplus\gamma\;.\]
 However, by the construction of $\sigma$ and $\theta$ in Lemma~\ref{Urilemma2}, if $\alpha=d(a,b)$, $\beta=d(a,c)$, we have that $d(\sigma(b), \theta(c))=\alpha\oplus\beta$. Comparing the above equations, we obtain that
 \[\alpha\oplus\beta<\epsilon'\oplus\delta\oplus\gamma\;.\]
Recall also that $\sigma(a)=a$, and so, since $\sigma$ is an {isometric embedding,}
$d(a, \sigma(b))=\alpha$. From this and the above inequalities, we obtain that
 
 \[d(a, \sigma^\star(e))\leq \alpha \oplus d(\sigma(b), \sigma^\star(e))=\alpha\oplus\gamma\leq \alpha\oplus\beta\oplus\gamma <\epsilon'\oplus\delta\oplus\gamma\oplus\gamma<\epsilon\;,\]
 where the last inequality follows from the choice of $\delta$ and $\gamma$. But then, as desired, $\sigma^\star(e)\in B_\epsilon(a)$.
    
\end{proof}

\begin{lemma}\label{lem: buildmetric} 
Let $\rr$ be a standard distance monoid, let $F$ be a finite $\rr$-metric space, and fix $a\in F$ and $\epsilon\in R\setminus \{0\}$. {Then there are a finite $\rr$-metric space $F\cup\{b\}$ with $b\notin F$ and $\delta\in R\setminus \{0\}$ with $\delta<\epsilon$ such that $d(a,b)=\delta$ and the map fixing $F\setminus\{a\}$ pointwise and {transposing} $a$ 
{and} $b$ is an isometry.}
\end{lemma}
\begin{proof} Since $\rr$ is standard, we can choose $\delta\in R\setminus\{0\}$ such that $\delta<\min\{\epsilon, \{d(f, a)\ \vert f\in F\setminus\{a\}\}\}$. We just need to prove that the finite structure with domain $F\cup \{b\}$ {with the distances of $F$ together with}
\[\{d(b, f):=d(a, f) \ \vert \ f\in F\setminus \{a\}\}\cup \{d(a, b)=\delta\}\]
is an $\rr$-metric space. {Since $F$ is an $\rr$-metric space and $\delta\neq 0$, we only need to verify the triangle inequalities for the triplets involving 
$b$.}
 {For a triplet $\{b,f,f'\}$ with $f,f'\in F\setminus \{a\}$, the required inequalities hold because $b$ has the same distances to $f$ and $f'$ as $a$ does. {Now consider a triplet} $\{a,b,f\}$ with $f\in F\setminus\{a\}$: first note that $d(a,b)=\delta< d(a,f)\leq d(a, b)\oplus d(f,b) $. Second $d(a,f)=d(b,f)\leq d(a, b)\oplus d(f,b) $ again by the assigned values. This shows $F\cup \{b\}$ is a finite $\rr$-metric space, and clearly by construction the map fixing $F\setminus\{a\}$ and %
 {swapping} $a$ {and} $b$ is an isometry.}
\end{proof}

\begin{theorem}\label{thm:urysohnmain} Let $\rr$ be 
\begin{itemize}
    \item EITHER a countably infinite metrically complete standard distance monoid; 
    \item OR {$(\mathbb{R}_{\geq 0}, \leq, +, 0)$};
    \item OR $([0,1], \leq, +_1, 0)$, where $+_1$ is truncated addition.
\end{itemize}
Let $\U:=\U_{\rr}$. 
{Then $\zar=\mpw\subsetneq\pw$ on $\ie(\U)$.} 
That is, $\mpw$ is minimal amongst Hausdorff semigroup topologies on $\ie(\U)$, and it is strictly coarser than $\pw$.
\end{theorem}
\begin{proof} 
The fact that $\zar=\mpw$, and hence minimality of $\mpw$, follows from the lemmas above and \cref{Urirealproof}. 
It remains to show that $\mpw$ is strictly coarser than $\pw$. {
For this, it suffices to find a 
$\pw$-open set which is not $\mpw$-open.}

Let $W$ be a non-empty $\mpw$-basic open set of the form
\[W:=\bigcap_{i=1}^n W_{(a_i, b_i, \epsilon_i)}\:.\]

{We show $W\nsubseteq\nn_\phi$  for every bijection $\phi:C\to D$ between non-empty finite subsets $C$ and $D$ of $\U$. Let $h\in W$. If $h\not\in \nn_\phi$, we are done; so assume $h\in \nn_\phi$ which means $h_{\upharpoonright C}=\phi$.
Put $A:=\{a_i\vert 1\leq i\leq n\}$. As $\rr$ is standard for each $i$ we can choose $t_i\in R\setminus \{0\}$ such that $d(h(a_i),b_i)\oplus t_i<\epsilon_i$ and set $\epsilon:=\min\{t_i\vert 1\leq i\leq n\}$.
Fix $a\in C$, and let $a':=h(a)=\phi(a)$. Let $F:=h(A\cup C)=h(A)\cup D$. By \cref{lem: buildmetric}, there is a finite $\rr$-metric space $F\cup\{c\}$ with $d(a', c)<\epsilon$ such that the map $\xi:F\cup\{c\}\to F\cup \{c\}$ fixing $F\setminus\{a'\}$ and sending $a'$ to $c$ is an isometry. By the extension property the embedding of $F$ into $\U$ extends to an embedding of $F\cup\{c\}$ into $\U$; and by homogeneity $\xi$ extends to an isometry $g\in\mathrm{I}(\U)$.
Consider the isometric embedding $\ell:=g\circ h$. First note that $\ell(a)=g h(a)=g(a')=c$, and $d(c, a')>0$, so $\ell(a)\neq \phi(a)$ and hence $\ell\not\in\nn_\phi$. Second, we check that $\ell\in W$. Fix $1\leq i\leq n$. If $h(a_i)\neq a'$, then $g$ fixes $h(a_i)$ so $\ell(a_i)=h(a_i)$ and therefore $d(\ell(a_i),b_i)=d(h(a_i),b_i)<\epsilon_i$. If $h(a_i)=a'$, then $\ell(a_i)=c$ and
\[d(\ell(a_i),b_i)=d(c,b_i)\leq d(c,a')\oplus d(a',b_i)<\epsilon\oplus d(h(a_i),b_i)\leq t_i\oplus d(h(a_i),b_i)<\epsilon_i,\]
where the last inequality follows from the choice of $t_i$.
So we can conclude that $\ell\in W$. In particular, $\ell \in W\setminus \nn_\phi$ as desired. This concludes the proof that $\mpw\subsetneq\pw$.}
\end{proof}

\begin{remark}
    In \cite{gonzalesghadernezhad}, all group topologies coarser than $\pw$ on the isometry groups of the Urysohn space and the Urysohn sphere have been classified. It seems plausible that, by modifying the arguments in Section 5 of \cite{gonzalesghadernezhad} to adapt them to the semigroup and monoid setting, one may show that there is no intermediate semigroup topology strictly between $\pw$ and $\mpw$.
\end{remark}

\textbf{Acknowledgements:} We would like to thank Tingxiang Zou for a lovely talk about lovely pairs. This work was supported by UKRI EP/X024431/1. Funded by the European Union (ERC, POCOCOP, 101071674). Views and opinions expressed are however those of the authors only and do not necessarily reflect those of the European Union or the European Research Council Executive Agency. Neither the European Union nor the granting authority can be held responsible for them. This research was funded in whole or in part by the Austrian Science Fund (FWF) [I 5948]. For the purpose of Open Access, the authors have applied a CC BY public copyright licence to any Author Accepted Manuscript (AAM) version arising from this submission. J. de la Nuez Gonz\'alez is supported by the Mid-Career Researcher Program (RS-2023-00278510) through the National Research Foundation funded by the government of Korea and by the KIAS individual grant SP084001. All content of this paper has been produced by the authors. Before submission, ChatGPT 5.6 Sol was used to review the paper for typos and grammatical/stylistic inconsistencies.
 \printbibliography

\end{document}